\documentclass{amsart}

\title{Uniqueness and functoriality of Igusa stacks}
\author{Dongryul Kim}
\date{\today}
\email{dkim04@stanford.edu}
\address{Department of Mathematics, Stanford University, 450 Jane Stanford Way
(Building 380), Stanford, California, USA}

\usepackage{amsmath}
\usepackage{amssymb}
\usepackage{amsthm}
\usepackage{mathrsfs}
\usepackage{hyperref}
\usepackage{tikz-cd}
\usepackage{cleveref}

\newtheorem{theorem}[subsection]{Theorem}
\newtheorem{lemma}[subsection]{Lemma}
\newtheorem{proposition}[subsection]{Proposition}
\newtheorem{corollary}[subsection]{Corollary}
\newtheorem{maintheorem}{Theorem}

\theoremstyle{definition}
\newtheorem{definition}[subsection]{Definition}
\newtheorem{example}[subsection]{Example}
\newtheorem{remark}[subsection]{Remark}

\renewcommand\paragraph{\subsection{}}

\DeclareMathOperator{\Gal}{Gal}
\DeclareMathOperator{\Spec}{Spec}

\DeclareMathOperator{\Spa}{Spa}
\DeclareMathOperator{\Spd}{Spd}
\DeclareMathOperator{\Hom}{Hom}
\DeclareMathOperator{\Aut}{Aut}
\DeclareMathOperator{\im}{im}
\newcommand\id{\mathrm{id}}
\let\textdef\textit

\NewDocumentCommand\gx{t'}{%
  (\mathsf{G}\IfBooleanT{#1}{^\prime},%
  \mathsf{X}\IfBooleanT{#1}{^\prime})%
}
\NewDocumentCommand\bung{t\c t\m t'}{%
  \mathrm{Bun}_{G%
  \IfBooleanTF{#1}{\IfBooleanTF{#3}{^{\prime\mathrm{c}}}{^\mathrm{c}}}{\IfBooleanT{#3}{^\prime}}%
  \IfBooleanT{#2}{,\{%
    \IfBooleanTF{#1}{(\mu^{\IfBooleanT{#3}{\prime}\mathrm{c}})^{-1}}{\mu^{\IfBooleanT{#3}{\prime}-1}}%
    \}}}%
}
\NewDocumentCommand\grg{t\c t\m t' o}{%
  \mathrm{Gr}_{G%
  \IfBooleanTF{#1}{\IfBooleanTF{#3}{^{\prime\mathrm{c}}}{^\mathrm{c}}}{\IfBooleanT{#3}{^\prime}}%
  \IfBooleanT{#2}{,\{%
    \IfBooleanTF{#1}{(\mu^{\IfBooleanT{#3}{\prime}\mathrm{c}})^{-1}}{\mu^{\IfBooleanT{#3}{\prime}-1}}\}}%
  ,\IfValueTF{#4}{#4}{\IfBooleanTF{#2}{E\IfBooleanT{#3}{^\prime}}{\qp}}}%
}
\NewDocumentCommand\igsu{t'}{%
  \mathrm{Igs}^{U\IfBooleanT{#1}{^\prime}}\IfBooleanTF{#1}{\gx'}{\gx}%
}
\NewDocumentCommand\sh{t\d t\ad t' o o}{%
  \mathbf{Sh}\IfValueT{#4}{_{#4}}\IfBooleanTF{#3}{\gx'}{\gx}%
  _{\IfValueTF{#5}{#5}{E\IfBooleanT{#3}{^\prime}}}%
  \IfBooleanT{#1}{^\lozenge}\IfBooleanT{#2}{^\mathrm{ad}}%
}
\NewDocumentCommand\shloc{t\c t' m o}{%
  \mathcal{M}_{G%
  \IfBooleanTF{#1}{\IfBooleanTF{#2}{^{\prime\mathrm{c}}}{^\mathrm{c}}}{\IfBooleanT{#2}{^\prime}}%
  ,#3,\{\mu%
  \IfBooleanTF{#1}{\IfBooleanTF{#2}{^{\prime\mathrm{c}}}{^\mathrm{c}}}{\IfBooleanT{#2}{^\prime}}%
  \}%
  \IfValueT{#4}{,#4}}%
}
\NewDocumentCommand\period{t\o m m t\n t\p}{%
  \IfBooleanT{#1}{\mathscr{O}}\mathbf{#2}_\mathrm{#3}%
  \IfBooleanTF{#4}{\IfBooleanTF{#5}{^{\nabla+}}{^\nabla}}{\IfBooleanT{#5}{^+}}%
}

\newcommand{\qp}{\mathbb{Q}_p}
\newcommand{\qpbr}{\breve{\mathbb{Q}}_p}
\newcommand{\qpbar}{\bar{\mathbb{Q}}_p}
\newcommand{\zp}{\mathbb{Z}_p}

\newcommand{\fp}{\mathbb{F}_p}
\newcommand{\fpbar}{\bar{\mathbb{F}}_p}
\newcommand{\cp}{\mathbb{C}_p}
\newcommand{\af}{\mathbb{A}_\mathbb{Q}^\infty}
\newcommand{\afp}{\mathbb{A}_\mathbb{Q}^{p,\infty}}

\begin{document}

\begin{abstract}
  We provide an axiomatic definition of an Igusa stack associated to an
  arbitrary Shimura datum. We then prove that Igusa stacks are unique and
  automatically functorial with respect to morphisms of Shimura data, assuming
  their existence. Using the same techniques, we also prove that the existence
  of the Igusa stack passes to a Shimura subdatum.
\end{abstract}

\maketitle
\tableofcontents

{\section{Introduction}
\def\igs{\mathrm{Igs}{\gx}}
\def\igsp{\mathrm{Igs}{\gx'}}
\def\igscirc{\mathrm{Igs}^\circ{\gx}}

Let $\gx$ be a Shimura datum with reflex field $\mathsf{E}$, let $p$ be a
rational prime, fix a place $v \mid p$ of $\mathsf{E}$, and write $E =
\mathsf{E}_v$. Attached to the infinite level $p$-adic Shimura variety $\sh$ is
the Hodge--Tate period map
\[
  \pi_\mathrm{HT} \colon \sh\d \to \grg\m,
\]
constructed by \cite{Sch15}, \cite{CS17} for PEL-type Shimura varieties and
\cite{Han16p}, \cite{PR24} in general. This allows one to study the \'{e}tale
cohomology of Shimura varieties via the sheaves $\mathscr{F} =
R\pi_{\mathrm{HT},\ast} \Lambda$ on $\grg\m$. Fargues \cite{Far16p} and Scholze
conjectured that the sheaf $\mathscr{F}$ arises from a sheaf on $\bung$, and
that this phenomenon can be explained by the existence of a v-stack $\igs$ that
fits in a Cartesian diagram
\[ \begin{tikzcd}
  \sh\d \arrow{r}{\pi_\mathrm{HT}} \arrow{d} & \grg\m \arrow{d}{\mathrm{BL}} \\
  \igs \arrow{r} & \bung\m.
\end{tikzcd} \]
The v-stack $\igs \to \bung\m$ may then be interpreted as interpolating
different Igusa varieties for different Newton strata into a single family.

In a foundational work, Zhang \cite{Zha23p} constructed Igusa stacks for
PEL-type Shimura data on the good reduction locus (as well as on the minimal
compactification in many cases). This result was later extended in
\cite{DvHKZ24p} to all Hodge type Shimura data, using fine analysis of integral
models of Shimura varieties due to Pappas--Rapoport \cite{PR24}. However, in
both \cite{Zha23p} and \cite{DvHKZ24p}, the Igusa stack was constructed directly
without a defining characterization. We provide an axiomatic definition of an
Igusa stack, which is made more precise in \Cref{Sec:Uniformization}.

\begin{definition}[\Cref{Def:Igusa}]
  An \textdef{Igusa stack} is a v-stack $\igs$ together with a
  $\underline{\mathsf{G}(\afp)}$-action and a Cartesian diagram
  \[ \begin{tikzcd}
    \sh\d \arrow{r}{\pi_\mathrm{HT}} \arrow{d} & \grg\m \arrow{d}{\mathrm{BL}} \\
    \igs \arrow{r}{\bar{\pi}_\mathrm{HT}} & \bung\m
  \end{tikzcd} \]
  satisfying the property that
  \begin{itemize}
    \item the $\mathsf{G}(\afp) \times G(\qp)$-action on $\igs \times_{\bung\m}
      \grg\m$ recovers the Hecke $\mathsf{G}(\af)$-action on $\sh\d$,
    \item for $\phi$ the absolute Frobenius, $\phi \times \id$ acts trivially on
      $\igs \times_{\bung\m} \grg\m$.
  \end{itemize}
\end{definition}

This definition makes sense for arbitrary Shimura data $\gx$, as it does not
make reference to moduli interpretations or integral models of Shimura
varieties. We also remark that the notion of an Igusa stack $\igs$ depends not
only on $\gx$ but also the place $v \mid p$ of $\mathsf{E}$, even though the
choice of $v$ is suppressed in the notation.

\begin{theorem}[\cite{DvHKZ24p}]
  When $\gx$ is of Hodge type, the Igusa stack on the good reduction locus
  $\igscirc$ exists.
\end{theorem}

\begin{proof}
  The axioms are verified in \cite[Section~5.2.9, Proposition~5.2.15,
  Theorem~6.0.1]{DvHKZ24p}.
\end{proof}

It was observed in \cite{HL24p} and \cite{DvHKZ24p} that in many cases, given an
Igusa stack $\igs$ satisfying the above axioms, the cohomology of Shimura
varieties as a $\mathsf{G}(\af) \times W_E$-representation can be recovered from
the sheaf $\bar{\mathscr{F}} = R\bar{\pi}_{\mathrm{HT},\fpbar,\ast} \Lambda$. In
particular, we have
\[
  R\Gamma(\sh[K^p][\qpbar], \Lambda) =
  T_\mu^{[1]}(\bar{\mathscr{F}}[-d])(-\tfrac{d}{2}),
\]
where $d = \langle 2\rho, \mu \rangle$ and $T_\mu^{[1]}$ is a Hecke operator
constructed in \cite[Section~IX.2]{FS21p}, see \cite[Theorem~8.4.10]{DvHKZ24p}.
Using compatibility of the Fargues--Scholze spectral action with parabolic
induction and methods developed in \cite{Kos21p}, it was shown in
\cite[Theorem~9.4.11]{DvHKZ24p} and \cite[Theorem~5.4]{vdH24p} that this is
enough to deduce Eichler--Shimura relations for Shimura varieties at Iwahori
level in many cases.

\subsection{Statement of main theorems}
We now state the main results of the paper. We refer the reader to
\Cref{Sec:Uniformization} for the precise definition of morphisms between
Igusa stacks and isomorphisms between such morphisms.

\begin{maintheorem}
  Given a Shimura datum $\gx$ with a place $v \mid p$ of $\mathsf{E}$, the Igusa
  stack $\igs$ is unique up to unique isomorphism, if it exists.
\end{maintheorem}

This immediately follows from a more general functoriality result.

\begin{maintheorem}[\Cref{Thm:Functoriality}]
  Let $\gx \to \gx'$ be a morphism of Shimura data, and choose places $v \mid
  v^\prime \mid p$ of $\mathsf{E}$ and $\mathsf{E}^\prime$. Assume that there
  exist corresponding Igusa stacks $\igs$ and $\igsp$. Then there exists,
  uniquely up to unique isomorphism, a
  $\underline{\mathsf{G}(\afp)}$-equivariant morphism of Igusa stacks $\igs \to
  \igsp$ fitting in the diagram
  \[ \begin{tikzcd}[row sep=0em]
    \sh\d \arrow{rr} \arrow{rd} \arrow{dd} & & \sh\d' \arrow{rd} \arrow{dd} \\ &
    \grg\m \arrow{rr} \arrow{dd} & & \grg\m' \arrow{dd} \\ \igs \arrow{rd}
    \arrow{rr} & & \igsp \arrow{rd} \\ & \bung\m \arrow{rr} & & \bung\m'.
  \end{tikzcd} \]
\end{maintheorem}

\begin{remark}
  The actual statement of \Cref{Thm:Functoriality} is flexible enough to
  handle uniqueness and functoriality of Igusa stacks on good reduction loci as
  well.
\end{remark}

\begin{remark}
  In upcoming work \cite{DvHKZ25p}, this functoriality result is used crucially
  in constructing Igusa stacks for abelian type Shimura data.
\end{remark}

One curious aspect of the Igusa stack is that even though its definition only
involves the Shimura variety over $E$, it contains information of the special
fiber as well. For example, when $\gx$ is of Hodge type,
\cite[Remark~6.5.10]{DvHKZ24p} states that for an arbitrary choice of parahoric
$K_p \subseteq G(\qp)$, the $\fpbar$-points of the good reduction Igusa stack
$\igscirc$ is the set of $p$-power quasi-isogeny classes of $\fpbar$-points of
the integral model $\mathscr{S}_{K_p}\gx$. This seems to suggest that, assuming
the existence of Igusa stacks, the Langlands--Rapoport conjecture could be more
naturally stated in terms of $\fpbar$-points of Igusa stacks.

Applying methods developed in the paper, we also prove an existence result for
Shimura subdata.

\begin{maintheorem}[\Cref{Thm:IgusaSubdata}]
  Let $\gx \hookrightarrow \gx'$ be an embedding of Shimura data, and let $E =
  \mathsf{E}_v$ and $E^\prime = \mathsf{E}^\prime_{v^\prime}$ be as before. If
  there exists an Igusa stack $\igsp$ for $\gx'$, then there also exists an
  Igusa stack $\igs$ for $\gx$.
\end{maintheorem}

This theorem, together with \cite{Zha23p}, gives an alternative proof of the
existence of good reduction Igusa stacks for Hodge type Shimura data, as well as
an extension to the bad reduction locus.

\begin{maintheorem}[\Cref{Cor:HodgeTypeBadReduction}]
  When $\gx$ is of Hodge type, there exist Igusa stacks
  \[
    \igscirc \subseteq \igs
  \]
  for both the good reduction locus and the entire Shimura variety, where the
  inclusion is an open embedding.
\end{maintheorem}

However, we note that the interplay between Igusa stacks and canonical integral
models of Shimura varieties discussed in \cite{DvHKZ24p} cannot be obtained
directly using our methods.

\begin{remark}
  Once there exists an Igusa stack $\mathrm{Igs}\gx$ over the entire Shimura
  variety, we may consider the sheaf
  \[
    \bar{\mathscr{F}}_c = R\bar{\pi}_{\mathrm{HT},\fpbar,!} \Lambda
  \]
  on $\mathrm{Bun}_{G,\{\mu^{-1}\},\fpbar}$. The discussion of
  \cite[Section~8.4]{DvHKZ24p} shows that the compactly supported cohomology of
  Shimura varieties can be identified with
  \[
    R\Gamma_c(\sh[K^p][\qpbar], \Lambda) \cong T_\mu^{[1]} (\bar{\mathscr{F}}_c
    [-d]) (-\tfrac{d}{2})
  \]
  as a $G(\qp) \times W_E$-representation. This allows us to access the
  cohomology of Shimura varieties without comparing it to the cohomology its
  good reduction locus, i.e., without appealing to the work of Lan--Stroh
  \cite{LS18}, which is currently unavailable in the abelian type case. The
  argument of \cite[Section~5]{HL24p} can then be reinterpreted more naturally
  in terms of semi-perversity of $\bar{\mathscr{F}}_c$, see
  \cite[Proposition~3.7]{HL24p}.
\end{remark}

\subsection{Outline of the proof}
We discuss the proof of the main functoriality result. The first step is to
realize that an Igusa stack contains the same information as the morphism
\begin{align*}
  \Theta \colon \sh\d \times_{\bung\m} \grg\m &\cong \sh\d \times_{\igs} \sh\d
  \\ &\xrightarrow{\mathrm{pr}_2} \sh\d
\end{align*}
satisfying certain conditions. This is because $\Theta, \mathrm{pr}_1$
determines a groupoid object on $\sh\d$ that identifies $\igs$ as the
quotient. Then functoriality of Igusa stacks amounts to the fact that the
diagram
\[ \begin{tikzcd}
  \sh\d \times_{\bung\m} \grg\m \arrow{d}{\Theta} \arrow{r} & \sh\d'
  \times_{\bung\m'} \grg\m' \arrow{d}{\Theta^\prime} \\ \sh\d \arrow{r} & \sh\d'
\end{tikzcd} \]
commutes, see \Cref{Prop:IgusaThetaFunct}.

We note that there is a natural projection map
\[
  \mathcal{R} = \sh\d \times_{\bung\m} \grg\m \to \Spd E \times \Spd E,
\]
and that there are closed v-subsheaves
\[
  \Spd E \xrightarrow{\Delta} \Spd E \times_{\Spd k_E} \Spd E \hookrightarrow
  \Spd E \times \Spd E,
\]
where $k_E$ is the residue field of $E$. We base change $\mathcal{R}$ to these
closed subsheaves and denote them by $\mathcal{R}^E \hookrightarrow
\mathcal{R}^{k_E} \hookrightarrow \mathcal{R}$.

We next prove that the diagram commutes upon restricting to $\mathcal{R}^E
\subseteq \mathcal{R}$, see \Cref{Prop:SameUntilt}. To achieve that, we study
the local structure of
\[
  \mathcal{R}^E = \sh\d \times_{\bung\m \times \Spd E} \grg\m
  \xrightarrow{\Theta^E} \sh\d.
\]
Upon passing to finite level using $\mathsf{G}(\af)$-equivariance, we see that
the v-sheaf $\mathcal{R}_K^E$ at finite level $K$ behaves like a local Shimura
variety. By trivializing around each classical point $y \in \sh[K]$ the de Rham
realization of the canonical $p$-adic local system on $\sh[K]$, we show that on
a small neighborhood of the graph $\sh[K] \to \mathcal{R}_K^E$ the behavior of
$\Theta_K^E$ is uniquely determined and functorial. On the other hand, because
these spaces are smooth rigid analytic varieties over $E$, this implies that the
maps agree on the entire connected component of $\mathcal{R}_{K,y}^E$. Passing
to the limit and using $G(\qp)$-equivariance, we conclude using the idea of
\cite{GLX23p} that the locus on which the diagram commutes contains
$\mathcal{R}^E$.

To pass from $\mathcal{R}^E$ to $\mathcal{R}^{k_E}$, we observe that the v-sheaf
\[
  \mathcal{R}^{k_E} = \sh\d \times_{\bung\m \times \Spd k_E} \grg\m
  \xrightarrow{\mathrm{pr}_2} \Spd E
\]
is far from being representable by an adic space, as it does not have enough
functions defined on it. This makes the map $\Theta^{k_E} \colon
\mathcal{R}^{k_E} \to \sh\d$ quite restrictive. Indeed, we prove quite generally
that for a smooth rigid analytic variety $X/E$ and a map of v-sheaves
\[
  f \colon \mathcal{R}_K^{k_E} \to X^\lozenge
\]
living over $\mathrm{pr}_2 \colon \Spd E \times_{\Spd k_E} \Spd E \to \Spd E$,
if $f$ sends $\mathcal{R}_K^E$ to a closed rigid analytic subvariety $Z^\lozenge
\subseteq X^\lozenge$, then $f$ factors through $Z^\lozenge$, see
\Cref{Prop:RFinFactorPt}. Applying this result to the diagonal $\Delta \colon
\sh'[K^\prime] \hookrightarrow \sh'[K^\prime] \times_{E^\prime} \sh'[K^\prime]$
and taking the limit as $K \to \{1\}$, we conclude that the diagram commutes
over $\mathcal{R}^{k_E}$.

Finally, to pass from $\mathcal{R}^{k_E}$ to all of $\mathcal{R}$, we simply use
the absolute Frobenius $\phi$, which translates $\mathcal{R}^{k_E}$ around to
cover all $\mathcal{R}$. This shows that the diagram commutes on all of
$\mathcal{R}$.

\subsection{Outline of paper}
In \Cref{Sec:Preliminaries}, we review the theory of mixed characteristic
local shtukas. In \Cref{Sec:DeRham} and \Cref{Sec:HodgeTate}, we
review the construction of the Hodge--Tate period map for general Shimura
varieties from \cite{Han16p} and \cite{PR24}. Our construction makes use of a
morphism of topoi $\mathsf{Shv}(X_\mathrm{v}) \to
\mathsf{Shv}(X_\mathrm{proet})$ instead of a pro-\'{e}tale descent result for
$\period{B}{dR}\p$-local systems over perfectoid spaces.

In \Cref{Sec:Uniformization}, we give the axiomatic definition of the
Igusa stack and work out the formalism for translating the information of an
Igusa stack to the information of a global uniformization map $\Theta$. In
\Cref{Sec:FiniteLevel}, we discuss how to pass to finite level and regard
$\Theta$ as a limit of finite level global uniformization maps $\Theta_K$ for
neat compact open subgroups $K \subseteq \mathsf{G}(\af)$. The complications in
this section mainly arise from the anti-cuspidal part of the center of
$\mathsf{G}$, and thus is unnecessary when dealing with Shimura data with
cuspidal center.

The proof of the main functoriality result is carried out in
Sections~\ref{Sec:LocalShimura}--\ref{Sec:DifferentUntilt}. In
\Cref{Sec:LocalShimura}, we define and study local Shimura varieties in a
setting that is slightly more general compared to \cite{SW20}. Using these
properties, we prove in \Cref{Sec:SameUntilt} the functoriality of global
uniformization restricted to $\mathcal{R}^E \hookrightarrow \mathcal{R}$. In
\Cref{Sec:StablyVComp}, we prove a criterion for a map between v-sheaves
associated to adic spaces to uniquely descent to a map between adic spaces, and
apply it in \Cref{Sec:DifferentUntilt} to obtain the full functoriality
of global uniformization maps.

As an application of these techniques, we prove in \Cref{Sec:Subdata}
that if $\gx \hookrightarrow \gx'$ is an embedding of Shimura data, then the
existence of an Igusa stack for $\gx'$ implies the existence of an Igusa stack
for $\gx$.

\subsection{Conventions and notations}
We list some conventions and notations we use throughout the paper.

\begin{itemize}
  \item We fix a prime $p$, which does not change throughout the paper.
  \item Perfectoid Tate--Huber pairs in characteristic $p$ are denoted by $(R,
    R^+)$, and $(R^\sharp, R^{\sharp+})$ denotes an untilt. Conversely, a
    general perfectoid Tate--Huber is denoted by $(R^\sharp, R^{\sharp+})$,
    which has tilt $(R, R^+)$. We denote by $\mathsf{Perf}$ the category of
    characteristic $p$ perfectoid spaces.
  \item Shimura data are denote by $\gx$, with reflex field $\mathsf{E}$. We
    often choose a place $v \mid p$ of $\mathsf{E}$, and write $G =
    \mathsf{G}_{\qp}$ and $E = \mathsf{E}_v$. We also often fix an algebraic
    closure $\qpbar$ and an embedding $E \hookrightarrow \qpbar$, which induces
    a conjugacy class of (geometric) Hodge cocharacters $\{ \mu \colon
    \mathbb{G}_{m,\qpbar} \to G_{\qpbar} \}$.
  \item Our definition of the Hodge cocharacter is the composition of a point $x
    \in \mathsf{X}$ with the natural inclusion $\mathbb{G}_{m,\mathbb{C}}
    \hookrightarrow (\operatorname{Res}_{\mathbb{C}/\mathbb{R}}
    \mathbb{G}_{m,\mathbb{C}})_\mathbb{C}$ given by $(1, 0)$. We note that this
    agrees with \cite{Mil05} but is the opposite of \cite{Pin90}, \cite{KSZ21p}.
    However, in both conventions, the Hodge cocharacter for the Shimura datum
    giving rise to modular curves is $\mu(z) = \operatorname{diag}(1, z)$.
  \item When discussing rational $p$-adic Hodge theory, we usually denote by $K$
    a discretely-valued complete non-archimedean field containing $\qp$ with
    perfect residue field. We usually denote by $X/K$ a smooth rigid analytic
    variety and $G/\qp$ a linear algebraic group that is often connected
    reductive.
  \item We use the convention that all torsors are left torsors unless denoted
    otherwise.
  \item For the Schubert cells of the affine Grassmannian, flag variety, etc.,
    we use the sign conventions compatible with \cite{SW20}, \cite{PR24} but the
    opposite of \cite{CS17}.
\end{itemize}

\subsection{Acknowledgements}
I would like to thank my advisors Richard Taylor and Pol van Hoften for their
continued encouragements and guidance, and the latter also for suggesting the
argument of \Cref{Sec:Boundary}. I would also like to thank Patrick Daniels,
Stepan Kazanin, Kiran Kedlaya, Vaughan McDonald, Mingjia Zhang, and Xinwen Zhu
for helpful comments, discussions, and suggestions.

}
{\section{Preliminaries} \label{Sec:Preliminaries}
\def\PerfPair{(R^\sharp, R^{\sharp+})}
\def\grgmu{\mathrm{Gr}_{G,\{\mu\},E}}
\def\flgmu{\mathrm{Fl}_{G,\{\mu\},E}}
\def\bungmu{\mathrm{Bun}_{G,\{\mu\}}}
\def\grge{\mathrm{Gr}_{G,E}}
\def\grgmuqp{\mathrm{Gr}_{G,\{\mu\},\qp}}

In this section, we review the foundations of the theory of mixed characteristic
shtukas in characteristic zero. We fix a prime $p$.

\paragraph
Throughout the paper, we fix an uncountable strong limit cardinal $\kappa$. We
denote by $\mathsf{Perf}$ the category of $\kappa$-small characteristic $p$
perfectoid spaces, defined as in \cite[Definition~4.2]{Sch17p}. Let
$\mathsf{Perf}_\mathrm{v}$ be the site whose underlying category is
$\mathsf{Perf}$ and covers are v-covers in the sense of
\cite[Definition~8.1(iii)]{Sch17p}, i.e., a collection $\lbrace Y_i \to X
\rbrace_{i \in I}$ is a cover when for every quasi-compact open subset $U
\subseteq X$ there exists a finite subset $J \subseteq I$ and quasi-compact open
subsets $V_j \subseteq Y_j$ for $j \in J$ such that $\bigcup_{j \in J} V_j \to
U$ is surjective. We refer to sheaves and stacks on $\mathsf{Perf}_\mathrm{v}$
simply as v-sheaves and v-stacks.

\paragraph \label{Par:Lozenge}
Given a (pre-)adic space $X \to \Spa \zp$, there is a presheaf $X^\lozenge$ on
$\mathsf{Perf}_\mathrm{v}$ defined as
\[
  (X^\lozenge)(S) = \lbrace (\iota, S^\sharp, f) \text{ where } \iota \colon
  (S^\sharp)^\flat \cong S, f \colon S^\sharp \to X \rbrace / \cong,
\]
which is in fact a v-sheaf by \cite[Lemma~18.1.1]{SW20}. When $X = \Spa(A, A^+)$
is affinoid, we also write $X^\lozenge = \Spd(A, A^+)$. Given a smooth algebraic
variety $Y / \qp$, we may take its analytification $Y^\mathrm{ad} \to \Spa \qp$,
and then consider its associated v-sheaf. We denote this by $Y^\lozenge \in
\mathsf{Shv}(\mathsf{Perf}_\mathrm{v})$ as well. We note that when $Y$ is
quasi-compact as a scheme, the v-sheaf $Y^\lozenge$ need not be quasi-compact. 

\paragraph \label{Par:BdR}
Given a perfectoid Tate--Huber pair $\PerfPair$ with tilt $(R,
R^+)$, there is a natural ring homomorphism
\[
  \theta \colon \period{A}{inf}\PerfPair = W(R^+) \twoheadrightarrow R^{\sharp+}
\]
whose kernel is a principal ideal $\ker \theta = (\xi)$. The absolute Frobenius
$\phi$ on $R^+$ induces an Frobenius morphism $\phi \colon
\period{A}{inf}\PerfPair \xrightarrow{\cong} \period{A}{inf}\PerfPair$. After
choosing a pseudo-uniformizer $\varpi \in R^+$, we define
\begin{align*}
  \period{B}{dR}\p\PerfPair &= \varprojlim_n (\period{A}{inf}\PerfPair /
  (\xi^n))[1/[\varpi]], \\ \period{B}{dR}\PerfPair &=
  \period{B}{dR}\p\PerfPair[\xi^{-1}],
\end{align*}
where both period rings are independent of the choices of $\xi$ or $\varpi$.
Using \cite[Theorem~8.7]{Sch17p} and following the argument of
\cite[Corollary~17.1.9]{SW20}, we see that $(R, R^+) \mapsto \coprod_{(\iota,
R^\sharp)} \period{B}{dR}\p(R^\sharp, R^{\sharp+})$ defines a v-sheaf that we
denote by $\period{B}{dR}\p \to \Spd \zp$. 

\paragraph
Assume furthermore that $p$ is invertible in $R^\sharp$, so that the
$\varpi^\sharp$-adic topology on $R^{\sharp+}$ agrees with the $p$-adic
topology. By \cite[Lemma~3.9]{BMS18}, there exists an pseudo-uniformizer
$\varpi^\prime \in R^+$ with the property that $(\varpi^\prime)^\sharp \in p
(R^{\sharp+})^\times$. We replace $\varpi$ with $\varpi^\prime$ and henceforth
assume that $\varpi^\sharp \in p (R^{\sharp+})^\times$. We define the
sousperfectoid analytic adic space
\[
  \mathcal{Y}_{[0,\infty]}\PerfPair = \Spa (\period{A}{inf}\PerfPair,
  \period{A}{inf}\PerfPair) \setminus \lbrace \lvert p \rvert = \lvert [\varpi]
  \rvert = 0 \rbrace,
\]
see \cite[Proposition~13.1.1]{SW20}. For rational numbers $0 \le r \le s \le
\infty$, we define the open subspaces
\[
  \mathcal{Y}_{[r,s]}\PerfPair = \lbrace x \in \mathcal{Y}_{[0,\infty]}\PerfPair
  : \lvert p \rvert_x^s \le \lvert [\varpi] \rvert_x \le \lvert p \rvert_x^r
  \rbrace \subseteq \mathcal{Y}_{[0,\infty]}\PerfPair.
\]
We note that this subset does not depend on the choice of $\varpi$, because any
two such $\varpi$ differ by an element of $(R^+)^\times$. The absolute Frobenius
$\phi$ induces an isomorphism
\[
  \phi \colon \mathcal{Y}_{[r,s]}\PerfPair \xrightarrow{\cong} \mathcal{Y}_{[pr,
  ps]}\PerfPair.
\]

\paragraph
We define the analytic adic space
\[
  \mathcal{Y}_{(0,\infty)}\PerfPair = \bigcup_{0 < r < s < \infty}
  \mathcal{Y}_{[r,s]}\PerfPair = \mathcal{Y}_{[0,\infty]}\PerfPair \setminus
  \lbrace \lvert p[\varpi] \rvert = 0 \rbrace,
\]
and the \textdef{Fargues--Fontaine curve} as
\[
  \mathcal{X}_\mathrm{FF}(R, R^+) = \mathcal{X}_\mathrm{FF}\PerfPair =
  \mathcal{Y}_{(0,\infty)}\PerfPair / \phi^\mathbb{Z},
\]
where the $\phi$-action is free totally discontinuous, see
\cite[Section~II.1.1]{FS21p}. This is again an sousperfectoid analytic adic
space over $\Spa \qp$. There is a closed Cartier divisor
\[
  \Spa \PerfPair \to \mathcal{Y}_{[0,\infty)}\PerfPair = \bigcup_{0 < r <
  \infty} \mathcal{Y}_{[0,r]}\PerfPair
\]
induced by the ring homomorphism $\theta$, see \cite[Proposition~11.3.1]{SW20},
even if $p$ is not invertible in $R^\sharp$. When $p$ is invertible in
$R^\sharp$, this Cartier divisor lies in $\mathcal{Y}_{(0,\infty)}$ and hence
also defines a Cartier divisor $\Spa \PerfPair \to
\mathcal{X}_\mathrm{FF}\PerfPair$.

\paragraph \label{Par:AffineGrass}
Let $G/\qp$ be a linear algebraic group. We define, following
\cite[Proposition~19.1.2]{SW20}, the \textdef{affine Grassmannian} $\grg \in
\mathsf{Shv}(\mathsf{Perf}_\mathrm{v})$ as the functor sending $(R, R^+) \in
\mathsf{Perf}$ to the isomorphism classes of tuples $(\iota, R^\sharp,
\mathscr{P}, \alpha)$ where
\begin{itemize}
  \item $\iota \colon (R^\sharp)^\flat \cong R$ is an untilt where $R^\sharp$ is
    a $\qp$-algebra,
  \item $\mathscr{P}$ is a left $G$-torsor on $\Spec
    \period{B}{dR}\p\PerfPair$,
  \item $\alpha \colon G \times_{\qp} \Spec \period{B}{dR}\PerfPair
    \xrightarrow{\cong} \mathscr{P} \vert_{\period{B}{dR}\PerfPair}$ is a
    trivialization over $\period{B}{dR}\PerfPair$.
\end{itemize}
This is indeed a v-sheaf, and can be understood as the sheafification of the
functor
\[
  (R, R^+) \mapsto \coprod_{\iota \colon (R^\sharp)^\flat \cong R}
  G(\period{B}{dR}\PerfPair) / G(\period{B}{dR}\p\PerfPair).
\]
Concretely, for a local trivialization $\tau \colon G \times_{\qp} \Spec
\period{B}{dR}\p\PerfPair \cong \mathscr{P}$, the composition
\[
  G \times_{\qp} \Spec \period{B}{dR}\PerfPair \xrightarrow{\alpha} \mathscr{P}
  \vert_{\period{B}{dR}\PerfPair} \xrightarrow{\tau^{-1}} G \times_{\qp} \Spec
  \period{B}{dR}\PerfPair
\]
is right multiplication by an element $g \in G(\period{B}{dR}\PerfPair)$, where
changing $\tau$ amounts to multiplying $g$ by an element of
$G(\period{B}{dR}\p\PerfPair)$ on the right. There is a natural morphism
\[
  \grg \to \Spd \qp
\]
that remembers only the untilt $(\iota, R^\sharp)$. There is also a group action
\[
  G(\period{B}{dR}\p) \times_{\Spd \qp} \grg \to \grg; \quad (g, (\mathscr{P},
  \alpha)) \mapsto (\mathscr{P}, \alpha \circ r_g)
\]
where $r_g$ is the $G$-torsor automorphism of $G \times_{\qp} \Spec
\period{B}{dR}\PerfPair$ given by right multiplication by $g$.

\paragraph
Let $G/\qp$ be a connected reductive group, and let $\lbrace \mu \colon
\mathbb{G}_{m,\qpbar} \to G_{\qpbar} \rbrace$ be a conjugacy class of
cocharacters, whose field of definition is $E \subset \qpbar$. Let $T$ be the
abstract maximal torus of $G$, and let $X_\ast(T)$ be the lattice of
cocharacters with the action of $\Gal(\qpbar/\qp)$. There is a canonical subset
$X_\ast(T)^+ \subseteq X_\ast(T)$ of dominant cocharacters, which is
$\Gal(\qpbar/\qp)$-stable, and the geometric conjugacy class $\lbrace \mu
\rbrace$ corresponds to an element $\mu \in X_\ast(T)^+$ whose stabilizer under
the Galois action is $\Gal(\qpbar/E)$. As in \cite[Proposition~19.2.1]{SW20},
for each algebraically closed perfectoid field $C^\sharp$ with a fixed embedding
$\qpbar \hookrightarrow C^\sharp$, upon choosing an embedding $T_{\qpbar}
\hookrightarrow B \hookrightarrow G_{\qpbar}$ under which $X_\ast(T)^+$
corresponds to $B$, we obtain a Cartan decomposition
\[
  G(\period{B}{dR}(C^\sharp, C^{\sharp+})) = \coprod_{\lambda \in
  X_\ast(T)^+} G(\period{B}{dR}\p(C^\sharp, C^{\sharp+})) \xi^{\lambda}
  G(\period{B}{dR}\p(C^\sharp, C^{\sharp+})).
\]
These double cosets are independent of the choice of $\xi$, because any two
differ by an element of $\period{B}{dR}\p(C^\sharp, C^{\sharp+})^\times$, and
also independent of the choice of $T_{\qpbar} \hookrightarrow G_{\qpbar}$,
because any two are conjugate by an element of $G(\qpbar) \subseteq
G(\period{B}{dR}\p)$. It also follows that the double coset
\[
  G(\period{B}{dR}\p(C^\sharp, C^{\sharp+})) \xi^\mu
  G(\period{B}{dR}\p(C^\sharp, C^{\sharp+})) \subseteq
  G(\period{B}{dR}(C^\sharp, C^{\sharp+}))
\]
only depends on the embedding $E \hookrightarrow C^\sharp$.

\begin{definition}[{\cite[Definition~19.2.2]{SW20}}]
  When $(C, C^+)$ is an algebraically closed perfectoid field in characteristic
  $p$, we define
  \begin{align*}
    \grgmu(C, C^+) &= \coprod_{E \to C^\sharp} G(\period{B}{dR}\p(C^\sharp,
    C^{\sharp+})) \xi^\mu G(\period{B}{dR}\p(C^\sharp, C^{\sharp+})) /
    G(\period{B}{dR}\p(C^\sharp, C^{\sharp+})) \\ &\subseteq \grge(C, C^+),
  \end{align*}
  where the disjoint union is over all untilts of $C$ over $E$. More generally,
  we define
  \[
    \grgmu(R, R^+) \subseteq \grge(R, R^+)
  \]
  as the subset of those elements whose pushforward along every $(R, R^+) \to
  (C, C^+)$, where $C$ is a algebraically closed field, lies in $\grgmu(C,
  C^+)$.
\end{definition}

\begin{example}
  When $G = \mathrm{GL}_{n,\qp}$, we may use the construction $\mathscr{V}
  \mapsto \operatorname{Isom}(\mathscr{V}, \mathscr{O}_X^{\oplus n})$ to
  identify rank $n$ vector bundles and left $\mathrm{GL}_{n,\qp}$-torsors. Under
  this equivalence, a $\period{B}{dR}\p\PerfPair$-lattice $\Xi \subseteq
  \period{B}{dR}\PerfPair^{\oplus n}$ induces a point in $\grg\PerfPair$, where
  the normalization is such that when
  \[
    \Xi = \xi \period{B}{dR}\p\PerfPair^{\oplus a} \oplus
    \period{B}{dR}\p\PerfPair^{\oplus (n-a)} \subseteq
    \period{B}{dR}\PerfPair^{\oplus n},
  \]
  the induced point lies in $\grgmuqp$ with $\mu = (1^a, 0^{n-a})$.
\end{example}

\paragraph \label{Par:RelativePosition}
Consider an $E$-algebra $\PerfPair$, two $G$-torsors $\mathscr{P}, \mathscr{Q}$
on $\Spec \period{B}{dR}\p\PerfPair$, and an isomorphism
\[
  \alpha \colon \mathscr{P} \vert_{\period{B}{dR}\PerfPair} \xrightarrow{\cong}
  \mathscr{Q} \vert_{\period{B}{dR}\PerfPair}.
\]
We say that $\alpha$ has \textdef{relative position} $\mu$ when for every
geometric point $\PerfPair \to (C^\sharp, C^{\sharp+})$ and trivializations
$\tau_\mathscr{P} \colon G_{\period{B}{dR}\p(C^\sharp, C^{\sharp+})} \cong
\mathscr{P} \vert_{\period{B}{dR}\p(C^\sharp, C^{\sharp+})}$ and
$\tau_\mathscr{Q} \colon G_{\period{B}{dR}\p(C^\sharp, C^{\sharp+})} \cong
\mathscr{Q} \vert_{\period{B}{dR}\p(C^\sharp, C^{\sharp+})}$, the composition
\[
  G_{\period{B}{dR}(C^\sharp, C^{\sharp+})} \xrightarrow{\tau_\mathscr{P}}
  \mathscr{P} \vert_{\period{B}{dR}(C^\sharp, C^{\sharp+})} \xrightarrow{\alpha}
  \mathscr{P} \vert_{\period{B}{dR}(C^\sharp, C^{\sharp+})}
  \xrightarrow{\tau_\mathscr{Q}^{-1}} G_{\period{B}{dR}(C^\sharp, C^{\sharp+})}
\]
is right multiplication by an element in the double coset
\[
  G(\period{B}{dR}\p(C^\sharp, C^{\sharp+})) \xi^\mu
  G(\period{B}{dR}\p(C^\sharp, C^{\sharp+})) \subseteq
  G(\period{B}{dR}(C^\sharp, C^{\sharp+})).
\]

\begin{proposition}[{\cite[Proposition~19.2.3]{SW20}}]
  The subfunctor $\grgmu \hookrightarrow \grge$ is a locally closed v-subsheaf,
  i.e., it is a closed subsheaf in an open subsheaf in the sense of
  \cite[Definition~10.7]{Sch17p}.
\end{proposition}

\paragraph
As before, let $G/\qp$ be a connected reductive group, and let $\lbrace \mu
\colon \mathbb{G}_{m,\qpbar} \to G_{\qpbar} \rbrace$ be a geometric conjugacy
class of cocharacters. For each geometric cocharacter $\mu \in \lbrace \mu
\rbrace$, there exists a parabolic subgroup
\[
  P_\mu = \Bigl\lbrace g \in G : \lim_{t \to \infty} \mu(t) g \mu(t)^{-1}
  \text{ exists} \Bigr\rbrace \subseteq G_{\qpbar}.
\]
Because $P_\mu$ is its own normalizer, the space of parabolic subgroups
conjugate to $P_\mu$ is $G_{\qpbar} / P_\mu$. Moreover, this projective variety
has a natural Galois descent datum over $E$, and hence defines a projective
variety
\[
  \flgmu \to \Spec E.
\]

\begin{theorem}[{\cite[Proposition~19.4.2]{SW20}, \cite[Proposition~3.4.3]{CS17}}]
  \label{Thm:BialynickiBirula}
  If $\lbrace \mu \rbrace$ is a geometric conjugacy class of minuscule
  cocharacters, then there is a natural Bia\l{}ynicki-Birula isomorphism
  \[
    \grgmu \cong \flgmu^\lozenge
  \]
  over $\Spd E$.
\end{theorem}

\begin{remark}
  For $P_\mu$, $\grgmu$, and $\flgmu$, our sign convention agrees with
  \cite{SW20}, \cite{PR24} and is the opposite of \cite{CS17}.
\end{remark}

\paragraph \label{Par:BeauvilleLaszlo}
For $G/\qp$ a linear algebraic group, the functor
\[
  \bung(R, R^+) = \lbrace \text{left } G\text{-torsors on }
  \mathcal{X}_\mathrm{FF}(R, R^+) \rbrace
\]
defines a v-stack $\bung \in \mathsf{Stk}(\mathsf{Perf}_\mathrm{v})$, where the
notion of a $G$-torsor is defined as in \cite[Theorem~19.5.2]{SW20}. There is a
Beauville--Laszlo map
\[
  \mathrm{BL} \colon \grg \to \bung
\]
that sends the tuple $(\iota, R^\sharp, \mathscr{P}, \alpha)$ to the torsor on
$\mathcal{X}_\mathrm{FF}\PerfPair$ obtained by modifying the trivial $G$-torsor
along $\alpha$ to $\mathscr{P}$.

\paragraph \label{Par:BLEquivariance}
The affine Grassmannian $\grg$ has a natural $\underline{G(\qp)}$-action, and
moreover for every $x \in \grg(R^\sharp, R^{\sharp+})$ and $g \in
\underline{G(\qp)}(R^\sharp, R^{\sharp+})$ we have a natural isomorphism
\[
  r_g \colon \mathrm{BL}(gx) \xrightarrow{\cong} \mathrm{BL}(x)
\]
of $G$-torsors on $\mathcal{X}_\mathrm{FF}(R, R^+)$, which is the unique
isomorphism extending right multiplication by $g$ away from the divisor
$\Spa(R^\sharp, R^{\sharp+}) \hookrightarrow \mathcal{X}_\mathrm{FF}(R, R^+)$.
This satisfies the natural cocycle condition, and hence define a
$\underline{G(\qp)}$-equivariant structure on the map $\mathrm{BL} \colon \grg
\to \bung$, where $\bung$ is given the trivial action. Therefore we further
obtain a map
\[
  \mathrm{BL} \colon [\underline{G(\qp)} \backslash \grg] \to \bung.
\]

\paragraph \label{Par:ConstantTorsor}
For every element $b \in G(\qpbr)$ and a perfectoid $\fpbar$-algebra $(R, R^+)$,
we define a left $G$-torsor $\mathscr{P}_b$ on $\mathcal{X}_\mathrm{FF}$ given
by descending the Frobenius-twisted automorphism
\[
  \phi^\ast(G \times_{\qp} \mathcal{Y}_{(0,\infty)}(R, R^+)) \cong G
  \times_{\qp} \mathcal{Y}_{(0,\infty)}(R, R^+) \xrightarrow{r_{b^{-1}}} G
  \times_{\qp} \mathcal{Y}_{(0,\infty)}(R, R^+),
\]
where $r_{b^{-1}}$ is right multiplication by $b^{-1}$, noting that
$\mathcal{Y}_{(0,\infty)}(R, R^+)$ is naturally an adic space over $\Spa \qpbr$.
For every $g \in G(\qpbr)$ and $b^\prime = g b \phi(g)^{-1}$ there is an
isomorphism $r_g \colon \mathscr{P}_{b^\prime} \to \mathscr{P}_b$. In
particular, the isomorphism class of $\mathscr{P}_b$ only depends on the
$\phi$-conjugacy class $[b]$ of $b$, defined as in \cite[Section~1.7]{Kot85}.

\begin{remark}
  An alternative way of describing this $G$-torsor $\mathscr{P}_b$ on
  $\mathcal{X}_\mathrm{FF}(R, R^+)$ is to view it as the quotient
  \[
    \mathscr{P}_b = (G \times_{\qp} \mathcal{Y}_{(0,\infty)}(R, R^+)) / (r_b
    \times \phi)^\mathbb{Z}.
  \]
  This also agrees with the convention in \cite{FS21p}.
\end{remark}

\paragraph
Assume moreover that $G$ is connected reductive. By \cite[Theorem~1.1]{Vie24},
there exists a homeomorphism
\[
  B(G) \cong \lvert \bung \rvert; \quad b \mapsto \mathscr{P}_b,
\]
where $B(G)$ is the set of $\phi$-conjugacy classes of $G(\qpbr)$ and is given
the topology where $[b^\prime]$ is in the closure of $[b]$ if and only if
$[b^\prime] \ge [b]$ in the partial order constructed in
\cite[Section~2.1]{Vie24}. Write $B(G, \{\mu\}) \subseteq B(G)$
for the open subset consisting of $b \in B(G)$ satisfying $\kappa(b) =
\mu^\natural \in \pi_1(G)_{\Gal(\qpbar/\qp)}$ and $\nu(b) \le \mu^\diamond \in
X_\ast(T)^+_\mathbb{Q}$, where $\mu^\natural$ and $\mu^\diamond$ are the image
and the Galois average of $\mu \in X_\ast(T)^+$. This corresponds to an open
v-substack
\[
  \bungmu \hookrightarrow \bung.
\]

\begin{proposition}[{\cite[Proposition~A.9]{Rap18}, \cite[Corollary~6.4.2]{DvHKZ24p}}]
  \label{Prop:BLSurjective}
  The Beauville--Laszlo map $\mathrm{BL} \colon \grg \to \bung$ sends $\grgmu$
  into $\bungmu$. Moreover, the induced map
  \[
    \mathrm{BL} \colon \grgmu \to \bungmu \times \Spd E
  \]
  is surjective as maps of v-stacks.
\end{proposition}
}
{\section{de Rham local systems} \label{Sec:DeRham}
\def\grgln{\mathrm{Gr}_{\mathrm{GL}_n,\qp}}
\def\grge{\mathrm{Gr}_{G,E}}
\def\grgmu{\mathrm{Gr}_{G,\{\mu\},E}}

Throughout this section, let $K$ be a discretely valued complete nonarchimedean
field extension of $\qp$ with perfect residue field, and let $X/K$ be a smooth
rigid analytic variety. We review the construction of the Hodge--Tate period map
attached to a de Rham local system or a de Rham $\underline{G(\qp)}$-torsor for
$G/\qp$ a linear algebraic group. While the construction was already discussed
in \cite[Section~5.2]{Han16p} and \cite[Section~2.6]{PR24}, we provide another
interpretation using a morphism of topoi $\gamma \colon
\mathsf{Shv}(X_\mathrm{v}) \to \mathsf{Shv}(X_\mathrm{proet})$.

\paragraph
As in \cite[Section~3]{Sch13} and \cite{Sch13c}, we define the pro-\'{e}tale
site $X_\mathrm{proet}$ of $X$ as a full subcategory of
$\mathsf{Pro}(X_\mathrm{et})$. There is also a v-site $X_\mathrm{v}$ whose
underlying category is the category of pairs $(S, f \colon S \to X)$ with $S$ a
perfectoid space, where a collection of morphisms $\lbrace (S_i, f_i) \to (T, g)
\rbrace$ is a cover when $\lbrace S_i \to T \rbrace$ is a v-cover. We say that
an object in $X_\mathrm{proet}$ is affinoid perfectoid when it has a
pro-\'{e}tale presentation $U = \varprojlim U_i$ with $U_i = \Spa(R_i, R_i^+)
\in X_\mathrm{et}$, where the $p$-adically completed limit
\[
  (R, R^+) = ((\varinjlim R_i^+)^\wedge[p^{-1}], (\varinjlim R_i^+)^\wedge)
\]
is perfectoid. Denote by $X_\mathrm{proet}^\mathrm{ap} \subseteq
X_\mathrm{proet}$ the full subcategory of affinoid perfectoid objects. Then
there is a functor
\[
  c \colon X_\mathrm{proet}^\mathrm{ap} \to X_\mathrm{v}; \quad \varprojlim
  \Spa(R_i, R_i^+) \mapsto \Spa(R, R^+),
\]
see the paragraph after \cite[Lemma~4.5]{Sch13}. In fact, this may be regarded
as the restriction of the functor
\[
  c \colon \mathsf{Pro}(X_\mathrm{et}) \to \mathsf{Shv}(X_\mathrm{v}); \quad
  \varprojlim U_i \mapsto \varprojlim U_i^\lozenge,
\]
which commutes with all small limits because $X_\mathrm{et} \to
\mathsf{Shv}(X_\mathrm{v})$ commutes with all finite limits, see
\cite[Section~8.7.1.7, Corollary~8.9.8]{SGA4I}.

\begin{remark}
  It is unclear whether the functor $c \colon X_\mathrm{proet}^\mathrm{ap} \to
  X_\mathrm{v}$ is fully faithful.
\end{remark}

\begin{remark}
  There is a natural identification $\mathsf{Shv}(X_\mathrm{v}) \cong
  \mathsf{Shv}(\mathsf{Perf}_\mathrm{v})_{/X^\lozenge}$ where a sheaf
  $\mathscr{F} \in \mathsf{Shv}(X_\mathrm{v})$ corresponds to the total
  space $\operatorname{Tot}(\mathscr{F}) \in
  \mathsf{Shv}(\mathsf{Perf}_\mathrm{v})$ defined as
  \[
    \operatorname{Tot}(\mathscr{F})(S) = \coprod_{S^\sharp \to X}
    \mathscr{F}(S^\sharp \to X).
  \]
\end{remark}

\begin{lemma} \label{Lem:Subcanonical}
  For every object $U \in \mathsf{Pro}(X_\mathrm{et})$ the functor $\Hom(-, U)
  \colon X_\mathrm{proet}^\mathrm{op} \to \mathsf{Set}$ is a sheaf on the site
  $X_\mathrm{proet}$.
\end{lemma}

\begin{proof}
  If we write $U = \varprojlim U_i$ then $\Hom(-, U_i) = \varprojlim \Hom(-,
  U_i)$. Because limits of sheaves are sheaves, we reduce to the case when $U
  \in X_\mathrm{et}$. Denote by $\mathcal{P}_U$ the class of morphisms $W \to V$
  that satisfy the sheaf condition
  \[
    \Hom(V, U) \xrightarrow{\cong} \operatorname{eq}(\Hom(W, U)
    \rightrightarrows \Hom(W \times_V W, U)).
  \]
  It suffices to show that $\mathcal{P}_U$ contains all surjective \'{e}tale
  maps in $X_\mathrm{et}$, and $\mathcal{P}_U$ is stable under cofiltered limits
  in $W$. The first part follows from \cite[Proposition~2.2.2]{Hub96}.
  For the second part, assume we have $W \to V$ that is a cofiltered limits.
  Since $U \in X_\mathrm{et}$ is a cocompact object in
  $\mathsf{Pro}(X_\mathrm{et})$, we have
  \begin{align*}
    \Hom(V, U) &= \varinjlim \Hom(V_i, U), \quad \Hom(W, U) = \varinjlim
    \Hom(W_i, U), \\ \Hom(W \times_V W, U) &= \Hom(\varprojlim (W_i \times_{V_i}
    W_i), U) = \varinjlim \Hom(W_i \times_{V_i} W_i, U).
  \end{align*}
  The claim now follows from filtered colimits commuting with equalizers.
\end{proof}

\paragraph
Recall from \cite[Lemma~3.10(vii)]{Sch13} that the category $X_\mathrm{proet}$
has all finite limits, which can be computed in $\mathsf{Pro}(X_\mathrm{et})$.
If $U_i \to U$ and $V \to U$ are morphisms in $X_\mathrm{proet}^\mathrm{ap}$
where $U_i \to U$ is pro-\'{e}tale, then $U_i \times_U V$ is also affinoid
perfectoid by the argument of \cite[Lemma~4.6]{Sch13}. Hence
$X_\mathrm{proet}^\mathrm{ap}$ is a site where the coverings are given by those
collections of morphisms that are coverings in $X_\mathrm{proet}$. Because
affinoid perfectoids form a basis in $X_\mathrm{proet}$, we obtain an
equivalence of topoi
\[
  \mathsf{Shv}(X_\mathrm{proet}^\mathrm{ap}) \simeq
  \mathsf{Shv}(X_\mathrm{proet}).
\]

\paragraph \label{Par:VToProet}
Let $\lbrace U_i \to U \rbrace$ be a cover in $X_\mathrm{proet}^\mathrm{ap}$ and
let $V \to U$ be a morphism in $X_\mathrm{proet}^\mathrm{ap}$. The collection
$\lbrace c(U_i) \to c(U) \rbrace$ of morphisms in $X_\mathrm{v}$ is a v-cover
because $\lvert c(U_i) \rvert \to \lvert U_i \rvert$ and $\lvert c(U) \rvert \to
\lvert U \rvert$ are all homeomorphisms. The map $c(U_i \times_U V) \to c(U_i)
\times_{c(U)} c(V)$ is an isomorphism since both sides are affinoid perfectoid
spaces representing the same v-sheaf. This shows that the functor $c \colon
X_\mathrm{proet}^\mathrm{ap} \to X_\mathrm{v}$ is a continuous functor in the
sense of \cite[Definition~00WV]{Stacks}. It follows that there is a functor
\[
  \gamma_\ast \colon \mathsf{Shv}(X_\mathrm{v}) \to
  \mathsf{Shv}(X_\mathrm{proet}^\mathrm{ap}); \quad \mathscr{F} \mapsto [U
  \mapsto \mathscr{F}(c(U))]
\]
with a left adjoint $\gamma^{-1}$. The following lemma implies that the functor
$c$ defines a morphism of sites $X_\mathrm{v} \to X_\mathrm{proet}^\mathrm{ap}$
in the sense of \cite[Definition~00X1]{Stacks}, and hence we obtain a morphism
of topoi
\[
  \gamma \colon \mathsf{Shv}(X_\mathrm{v}) \to
  \mathsf{Shv}(X_\mathrm{proet}^\mathrm{ap}) \simeq
  \mathsf{Shv}(X_\mathrm{proet}).
\]

\begin{lemma} \label{Lem:Topoi}
  The functor $\gamma^{-1}$ is left exact, i.e., commutes with finite limits.
\end{lemma}

\begin{proof}
  For $\mathscr{F} \in \mathsf{Shv}(X_\mathrm{proet}^\mathrm{ap})$, the inverse
  image $\gamma^{-1} \mathscr{F}$ is the sheafification of the assignment
  \[
    X_\mathrm{v} \ni W \mapsto \varinjlim_{W \to c(V)} \mathscr{F}(V).
  \]
  Here, the colimit is over the category where a morphism from $f \colon W \to
  c(V)$ to $f^\prime \colon W \to c(V^\prime)$ is given by a morphism $\alpha
  \colon V^\prime \to V$ for which $c(\alpha) \circ f^\prime = f$. Because
  sheafification is exact, and strictly totally disconnected spaces in the sense
  of \cite[Definition~7.15]{Sch17p} form a basis for the site $X_\mathrm{v}$, it
  suffices to show that the indexing category is filtered for strictly totally
  disconnected $W \in X_\mathrm{v}$. This follows from \Cref{Lem:Filtered}
  applied to $U = X$.
\end{proof}

\begin{lemma} \label{Lem:Filtered}
  Let $U \in X_\mathrm{proet}$ and $W \in X_\mathrm{v}$ be objects, and fix a
  map $\pi \colon W \to c(U)$ in $\mathsf{Shv}(X_\mathrm{v})$. Let
  $\mathcal{C}_{W,U,\pi}$ be the category where
  \begin{itemize}
    \item objects are triples $(V, f, g)$ where $V \in
      X_\mathrm{proet}^\mathrm{ap}$, $f \colon V \to U$ is a morphism in
      $X_\mathrm{proet}$, $g \colon W \to c(V)$ is a morphism in
      $X_\mathrm{v}$, satisfying $c(f) \circ g = \pi$,
    \item a morphism from $(V, f, g)$ to $(V^\prime, f^\prime, g^\prime)$ is a
      map $\alpha \colon V^\prime \to V$ for which the following diagrams
      commute.
      \[ \begin{tikzcd}
        W \arrow{r}{g} \arrow{d}{g^\prime} & c(V) \\ c(V^\prime)
        \arrow{ru}[']{c(\alpha)}
      \end{tikzcd} \begin{tikzcd}
        & V \arrow{d}{f} \\ V^\prime \arrow{r}{f^\prime} \arrow{ru}{\alpha} & U
      \end{tikzcd} \]
  \end{itemize}
  If $W$ is a strictly totally disconnected perfectoid space, then the category
  $\mathcal{C}_{W,U,\pi}$ is filtered.
\end{lemma}

\begin{proof}
  Given a finite (possibly empty) diagram $\{W \to c(V_i) \to c(U)\}_{i \in I}$
  in the category $\mathcal{C}_{W,U,\pi}$, we wish to show that there exists a
  cocone over the diagram. Write $U = V_\infty$ so that the diagram $V_\bullet
  \colon I^\mathrm{op} \to X_\mathrm{proet}$ extends to a diagram
  $(I^\mathrm{op})^\vartriangleright \to X_\mathrm{proet}$ sending the cone
  point to $U = V_\infty$. Consider the limit
  \[
    V = \varprojlim_{(I^\mathrm{op})^\vartriangleright} V_i
  \]
  in $X_\mathrm{proet}$, which exists by \cite[Lemma~3.10(vii)]{Sch13}. We
  obtain a map
  \[
    W \to \varprojlim_{(I^\mathrm{op})^\vartriangleright} c(V_i) \cong c(V)
  \]
  in $\mathsf{Shv}(X_\mathrm{v})$, which induces a continuous map of
  topological spaces $\lvert W \rvert \to \lvert c(V) \rvert \cong \lvert V
  \rvert$ by \cite[Lemma~11.22]{Sch17p}. By \cite[Lemma~3.10(iii)]{Sch13} there
  exists a quasi-compact open subobject $V_{-1} \subseteq V$ in
  $X_\mathrm{proet}$ for which the map $W \to c(V)$ factors through $c(V_{-1})$.

  Choose an affinoid perfectoid pro-\'{e}tale cover $\tilde{V} =
  \varprojlim_{\mu < \lambda} V_\mu \to V_{-1}$ in $X_\mathrm{proet}$ for some
  ordinal $\lambda$, where $V_0 \to V_{-1}$ is surjective \'{e}tale and each
  $V_\mu \to \varprojlim_{\mu^\prime < \mu} V_{\mu^\prime}$ is surjective finite
  \'{e}tale, see \cite{Sch13c}. Because $c$ commutes with all limits, we have
  $c(\tilde{V}) = \varprojlim_{\mu < \lambda} c(V_\mu)$ where $c(V_\mu) \to
  \varprojlim_{\mu^\prime < \mu} c(V_{\mu^\prime})$ are finite surjective
  \'{e}tale. Because $W$ is strictly totally disconnected, any map $W \to c(V)$
  now transfinitely lifts to $W \to c(\tilde{V})$. This gives maps
  \[
    W \to c(\tilde{V}), \quad \tilde{V} \to V_i \to U
  \]
  with $\tilde{V} \in X_\mathrm{proet}^\mathrm{ap}$. This shows that the
  indexing category is indeed filtered.
\end{proof}

\begin{lemma} \label{Lem:PullbackOfProet}
  Let $U \in X_\mathrm{proet}$ be an object, which we regard as an object of
  $\mathsf{Shv}(X_\mathrm{proet})$ using \Cref{Lem:Subcanonical}. Then
  $\gamma^{-1}(U) \cong c(U) \in \mathsf{Shv}(X_\mathrm{v})$.
\end{lemma}

\begin{proof}
  For every $W \in X_\mathrm{v}$, there is an obvious homomorphism
  \[
    f \colon \varinjlim_{W \to c(V)} \Hom(V, U) \to \Hom(W, c(U)),
  \]
  which upon sheafification, defines a map of sheaves $\gamma^{-1}(U) \to c(U)$.
  As in the proof of \Cref{Lem:Topoi}, it suffices to check that the above
  map is bijection when $W$ is strictly totally disconnected. Given a map $\pi
  \colon W \to c(U)$, we now observe that the fiber $f^{-1}(\pi)$ is identified
  with $\pi_0(\mathcal{C}_{W,U,\pi})$, which is a singleton by
  \Cref{Lem:Filtered} as filtered categories are connected.
\end{proof}

\paragraph
Let $T$ be a topological space. There is a sheaf $\underline{T}$ on
$X_\mathrm{proet}$ defined by the formula
\[
  \underline{T}(U) = \Hom_\mathrm{cts}(\lvert U \rvert, T),
\]
where the sheaf axiom follows from \cite[Lemma~3.10(i, iv)]{Sch13}. The same
formula also defines a sheaf $\underline{T}$ on the site $X_\mathrm{v}$.

\begin{corollary} \label{Cor:LocallyProfinteSet}
  For $T$ a locally profinite set, we have $\gamma_\ast \underline{T} \cong
  \underline{T}$ and $\gamma^{-1} \underline{T} \cong \underline{T}$.
\end{corollary}

\begin{proof}
  The first isomorphism follows from the homeomorphism $\lvert U \rvert \cong
  \lvert c(U) \rvert$ for $U \in X_\mathrm{proet}^\mathrm{ap}$. For the second,
  we write $\underline{T} = \varprojlim \underline{T_i}$ where $T_i$ are
  discrete sets and transitions maps are surjective with finite fibers. Then
  each $\underline{T_i}$ is representable by the object $T_i \times X \subseteq
  X_\mathrm{et}$ and hence $\underline{T}$ is representable by the object
  \[
    \varprojlim_i (T_i \times X) \in X_\mathrm{proet}.
  \]
  Applying \Cref{Lem:PullbackOfProet} we deduce that $\gamma^{-1}
  \underline{T}$ is representable by the object
  \[
    \varprojlim_i (T_i \times X)^\lozenge = \underline{T} \times X^\lozenge \in
    \mathsf{Shv}(X_\mathrm{v}).
  \]
  This agrees with the sheaf $\underline{T}$ on $X_\mathrm{v}$.
\end{proof}

\begin{proposition}[{\cite[Proposition~4.3]{Han16p}}] \label{Prop:Torsors}
  Let $G$ be a locally profinite group. The functors $\gamma_\ast$ and
  $\gamma^{-1}$ induce an equivalence of categories between
  $\underline{G}$-torsors on $X_\mathrm{proet}$ and $\underline{G}$-torsors on
  $X_\mathrm{v}$.
\end{proposition}

\begin{proof}
  Since the question is local in $X$, we may as well assume that $X$ is
  quasi-compact.

  Given a $\underline{G}$-torsor $\mathbb{P}$ on $X_\mathrm{proet}$, we first
  check that $\gamma^{-1} \mathbb{P}$ is a $\underline{G}$-torsor on
  $X_\mathrm{v}$ and that the natural map $\mathbb{P} \to \gamma_\ast
  \gamma^{-1} \mathbb{P}$ is an isomorphism. Because $X$ is quasi-compact,
  there exists an affinoid perfectoid pro-\'{e}tale cover $\tilde{X} \in
  X_\mathrm{proet}^\mathrm{ap}$ on which $\mathbb{P}$ is trivialized. Fixing
  one such a trivialization, we observe that $\mathbb{P}$ is determined by the
  descent datum, which corresponds to a continuous map $\sigma \colon \lvert
  \tilde{X} \times_X \tilde{X} \rvert \to G$. More precisely, we have
  \[
    \mathbb{P} = \operatorname{coeq}(\underline{G} \times \tilde{X} \times_X
    \tilde{X} \rightrightarrows \underline{G} \times \tilde{X}) \in
    \mathsf{Shv}(X_\mathrm{proet}),
  \]
  where the two maps are $\mathrm{pr}_{12}$ and $(g, \tilde{x}_1, \tilde{x}_2)
  \mapsto (g \sigma(\tilde{x}_1, \tilde{x}_2), \tilde{x}_2)$. Because
  $\gamma^{-1}$ commutes with colimits and finite limits, we have
  \[
    \gamma^{-1} \mathbb{P} = \operatorname{coeq}(\underline{G} \times
    c(\tilde{X}) \times_{X^\lozenge} c(\tilde{X}) \rightrightarrows
    \underline{G} \times c(\tilde{X})) \in \mathsf{Shv}(X_\mathrm{v})
  \]
  by \Cref{Lem:PullbackOfProet} and \Cref{Cor:LocallyProfinteSet}.
  Because $c(\tilde{X}) \to X^\lozenge$ is v-surjective, we conclude that
  $\gamma^{-1} \mathbb{P}$ is a $\underline{G}$-torsor on $X_\mathrm{v}$.

  To check that $\mathbb{P} \to \gamma_\ast \gamma^{-1} \mathbb{P}$ is an
  isomorphism, it is enough to check on those affinoid perfectoid pro-\'{e}tales
  $U \in X_\mathrm{proet}^\mathrm{ap}$ with the property that it admits a map $U
  \to \tilde{X}$. Because $\mathbb{P}$ is trivial over $\tilde{X}$, it is also
  trivial over $U$, and similarly because $\gamma^{-1}\mathbb{P}$ is trivial
  over $c(\tilde{X})$, it is trivial over $c(U)$. It follows that both
  $\mathbb{P}(U)$ and $(\gamma_\ast \gamma^{-1} \mathbb{P})(U)$ can be
  identified with $\Hom_\mathrm{cts}(\lvert U \rvert = \lvert c(U) \rvert, G)$.

  We now show for a $\underline{G}$-torsor $\mathbb{P}$ on $X_\mathrm{v}$ that
  $\gamma_\ast \mathbb{P}$ is a $\underline{G}$-torsor and the natural map
  $\gamma^{-1} \gamma_\ast \mathbb{P} \to \mathbb{P}$ is an isomorphism. By
  \cite[Lemma~10.13]{Sch17p}, for each open compact subgroup $K \subseteq G$ we
  have
  \[
    \mathbb{P} = \varprojlim_{K \to 1} (\underline{K} \backslash \mathbb{P})
  \]
  where each $\underline{K} \backslash \mathbb{P}$ is representable by a smooth
  rigid analytic variety $\mathbb{P}_K \to X$, \'{e}tale over $X$, where the
  transition maps are finite \'{e}tale. This defines an object $\varprojlim_{K
  \to 1} \mathbb{P}_K \in X_\mathrm{proet}$. We now claim that the natural map
  \[
    \varprojlim_{K \to 1} \mathbb{P}_K \to \gamma_\ast \mathbb{P}
  \]
  is an isomorphism in $\mathsf{Shv}(X_\mathrm{proet})$.

  To see this, it suffices to check equality after evaluating at objects $U \in
  X_\mathrm{proet}^\mathrm{ap}$ admitting a map $U \to \varprojlim_{K \to 1}
  \mathbb{P}_K$ in $X_\mathrm{proet}$. Because the left hand side is a
  $\underline{G}$-torsor in $X_\mathrm{proet}$, the $U$-points of the left hand
  side is identified with $\Hom_\mathrm{cts}(\lvert U \rvert, G)$. Similarly,
  $\mathbb{P} \to X^\lozenge$ is a $\underline{G}$-torsor splitting over
  $\varprojlim_{K \to 1} \mathbb{P}_K^\lozenge$, and hence the splitting over
  $c(U)$ induces an isomorphism $\mathbb{P}(c(U)) \cong \Hom_\mathrm{cts}(\lvert
  c(U) \rvert, G)$. This shows that $\gamma_\ast \mathbb{P} = \varprojlim_{K \to
  1} \mathbb{P}_K$ is a $\underline{G}$-torsor. Moreover
  \Cref{Lem:PullbackOfProet} implies that
  \[
    \gamma^{-1} \gamma_\ast \mathbb{P} = \varprojlim_{K \to 1}
    \mathbb{P}_K^\lozenge = \varprojlim_{K \to 1} \underline{K} \backslash
    \mathbb{P} = \mathbb{P}.
  \]
  This proves that $\gamma_\ast$ and $\gamma^{-1}$ induce an equivalence between
  the two notions of $\underline{G}$-torsors on $X$.
\end{proof}

\paragraph \label{Par:deRhamLocSys}
Let $K$ be a discretely valued complete nonarchimedean field extension
of $\qp$, and $X/K$ be a smooth rigid analytic variety. Let $\mathbb{L}$ be a
$\underline{\qp}$-local system over $X$ (meaning in either $X_\mathrm{proet}$ or
$X_\mathrm{v}$ under the equivalence of \Cref{Prop:Torsors}) of rank
$n$. There is a sheaf of rings $\period\o{B}{dR}$ on
$X_\mathrm{proet}$ defined in \cite[Definition~6.8]{Sch13} and \cite{Sch13c}.
Following \cite[Definition~8.3]{Sch13}, we say that $\mathbb{L}$ is a
\textdef{de Rham local system} when there exists a filtered vector bundle with
integrable connection $(D_\mathrm{dR}(\mathbb{L}), \nabla,
\mathrm{Fil}^\bullet)$ on $X$ satisfying Griffiths transversality together with
an isomorphism
\[
  \mathbb{L} \otimes_{\underline{\qp}} \period\o{B}{dR} \cong
  D_\mathrm{dR}(\mathbb{L}) \otimes_{\mathscr{O}_X} \period\o{B}{dR}
\]
in $\mathsf{Shv}(X_\mathrm{proet})$, respecting the connections and filtrations
on both sides.

\begin{theorem}[{\cite[Theorem~1.5]{LZ17}}] \label{Thm:LiuZhu}
  Assume that on each connected component of $X$, there is a classical point $x
  \in \lvert X \rvert$ (i.e., a point whose residue field $k(x)$ is a finite
  extension of $K$) at which $\mathbb{L} \vert_x$ corresponds to a de Rham
  representation $\Gal(\bar{K} / k(x)) \to \mathrm{GL}_n(\qp)$. Then
  $\mathbb{L}$ is a de Rham local system on $X$. In particular, $\mathbb{L}
  \vert_y$ is de Rham for every classical point $y \in \lvert X \rvert$, and for
  each integer $p$ the function $y \mapsto \dim_{k(y)} \mathrm{Fil}^p
  D_\mathrm{dR}(\mathbb{L} \vert_y)$ is locally constant.
\end{theorem}

\paragraph \label{Par:LocSysToGr}
Given a de Rham $\underline{\qp}$-local system $\mathbb{L} / X$ of rank $n$, we
have two subsheaves
\begin{align*}
  \mathbb{M} &= \mathbb{L} \otimes_{\underline{\qp}} \period{B}{dR}\p
  \subseteq \mathbb{L} \otimes_{\underline{\qp}} \period{B}{dR}, \\
  \mathbb{M}_0 &= (D_\mathrm{dR}(\mathbb{L}) \otimes_{\mathscr{O}_X}
  \period\o{B}{dR}\p)^{\nabla=0} \subseteq
  (D_\mathrm{dR}(\mathbb{L}) \otimes_{\mathscr{O}_X}
  \period\o{B}{dR})^{\nabla=0},
\end{align*}
where by \cite[Corollary~6.13]{Sch13} we have an isomorphism
\[
  \mathbb{L} \otimes_{\underline{\qp}} \period{B}{dR} \cong (\mathbb{L}
  \otimes_{\underline{\qp}} \period\o{B}{dR})^{\nabla=0} \cong
  (D_\mathrm{dR}(\mathbb{L}) \otimes_{\mathscr{O}_X}
  \period\o{B}{dR})^{\nabla=0}.
\]
On the other hand, both $\mathbb{L} \otimes_{\underline{\qp}}
\period{B}{dR}\p$ and $(D_\mathrm{dR}(\mathbb{L})
\otimes_{\mathscr{O}_X} \period\o{B}{dR}\p)^{\nabla=0}$ are
$\period{B}{dR}\p$-local systems on $X_\mathrm{proet}$ of rank $n$,
where for the latter we use \cite[Theorem~7.2]{Sch13}. For those objects $U \in
X_\mathrm{proet}^\mathrm{ap}$ over which both $\mathbb{L}$ and $\mathbb{M}_0$
admit trivializations, we obtain a $\Hom_\mathrm{cts}(\lvert U \rvert,
\mathrm{GL}_n(\qp))$-equivariant map
\[
  \lbrace \text{trivializations of } \mathbb{L} \vert_U \rbrace \to
  \grgln(c(U)),
\]
where we send a trivialization $\mathbb{L} \vert_U \cong \underline{\qp}^{\oplus
n}$ to the free $\period{B}{dR}\p(U)$-lattice
\[
  \mathbb{M}_0(U) \subseteq \mathbb{M}(U) \otimes_{\period{B}{dR}\p(U)}
  \period{B}{dR}(U) \cong \period{B}{dR}(U)^{\oplus n},
\]
noting that $\period{B}{dR}\p(U) \cong \period{B}{dR}\p(c(U))$
by \cite[Theorem~6.5]{Sch13}. The full subcategory of such objects $U \in
X_\mathrm{proet}^\mathrm{ap}$ form a basis of the topology, and hence we obtain
a $\underline{\mathrm{GL}_n(\qp)}$-equivariant map
\[
  \operatorname{Isom}(\mathbb{L}, \underline{\qp}^{\oplus n}) \to
  \gamma_\ast \grgln
\]
in $\mathsf{Shv}(X_\mathrm{proet})$. Using adjunction and
\Cref{Prop:Torsors}, we further obtain a
$\underline{\mathrm{GL}_n(\qp)}$-equivariant map
\[
  \mathrm{DRT}(\mathbb{L}) \colon \operatorname{Isom}(\mathbb{L},
  \underline{\qp}^{\oplus n}) = \gamma^{-1} \operatorname{Isom}(\mathbb{L},
  \underline{\qp}^{\oplus n}) \to \grgln
\]
in $\mathsf{Shv}(X_\mathrm{v})$.

\begin{proposition} \label{Prop:DRTOnAffPerf}
  For every affinoid perfectoid $U \in X_\mathrm{proet}^\mathrm{ap}$ together
  with a trivialization $\mathbb{L} \vert_U \cong \underline{\qp}^{\oplus n}$,
  the $\period{B}{dR}\p(U)$-module $\mathbb{M}_0(U)$ is finite projective and
  the composition
  \[
    c(U) \to \operatorname{Isom}(\mathbb{L}, \underline{\qp}^{\oplus n})
    \xrightarrow{\mathrm{DRT}(\mathbb{L})} \grgln
  \]
  recovers $\mathbb{M}_0(U) \subseteq \period{B}{dR}(U)^{\oplus n}$.
\end{proposition}

\begin{proof}
  Because $\mathbb{M}_0$, $\period{B}{dR}\p$, and $\gamma_\ast \grgln$ are all
  sheaves on $X_\mathrm{proet}^\mathrm{ap}$, see \cite[Corollary~17.1.9]{SW20},
  we may further pro-\'{e}tale localize $U$ and assume that $\mathbb{M}_0
  \vert_U$ is trivial as well. Recall from \Cref{Lem:PullbackOfProet} that
  $c(U)$ is naturally identified with $\gamma^{-1}(U)$. Under this
  identification, the map $c(U) \to \grgln$ of v-sheaves corresponds to the map
  \[
    U \to \operatorname{Isom}(\mathbb{L}, \underline{\qp}^{\oplus n}) \to
    \gamma_\ast(\grgln)
  \]
  in $\mathsf{Shv}(X_\mathrm{proet})$. This is given by the lattice
  $\mathbb{M}_0(U) \subseteq \period{B}{dR}(U)^{\oplus n}$ by construction.
\end{proof}

\begin{lemma} \label{Lem:deRhamExactTensor}
  \begin{enumerate}
    \item Let $0 \to \mathbb{L}_1 \to \mathbb{L}_2 \to \mathbb{L}_3 \to 0$ be a
      short exact sequence of $\underline{\qp}$-local systems on $X$. If
      $\mathbb{L}_2$ is de Rham, then both $\mathbb{L}_1$ and $\mathbb{L}_3$ are
      de Rham. Moreover, in this case the induced sequence $0 \to
      D_\mathrm{dR}(\mathbb{L}_1) \to D_\mathrm{dR}(\mathbb{L}_2) \to
      D_\mathrm{dR}(\mathbb{L}_3) \to 0$ of $\mathscr{O}_X$-modules is exact.
    \item Let $\mathbb{L}_1, \mathbb{L}_2$ to be two de Rham
      $\underline{\qp}$-local systems on $X$. Then $\mathbb{L}_1
      \otimes_{\underline{\qp}} \mathbb{L}_2$ is also de Rham and moreover we
      have an isomorphism of $\mathscr{O}_X$-modules
      \[
        D_\mathrm{dR}(\mathbb{L}_1 \otimes_{\underline{\qp}} \mathbb{L}_2) \cong
        D_\mathrm{dR}(\mathbb{L}_1) \otimes_{\mathscr{O}_X}
        D_\mathrm{dR}(\mathbb{L}_2)
      \]
      that respects the natural filtrations and integrable connections.
      Similarly, $\mathbb{L}_1^\vee$ is de Rham and there is an isomorphism
      $D_\mathrm{dR}(\mathbb{L}_1^\vee) \cong D_\mathrm{dR}(\mathbb{L}_1)^\vee$
      respecting the filtrations and connections.
  \end{enumerate}
\end{lemma}

\begin{proof}
  (1) For each classical point $x \in X(L)$, where $L/K$ is a finite extension,
  we have an induced short exact sequence
  \[
    0 \to \mathbb{L}_1 \vert_x \to \mathbb{L}_2 \vert_x \to \mathbb{L}_3 \vert_x
    \to 0
  \]
  of $\Gal(\bar{K}/L)$-representations. It follows that we have a left exact
  sequence
  \[
    0 \to D_\mathrm{dR}(\mathbb{L}_1 \vert_x) \to D_\mathrm{dR}(\mathbb{L}_2
    \vert_x) \to D_\mathrm{dR}(\mathbb{L}_3 \vert_x)
  \]
  of $L$-vector spaces. Note that $\dim_L D_\mathrm{dR}(\mathbb{L}_i \vert_x)
  \le \operatorname{rk} \mathbb{L}_i \vert_x$ for all $i$, with equality for $i
  = 2$. This implies that equality holds for all $1 \le i \le 3$, i.e., that all
  $\mathbb{L}_i \vert_x$ are de Rham representations, and moreover that the
  sequence is short exact. We now apply \Cref{Thm:LiuZhu} to conclude
  that both $\mathbb{L}_1$ and $\mathbb{L}_3$ are de Rham local systems.

  Let $\nu \colon \mathsf{Shv}(X_\mathrm{proet}) \to \mathsf{Shv}(X_\mathrm{et})$
  denote the natural morphism of topoi. Since $D_\mathrm{dR}(\mathbb{L}_i) =
  \nu_\ast (\mathbb{L}_i \otimes_{\underline{\qp}}
  \period\o{B}{dR})$, we have a left exact sequence
  \[
    0 \to D_\mathrm{dR}(\mathbb{L}_1) \to D_\mathrm{dR}(\mathbb{L}_2) \to
    D_\mathrm{dR}(\mathbb{L}_3)
  \]
  of locally free $\mathscr{O}_X$-modules. Because the map
  $D_\mathrm{dR}(\mathbb{L}_2) \to D_\mathrm{dR}(\mathbb{L}_3)$ is surjective
  over every classical point as noted above, see \cite[Theorem~3.9(ii)]{LZ17},
  we see that $D_\mathrm{dR}(\mathbb{L}_2) \to D_\mathrm{dR}(\mathbb{L}_3)$ is
  surjective as $\mathscr{O}_X$-modules.

  (2) For $i = 1, 2$, we have an isomorphism $\mathbb{L}_i
  \otimes_{\underline{\qp}} \period\o{B}{dR} \cong
  D_\mathrm{dR}(\mathbb{L}_i) \otimes_{\mathscr{O}_X}
  \period\o{B}{dR}$ preserving the filtrations and the
  integrable connections. Taking the tensor product of the two over
  $\period\o{B}{dR}$, we obtain an isomorphism
  \[
    (\mathbb{L}_1 \otimes_{\underline{\qp}} \mathbb{L}_2)
    \otimes_{\underline{\qp}} \period\o{B}{dR} \cong
    (D_\mathrm{dR}(\mathbb{L}_1) \otimes_{\mathscr{O}_X}
    D_\mathrm{dR}(\mathbb{L}_2)) \otimes_{\mathscr{O}_X}
    \period\o{B}{dR}
  \]
  of sheaves on $X_\mathrm{proet}$ that respects the filtrations and the
  integrable connections. This shows that $\mathbb{L}_1
  \otimes_{\underline{\qp}} \mathbb{L}_2$ is de Rham, and moreover provides the
  isomorphism $D_\mathrm{dR}(\mathbb{L}_1) \otimes_{\mathscr{O}_X}
  D_\mathrm{dR}(\mathbb{L}_2) \cong D_\mathrm{dR}(\mathbb{L}_1
  \otimes_{\underline{\qp}} \mathbb{L}_2)$ respecting the filtrations and the
  connections. The case of $\mathbb{L}_1^\vee$ can be proven similarly.
\end{proof}

\paragraph
We now choose a linear algebraic group $G / \qp$. Let $\mathbb{P}$ be a
$\underline{G(\qp)}$-torsor on $X$, where again we may use
\Cref{Prop:Torsors} to regard it as either a sheaf on
$X_\mathrm{proet}$ or a sheaf on $X_\mathrm{v}$. The construction $V \mapsto
\underline{G(\qp)} \backslash (\underline{V} \times \mathbb{P})$ defines an
exact tensor functor
\[
  \mathsf{Rep}_{\qp}(G) \to \mathsf{LocSys}_{\underline{\qp}}(X)
\]
from the category of algebraic $\qp$-representations of $G$ to the category of
$\underline{\qp}$-local systems on $X$.

\begin{definition}
  We say that a $\underline{G(\qp)}$-torsor $\mathbb{P}$ on $X$ is \textdef{de
  Rham} when for every algebraic representation $\rho \colon G \to
  \mathrm{GL}_{\qp}(V)$ (equivalently by \Cref{Lem:deRhamExactTensor}, for
  one faithful representation $\rho$) the induced $\underline{\qp}$-local system
  $\underline{G(\qp)} \backslash (\underline{V} \times \mathbb{P})$ on $X$ is de
  Rham.
\end{definition}

\paragraph \label{Par:TorsorToGr}
Let $\mathbb{P}$ be a de Rham $G(\qp)$-torsor on $X$. By
\Cref{Lem:deRhamExactTensor}, the functor
\[
  \mathsf{Rep}_{\qp}(G) \to \mathsf{Vect}(X); \quad V \mapsto
  D_\mathrm{dR}(\underline{G(\qp)} \backslash (\underline{V} \times \mathbb{P}))
\]
is an exact tensor functor. It follows that the functor
\[
  \mathsf{Rep}_{\qp}(G) \to
  \mathsf{LocSys}_{\period\o{B}{dR}\p}(X_\mathrm{proet}); \quad
  V \mapsto D_\mathrm{dR}(\underline{G(\qp)} \backslash (\underline{V} \times
  \mathbb{P})) \otimes_{\mathscr{O}_X} \period\o{B}{dR}\p
\]
is also an exact tensor functor. On the other hand, exactness of a sequence
$\mathbb{M}_{0,1} \to \mathbb{M}_{0,2} \to \mathbb{M}_{0,3}$ of
$\period{B}{dR}\p$-local systems on $X_\mathrm{proet}$ can be checked
after tensoring to $\period\o{B}{dR}\p$-local systems by
\cite[Proposition~6.10]{Sch13}, because $M \mapsto M[[X_1, \dotsc, X_n]]$
reflects exactness. We now apply \cite[Theorem~7.2]{Sch13} to conclude that
\[
  \mathsf{Rep}_{\qp}(G) \to
  \mathsf{LocSys}_{\period{B}{dR}\p}(X_\mathrm{proet}); \quad V \mapsto
  (D_\mathrm{dR}(\underline{G(\qp)} \backslash (\underline{V} \times
  \mathbb{P})) \otimes_{\mathscr{O}_X}
  \period\o{B}{dR}\p)^{\nabla=0}
\]
is an exact tensor functor. Using the Tannakian formalism, e.g.,
\cite[Theorem~19.5.2]{SW20}, we obtain a $\underline{G(\qp)}$-equivariant map
\[
  \mathrm{DRT}(\mathbb{P}) \colon \mathbb{P} \to \mathrm{Gr}_{G,\qp}
\]
in $\mathsf{Shv}(X_\mathrm{v}) \cong
\mathsf{Shv}(\mathsf{Perf}_\mathrm{v})_{/X^\lozenge}$.

\begin{proposition} \label{Prop:DRTFunctorial}
  Let $f \colon X \to Y$ be a map between smooth rigid analytic varieties over
  $K$, and let $\mathbb{P}$ be a de Rham $\underline{G(\qp)}$-torsor over $Y$.
  Then $f^\ast \mathbb{P} = \mathbb{P} \times_{Y^\lozenge} X^\lozenge$ is a de
  Rham $\underline{G(\qp)}$-torsor over $X$, and moreover the diagram
  \[ \begin{tikzcd}[row sep=small, column sep=huge]
    f^\ast \mathbb{P} \arrow{d} \arrow{r}{\mathrm{DRT}(f^\ast \mathbb{P})} &
    \mathrm{Gr}_{G,\qp} \arrow[equals]{d} \\ \mathbb{P}
    \arrow{r}{\mathrm{DRT}(\mathbb{P})} & \mathrm{Gr}_{G,\qp}
  \end{tikzcd} \]
  commutes.
\end{proposition}

\begin{proof}
  The fact that $f^\ast \mathbb{P}$ is de Rham follows directly from
  \Cref{Thm:LiuZhu}. For commutativity of the diagram, we reduce to the
  case of $G = \mathrm{GL}_n$, so that $\mathbb{P} \cong
  \operatorname{Isom}(\mathbb{L}, \underline{\qp}^{\oplus n})$ for some
  $\underline{\qp}$-local system $\mathbb{L}$ over $Y$. Note that $f$ induces a
  morphism of sites $f_\mathrm{proet} \colon X_\mathrm{proet} \to
  Y_\mathrm{proet}$ sending $\varprojlim U_i \in Y_\mathrm{proet}$ to
  $\varprojlim (X \times_Y U_i) \in X_\mathrm{proet}$, where we easily check the
  hypotheses of \cite[Proposition~00X6]{Stacks}. We see from
  \cite[Theorem~3.9(ii)]{LZ17} that there is a natural isomorphism $\alpha
  \colon f^\ast(D_\mathrm{dR}(\mathbb{L}), \nabla) \cong (D_\mathrm{dR}(f^\ast
  \mathbb{L}), \nabla)$ for which the diagram
  \[ \begin{tikzcd}[row sep=small]
    f^\ast \mathbb{L} \otimes_{\underline{\qp}} \period\o{B}{dR}
    \arrow{r}{\cong} \arrow[equals]{d} & f^\ast D_\mathrm{dR}(\mathbb{L})
    \otimes_{\mathscr{O}_X} \period\o{B}{dR} \arrow{d}{\alpha} \\ f^\ast
    \mathbb{L} \otimes_{\underline{\qp}} \period\o{B}{dR} \arrow{r}{\cong} &
    D_\mathrm{dR}(f^\ast \mathbb{L}) \otimes_{\mathscr{O}_X} \period\o{B}{dR}
  \end{tikzcd} \]
  of $\period\o{B}{dR}$-local systems on $X_\mathrm{proet}$ commutes, where the
  top row is obtained by applying $f_\mathrm{proet}^{-1}(-)
  \otimes_{f_\mathrm{proet}^{-1} \period\o{B}{dR}} \period\o{B}{dR}$ to the
  isomorphism $\mathbb{L} \otimes_{\underline{\qp}} \period\o{B}{dR} \cong
  D_\mathrm{dR}(\mathbb{L}) \otimes_{\mathscr{O}_Y} \period\o{B}{dR}$ on
  $Y_\mathrm{proet}$. Taking flat sections on all terms, we further obtain a
  diagram
  \[ \begin{tikzcd}[row sep=small]
    f^\ast \mathbb{L} \otimes_{\underline{\qp}} \period{B}{dR} \arrow{r}{\cong}
    \arrow[equals]{d} & (f^\ast D_\mathrm{dR}(\mathbb{L})
    \otimes_{\mathscr{O}_X} \period\o{B}{dR})^{\nabla=0} \arrow{d}{\alpha} \\
    f^\ast \mathbb{L} \otimes_{\underline{\qp}} \period{B}{dR} \arrow{r}{\cong}
    & (D_\mathrm{dR}(f^\ast \mathbb{L}) \otimes_{\mathscr{O}_X}
    \period\o{B}{dR})^{\nabla=0},
  \end{tikzcd} \]
  where the top horizontal map
  \begin{align*}
    f^\ast \mathbb{L} \otimes_{\underline{\qp}} \period{B}{dR} &\to (f^\ast
    D_\mathrm{dR}(\mathbb{L}) \otimes_{\mathscr{O}_X} \period\o{B}{dR})^{\nabla
    = 0} \\ &\cong f_\mathrm{proet}^{-1}((D_\mathrm{dR}(\mathbb{L})
    \otimes_{\mathscr{O}_Y} \period\o{B}{dR})^{\nabla = 0})
    \otimes_{f_\mathrm{proet}^{-1} \period{B}{dR}} \period{B}{dR}
  \end{align*}
  is obtained by applying $f_\mathrm{proet}^{-1}(-)
  \otimes_{f_\mathrm{proet}^{-1} \period{B}{dR}} \period{B}{dR}$ to $\mathbb{L}
  \otimes_{\underline{\qp}} \period{B}{dR} \cong (D_\mathrm{dR}(\mathbb{L})
  \otimes_{\mathscr{O}_Y} \period\o{B}{dR})^{\nabla=0}$. In particular, if we
  denote
  \[
    \mathbb{M}_{0,X} = (D_\mathrm{dR}(f^\ast \mathbb{L})
    \otimes_{\mathscr{O}_X} \period\o{B}{dR}\p)^{\nabla=0}, \quad
    \mathbb{M}_{0,Y} = (D_\mathrm{dR}(\mathbb{L}) \otimes_{\mathscr{O}_Y}
    \period\o{B}{dR}\p)^{\nabla=0}
  \]
  then $\alpha$ induces a natural identification
  \[ \label{Eq:DRTFunctorial}
    f_\mathrm{proet}^{-1} \mathbb{M}_{0,Y} \otimes_{f_\mathrm{proet}^{-1}
    \period{B}{dR}\p} \period{B}{dR}\p = \mathbb{M}_{0,X}, \tag{$\bowtie$}
  \]
  where both sides are naturally subsheaves of $f^\ast \mathbb{L}
  \otimes_{\underline{\qp}} \period{B}{dR}$.

  To prove the proposition, it is enough to show that for objects $U \in
  Y_\mathrm{proet}^\mathrm{ap}$ and $V \in X_\mathrm{proet}^\mathrm{ap}$, a map
  $V \to U \times_Y X$ in $X_\mathrm{proet}$, and a trivialization $\mathbb{L}
  \vert_U \cong \underline{\qp} \vert_U^{\oplus n}$, such that $\mathbb{M}_{0,Y}
  \vert_U$ is trivial, the diagram
  \[ \begin{tikzcd}[row sep=small, column sep=large]
    c(V) \arrow{r}{\mathrm{DRT}(f^\ast \mathbb{L})} \arrow{d} \arrow{r} & \grgln
    \arrow[equals]{d} \\ c(U) \arrow{r}{\mathrm{DRT}(\mathbb{L})} & \grgln
  \end{tikzcd} \]
  commutes. Because $\mathbb{M}_{0,Y} \vert_U$ is trivial we see from the
  identification \eqref{Eq:DRTFunctorial} that $\mathbb{M}_{0,X} \vert_{U
  \times_Y X}$ is also trivial and moreover
  \[
    \mathbb{M}_{0,Y}(U) \otimes_{\period{B}{dR}\p(U)} \period{B}{dR}\p(V) =
    \mathbb{M}_{0,X}(V) \subseteq \period{B}{dR}\p(V)^{\oplus n}.
  \]
  It now follows from \Cref{Prop:DRTOnAffPerf} that the diagram
  commutes.
\end{proof}

We end with a result that we shall use to construct the Hodge--Tate period map
on Shimura varieties.

\begin{lemma} \label{Lem:DRTCocharacter}
  Let $X/K$ be a connected smooth rigid analytic variety, let $G/\qp$ be a
  connected reductive group, and let $\mathbb{P}/X$ be a de Rham
  $\underline{G(\qp)}$-torsor. Let $\lbrace \mu \colon \mathbb{G}_{m,\qpbar} \to
  G_{\qpbar}\rbrace$ be a geometric conjugacy class of cocharacters with field
  of definition $E \subseteq \qpbar$, and fix an embedding $E \hookrightarrow K$
  so that there is an induced map
  \[
    \mathrm{DRT}(\mathbb{P}) \colon \mathbb{P} \to \grge.
  \]
  If the image intersects $\grgmu$ nontrivially, then it lies within $\grgmu$.
\end{lemma}

\begin{proof}
  \def\grglne{\mathrm{Gr}_{\mathrm{GL}_n,E}}
  \def\grgrhomu{\mathrm{Gr}_{\mathrm{GL}_n,\{\rho \circ \mu\},E}}
  \def\grgmup{\mathrm{Gr}_{G,\{\mu^\prime\},\qpbar}}
  \def\GeomPair{(C^\sharp, C^{\sharp+})}
  We first prove it in the case of $G = \mathrm{GL}_n$. Let $\mathbb{L}$ be the
  $\underline{\qp}$-local system corresponding to $\mathbb{P}$, so that as in
  \Cref{Par:LocSysToGr} the map $\mathrm{DRT}(\mathbb{L}) \colon
  \operatorname{Isom}(\mathbb{L}, \underline{\qp}^{\oplus n}) \to \gamma_\ast
  \grgln$ on $X_\mathrm{proet}$ is constructed via the lattice
  \[
    \mathbb{M}_0 = (D_\mathrm{dR}(\mathbb{L}) \otimes_{\mathscr{O}_X}
    \period\o{B}{dR}\p)^{\nabla=0} \subseteq \mathbb{L}
    \otimes_{\underline{\qp}} \period{B}{dR} = \mathbb{M}
    \otimes_{\period{B}{dR}\p} \period{B}{dR}.
  \]
  Since $X$ is connected, the filtrations $\mathrm{Fil}^i
  D_\mathrm{dR}(\mathbb{L})$ are locally free $\mathscr{O}_X$-modules with
  constant rank. On the other hand, writing $\ker \theta \subseteq
  \period{B}{dR}\p$ for the locally free ideal sheaf,
  \cite[Proposition~7.9]{Sch13} states that
  \[
    (\mathbb{M} \cap (\ker \theta)^i \mathbb{M}_0) / (\mathbb{M} \cap (\ker
    \theta)^{i+1} \mathbb{M}_0) \cong \mathrm{Fil}^{-i}
    D_\mathrm{dR}(\mathbb{L}) \otimes_{\mathscr{O}_X} \hat{\mathscr{O}}_X(i)
  \]
  is a locally free $\hat{\mathscr{O}}_X$-module of rank equal to that of
  $\mathrm{Fil}^{-i} D_\mathrm{dR}(\mathbb{L})$. It follows that the image of
  $\mathrm{DRT}(\mathbb{L})$ lies in a single stratum.

  We now consider the case of a general reductive group $G$. For each
  representation $\rho \colon G \to \mathrm{GL}_{n,\qp}$, the composition
  \[
    \mathbb{P} \xrightarrow{\mathrm{DRT}(\mathbb{P})} \grge \xrightarrow{\rho}
    \grglne
  \]
  factors through $\grgrhomu$. We wish to show that the image of
  every geometric point $\bar{x} \in \mathbb{P}\GeomPair$ under
  $\mathrm{DRT}(\mathbb{P})$ is contained in $\grgmu$. Fixing an embedding
  $\qpbar \hookrightarrow C^\sharp$, we have a decomposition
  \[
    \grge\GeomPair = \coprod_{\mu^\prime \in X_\ast(T)^+} \grgmup\GeomPair,
  \]
  where $T$ is the abstract maximal torus. Say
  $\mathrm{DRT}(\mathbb{P})(\bar{x})$ is in the stratum corresponding to
  $\mu^\prime \in X_\ast(T)^+$. Then for every representation $\rho$, the two
  cocharacters $\rho \circ \mu^\prime$ and $\rho \circ \mu$ of
  $\mathrm{GL}_{n,\qpbar}$ are conjugate, i.e., the weights of $\rho$ in
  $X^\ast(T)$ pair with $\mu^\prime$ and $\mu$ to the same multiset.

  We now claim that $\mu, \mu^\prime \in X_\ast(T)^+$ must be conjugate under
  the action of $\Gal(\qpbar/\qp)$. Let $E_0 / \qp$ be a sufficiently large
  finite Galois extension that splits $G$. For each $\nu \in X^\ast(T)^+$, there
  exists a highest weight representation $\rho_\nu \colon G_{E_0} \to
  \mathrm{GL}_{d_\nu,E_0}$, which induces the Weil restricted representation
  \[
    \operatorname{Res}_{E_0/\qp} \rho_\nu \colon G \to \mathrm{GL}_{d_\nu
    [E_0:\qp], \qp}.
  \]
  The weights of this representation are given as a disjoint union
  \[
    w(\operatorname{Res}_{E_0/\qp} \rho_\nu) = \coprod_{\sigma \in
    \Gal(E_0/\qp)} \sigma \cdot w(\rho_\nu).
  \]
  Paring this with $\mu, \mu^\prime$ and taking the maximal element, we obtain
  \[
    \max_{\sigma \in \Gal(E_0/\qp)} \langle \mu, \sigma \nu \rangle =
    \max_{\sigma \in \Gal(E_0/\qp)} \langle \mu^\prime, \sigma \nu \rangle
  \]
  for all $\nu \in X^\ast(T)^+$, hence for all $\nu \in X_\ast(T)_\mathbb{R}^+$.
  This implies that convex hulls of $\Gal(E_0/\qp) \mu + \mathbb{R}_{\ge 0}
  \Phi^-$ and $\Gal(E_0/\qp) \mu^\prime + \mathbb{R}_{\ge 0} \Phi^-$ agree, and
  by considering the vertices we deduce that $\mu, \mu^\prime$ are in the same
  $\Gal(E_0/\qp)$-orbit.

  This shows that we have a $\underline{G(\qp)} \times
  \underline{\Gal(\qpbar/E)}$-equivariant map
  \[
    \mathrm{DRT}(\mathbb{P})_{\qpbar} \colon \mathbb{P} \times_{\Spd E} \Spd
    \qpbar \to \coprod_{\sigma \in \Gal(\qpbar/\qp) / \Gal(\qpbar/E)}
    \mathrm{Gr}_{G,\{\sigma \mu\},\qpbar} \subseteq \mathrm{Gr}_{G,\qpbar}.
  \]
  Because the Galois conjugates $\{\sigma \mu\}$ are mutually incomparable under
  the partial ordering, each stratum is open and closed in the union. Moreover,
  the stratum $\mathrm{Gr}_{G,\{\mu\},\qpbar}$ is stable under the
  $\underline{G(\qp)} \times \underline{\Gal(\qpbar/E)}$-action. Because
  $X^\lozenge = (\underline{G(\qp)} \times \underline{\Gal(\qpbar/E)})
  \backslash (\mathbb{P} \times_{\Spd E} \Spd \qpbar)$ is assumed to be
  connected and the image of $\mathrm{DRT}(\mathbb{P})_{\qpbar}$ intersects the
  stratum $\mathrm{Gr}_{G,\{\mu\},\qpbar}$, we conclude that the image
  completely lies in $\mathrm{Gr}_{G,\{\mu\},\qpbar}$. Therefore the image of
  $\mathrm{DRT}(\mathbb{P})$ lies in $\grgmu$.
\end{proof}

}
{\section{The Hodge--Tate period map} \label{Sec:HodgeTate}
\def\flgmu{\mathrm{Fl}_{G,\{\mu^{-1}\},E}^\lozenge}
\def\flgcmu{\mathrm{Fl}_{G^\mathrm{c},\{(\mu^\mathrm{c})^{-1}\},E}^\lozenge}

In this section, we apply the results of \Cref{Sec:DeRham} to construct
the Hodge--Tate period map on Shimura varieties
\[
  \pi_\mathrm{HT} \colon \sh\d \to \grg\m,
\]
following \cite[Section~4.1]{PR24} and \cite[Section~5.4]{Han16p}. To deal with
groups $\mathsf{G}$ whose center is not cuspidal, we first consider the
Hodge--Tate period map associated to the $\underline{G^\mathrm{c}(\qp)}$-torsor,
and then lift the map $\pi_\mathrm{HT}^\mathrm{c} \colon \sh\d \to \grg\c\m$ to
$\pi_\mathrm{HT} \colon \sh\d \to \grg\m$.

\paragraph
Let $\gx$ be a Shimura datum with reflex field $\mathsf{E} \subseteq
\mathbb{C}$. Choose a finite place $v \mid p$ of $\mathsf{E}$ and write $E =
\mathsf{E}_v$ and $G = \mathsf{G}_{\qp}$. For each neat compact open subgroup $K
\subseteq \mathsf{G}(\mathbb{A}_\mathbb{Q}^\infty)$, there is a corresponding
Shimura variety $\sh[K][\mathsf{E}]$ over the reflex field $\mathsf{E}$, and we
can base change it to $E$ to obtain $\sh[K]$. For open compact subgroups
$K^\prime \subseteq K$, there are finite \'{e}tale transition maps
$\sh[K^\prime] \to \sh[K]$ that collectively define a pro-\'{e}tale left torsor
\[
  \sh\d = \varprojlim_{K^\prime \to 1} \sh\d[K^\prime] \to \sh\d[K]
\]
for the profinite group $K / (K \cap \mathsf{Z}(\mathbb{Q})^-)$, where
$\mathsf{Z}$ is the center of $\mathsf{G}$ and $\mathsf{Z}(\mathbb{Q})^-$ is the
closure of $\mathsf{Z}(\mathbb{Q})$ in
$\mathsf{G}(\mathbb{A}_\mathbb{Q}^\infty)$.

\begin{remark}
  When $\gx$ is of pre-abelian type, it is known that the limit
  $\varprojlim_{K^\prime \to \{1\}} \sh[K^\prime]$ exists as a perfectoid space
  by \cite{Sch15}, \cite{She17}, \cite{HJ23}. But regardless, we may still
  define the corresponding v-sheaf $\sh\d$ as an inverse limit of v-sheaves.
\end{remark}

\paragraph \label{Par:LocalSystem}
As in \cite[Section~1.5.8]{KSZ21p}, let $\mathsf{Z}_\mathrm{ac} \subseteq
\mathsf{Z}^0$ be the minimal $\mathbb{Q}$-subtorus with the property that the
$\mathbb{Q}$-split rank of $\mathsf{Z}^0 / \mathsf{Z}_\mathrm{ac}$ agrees with
its $\mathbb{R}$-split rank. We define the \textdef{cuspidal part} of
$\mathsf{G}$ as
\[
  \mathsf{G}^\mathrm{c} = \mathsf{G} / \mathsf{Z}_\mathrm{ac}, \quad
  G^\mathrm{c} = (\mathsf{G}^\mathrm{c})_{\qp},
\]
with a caveat that $G^\mathrm{c}$ depends on the global model $\mathsf{G}$
rather than only $G = \mathsf{G}_{\qp}$. By \cite[Lemma~1.5.7]{KSZ21p}, we have
$K \cap \mathsf{Z}(\mathbb{Q}) \subseteq \mathsf{Z}_\mathrm{ac}(\mathbb{Q})$ and
hence $K \cap \mathsf{Z}(\mathbb{Q})^- \subseteq \mathsf{Z}_\mathrm{ac}(\af)$.
It follows that the quotient $\mathsf{G} \to \mathsf{G}^\mathrm{c}$ induces a
group homomorphism
\[
  K / (K \cap \mathsf{Z}(\mathbb{Q})^-) \to
  \mathsf{G}^\mathrm{c}(\mathbb{A}_\mathbb{Q}^\infty).
\]
This allows us to push out the pro-\'{e}tale $\underline{K / (K \cap
\mathsf{Z}(\mathbb{Q})^-)}$-torsor $\sh\d \to \sh\d[K]$ to a pro-\'{e}tale
$\underline{\mathsf{G}^\mathrm{c}(\mathbb{A}_\mathbb{Q}^\infty)}$-torsor, and
hence to a pro-\'{e}tale $\underline{G^\mathrm{c}(\qp)}$-torsor
\[
  \xi_K = \underline{G^\mathrm{c}(\qp)} \times^{\underline{K / (K \cap
  \mathsf{Z}(\mathbb{Q})^-)}} \sh\d \to \sh\d[K]
\]
over $\sh\ad[K]$.

\begin{remark}
  When $\gx$ is of Hodge type, we have $\mathsf{Z}_\mathrm{ac} = 1$ and hence
  $\mathsf{G} = \mathsf{G}^\mathrm{c}$ by \cite[Lemma~5.1.2]{KSZ21p}.
\end{remark}

\begin{theorem}[{\cite[Corollary~4.9]{LZ17}}]
  The pro-\'{e}tale $\underline{G^\mathrm{c}(\qp)}$-torsor $\xi_K$ on the smooth
  rigid analytic Shimura variety $\sh\ad[K]$ is de Rham.
\end{theorem}

\paragraph \label{Par:HodgeTateC}
As in \Cref{Par:TorsorToGr}, we obtain a
$\underline{G^\mathrm{c}(\qp)}$-equivariant morphism
\[
  \mathrm{DRT}(\xi_K) \colon \xi_K \to \grg\c
\]
of v-sheaves. Composing with the natural morphism
\[
  \iota_K \colon \sh\d \to \underline{G^\mathrm{c}(\qp)} \times^{\underline{K /
  (K \cap \mathsf{Z}(\mathbb{Q})^-)}} \sh\d = \xi_K; \quad x \mapsto (1, x),
\]
we obtain a morphism
\[
  \pi_{\mathrm{HT},K}^\mathrm{c} \colon \sh\d \to \grg\c
\]
of v-sheaves.

\begin{proposition}
  The map $\pi_\mathrm{HT}^\mathrm{c} = \pi_{\mathrm{HT},K}^\mathrm{c}$ is
  independent of the choice of $K$, and moreover is
  $\underline{\mathsf{G}(\mathbb{A}_\mathbb{Q}^\infty)}$-equivariant, where the
  action on $\grg\c$ factors through $\mathsf{G}(\mathbb{A}_\mathbb{Q}^\infty)
  \twoheadrightarrow G(\qp)$.
\end{proposition}

This is a straightforward consequence of the functoriality of $\mathrm{DRT}$,
see \Cref{Prop:DRTFunctorial}.

\begin{proof}
  Let $K^\prime \subseteq K$ be an open compact subgroup, so that there is an
  induced finite \'{e}tale map $\sh[K^\prime] \to \sh[K]$. By construction, the
  $\underline{G^\mathrm{c}(\qp)}$-torsor $\xi_{K^\prime}$ is the pullback of
  $\xi_K$, and hence the diagram
  \[ \begin{tikzcd}[row sep=-0.2em]
    & \xi_{K^\prime} = \underline{G^\mathrm{c}(\qp)} \times^{\underline{K^\prime
    / (K^\prime \cap \mathsf{Z}(\mathbb{Q})^-)}} \sh\d \arrow{r} \arrow{dd} &
    \grg\c \arrow[equals]{dd} \\ \sh\d \arrow{ru}{\iota_{K^\prime}}
    \arrow{rd}[']{\iota_K} \\ & \xi_K = \underline{G^\mathrm{c}(\qp)}
    \times^{\underline{K / (K \cap \mathsf{Z}(\mathbb{Q})^-)}} \sh\d \arrow{r} &
    \grg\c
  \end{tikzcd} \]
  commutes by \Cref{Prop:DRTFunctorial}. This shows that
  $\pi_{\mathrm{HT},K^\prime}^\mathrm{c} = \pi_{\mathrm{HT},K}^\mathrm{c}$, and
  hence $\pi_\mathrm{HT}^\mathrm{c}$ is independent of the choice of $K$.

  Let $g \in \mathsf{G}(\mathbb{A}_\mathbb{Q}^\infty)$ be an arbitrary element.
  We let $K^\prime = g K g^{-1}$ and consider the isomorphism
  \[
    g \colon \sh[K] \xrightarrow{\cong} \sh[K^\prime].
  \]
  Then there is an isomorphism of $\underline{G^\mathrm{c}(\qp)}$-torsors
  \[
    g^\ast \xi_{K^\prime} = \xi_{K^\prime} \times_{\sh\d[K^\prime]} \sh\d[K]
    \xrightarrow{\cong} \xi_K; \quad ((h, gx), x) \mapsto (hg, x)
  \]
  on $\sh[K]$. It follows that we have a commutative diagram
  \[ \begin{tikzcd}[row sep=small]
    \sh\d \arrow{r}{\iota_K} \arrow{ddd}{g} & \xi_K \arrow{d}{g}
    \arrow{r}{\mathrm{DRT}} & \grg\c \arrow{d}{g} \\ & \xi_K
    \arrow{r}{\mathrm{DRT}} & \grg\c \arrow[equals]{d} \\ & g^\ast
    \xi_{K^\prime} \arrow{u}{\cong} \arrow{d} \arrow{r}{\mathrm{DRT}} & \grg\c
    \arrow[equals]{d} \\ \sh\d \arrow{r}{\iota_{K^\prime}} & \xi_{K^\prime}
    \arrow{r}{\mathrm{DRT}} & \grg\c,
  \end{tikzcd} \]
  where the lower right square commutes by \Cref{Prop:DRTFunctorial}
  and the upper right square commutes by $G^\mathrm{c}(\qp)$-equivariance of
  $\mathrm{DRT}(\xi_K)$. As the top row and the bottom row both compose to
  $\pi_{\mathrm{HT},K^\prime}^\mathrm{c} = \pi_{\mathrm{HT},K}^\mathrm{c}$, we
  conclude that $\pi_\mathrm{HT}^\mathrm{c}$ is
  $\underline{\mathsf{G}(\mathbb{A}_\mathbb{Q}^\infty)}$-equivariant.
\end{proof}

\paragraph
For a point $x \in \mathsf{X}$, we denote by
\[
  \mu_{x,\mathbb{C}} \colon \mathbb{G}_{m,\mathbb{C}} \xrightarrow{z \mapsto (z
  \otimes 1 - iz \otimes i)/2} (\operatorname{Res}_{\mathbb{C}/\mathbb{R}}
  \mathbb{G}_{m,\mathbb{C}})_\mathbb{C} \xrightarrow{x} \mathsf{G}_\mathbb{C}
\]
the induced Hodge cocharacter. Since they are $\mathsf{G}(\mathbb{R})$-conjugate
to each other, we obtain a conjugacy class of cocharacters $\{\mu_\mathbb{C}\}$
over $\mathbb{C}$, which has field of definition $\mathsf{E} \subseteq
\mathbb{C}$. The choice of the place $p \mid v$ induces an embedding $\mathsf{E}
\hookrightarrow E = \mathsf{E}_v$, and hence upon fixing an embedding $E
\hookrightarrow \qpbar$ we obtain a conjugacy class of geometric cocharacters
\[
  \{ \mu \colon \mathbb{G}_{m,\qpbar} \to G_{\qpbar} \}
\]
with field of definition $E$.

\begin{lemma} \label{Lem:GaloisRepCochar}
  Let $G/\qp$ be a connected reductive group, let $E/\qp$ be a finite extension,
  let $r_\mu \colon T = \operatorname{Res}_{E/\qp} \mathbb{G}_{m,E} \to G$ be
  a homomorphism, and let $\mu \colon \mathbb{G}_{m,E} \to G_E$ be the image of
  the cocharacter $(1, 0, \dotsc, 0) \in X_\ast(T)$ under $r_\mu$. Write
  $\mathrm{Art}_E \colon E^\times \hookrightarrow \Gal(E^\mathrm{ab}/E)$ for the
  Artin map, normalized so that a uniformizer maps to a lift of the geometric
  Frobenius. If a continuous Galois representation
  \[
    \rho \colon \Gal(E^\mathrm{ab}/E) \to G(\qp)
  \]
  satisfies the property that $\rho \circ \mathrm{Art}_E \colon E^\times \to
  G(\qp)$ agrees with $r_\mu$ on an open neighborhood of $1 \in E^\times$, then
  $\rho$ is de Rham and moreover the corresponding map
  \[
    \mathrm{DRT}(\rho) \colon \mathbb{P} = [(\underline{G(\qp)} \times \Spd
    E^\mathrm{ab}) / \underline{\Gal(E^\mathrm{ab}/E)}] \to \mathrm{Gr}_{G,E}
  \]
  has image in $\mathrm{Gr}_{G,\{\mu^{-1}\},E}$.
\end{lemma}

\begin{proof}
  Because $\period{B}{dR}(\cp, \mathcal{O}_{\cp})^{I_E} = \breve{E}$ where $I_E
  = \Gal(\qpbar/\breve{E})$ is the inertia group, the representation $\rho$ is
  de Rham if and only if $\rho \vert_{I_E}$ is, and moreover for each
  representation $G \to \mathrm{GL}(V)$ the corresponding $\mathbb{M}_0 =
  D_\mathrm{dR}(\mathbb{L}) \otimes \period{B}{dR}\p(\cp, \mathcal{O}_{\cp})$ is
  the same for $\rho$ and $\rho \vert_{I_E}$. This shows that we are free to
  replace $\rho$ with a composition $\Gal(E^\mathrm{ab}/E) \twoheadrightarrow
  \Gal(E^\mathrm{ab}/\breve{E}) \xrightarrow{\rho} G(\qp)$, upon choosing a
  uniformizer $\pi \in E$ inducing a splitting $\Gal(E^\mathrm{ab}/E) \cong
  \Gal(E^\mathrm{ab}/\breve{E}) \oplus \hat{\mathbb{Z}}$. Moreover, when we
  replace $E$ by a finite extension, both $r_\mu$ and $\mathrm{Art}_E$ changes
  by composition with the norm, and hence the property that $\rho \circ
  \mathrm{Art}_E$ agrees with $r_\mu$ does not change. Therefore we may as well
  assume that $\rho$ factors as
  \[
    \Gal(E^\mathrm{ab}/E) \twoheadrightarrow \Gal(E^\mathrm{ab}/\breve{E}) \cong
    \mathcal{O}_E^\times \hookrightarrow E^\times =T(\qp) \xrightarrow{r_\mu}
    G(\qp).
  \]
  Because $\mathrm{Gr}_{T,\{\mu^{-1}\},E} \to \grg\m$, we are reduce to the case
  when $G = T$ and $r_\mu = \id$. At this point, we observe that this is
  precisely the $p$-adic Tate module of the Lubin--Tate group for $\pi$, which
  is crystalline with Hodge--Tate weights $(-1, 0, \dotsc, 0)$, see
  \cite[Proposition~B.3]{Con11p} (where the normalization is such that the
  cyclotomic character has weight $-1$).
\end{proof}

\begin{proposition}
  The map $\pi_\mathrm{HT}^\mathrm{c} \colon \sh\d \to \grg\c[E]$ factors
  through
  \[
    \grg\c\m \subseteq \grg\c[E],
  \]
  where $\lbrace \mu^\mathrm{c} \rbrace$ is the projection of $\lbrace \mu
  \rbrace$ under $G \to G^\mathrm{c}$.
\end{proposition}

\begin{proof}
  In view of the construction in \Cref{Par:HodgeTateC}, we may choose a
  neat compact open $K$ and show that the morphism $\mathrm{DRT}(\xi_K) \colon
  \xi_K \to \grg\c[E]$ factors through $\grg\c\m$. By
  \Cref{Lem:DRTCocharacter}, it suffices to show that on each connected
  component of $\sh[K]$, there exists a point over which $\mathrm{DRT}(\xi_K)$
  maps to $\grg\c\m$. Let $i \colon (\mathsf{T}, \{x\}) \to \gx$ be a special
  point with reflex field $\mathsf{F}$. Then the points $[x, a]_K \in \sh[K]$
  for $a \in \mathsf{G}(\af)$ are Zariski dense, hence intersect every connected
  component of $\sh[K]$. Hence we focus on the such points.

  Let $F_0 \subseteq F^\mathrm{ab}$ be the field of definition of $[x, a]_K \in
  \sh[K]$, containing $F = \mathsf{F}_w$ for some place $w \mid v$. The
  cocharacter $\mu_x \colon \mathbb{G}_{m,F} \to T_F$ is defined over $F$, and
  hence induces a norm $r_{\mu,x} \colon \operatorname{Res}_{F/\qp}
  \mathbb{G}_{m,F} \to T$. Let $\lambda_\mathsf{T}$ and $\lambda_\mathsf{G}$ be
  the unique continuous extensions
  \[ \begin{tikzcd}
    F^\times \arrow[hook]{r}{r_{\mu,x}} \arrow[hook]{d}{\mathrm{Art}_F} & T(\qp)
    \arrow[hook]{r}{i} & G(\qp) \arrow{d} \\ \Gal(F^\mathrm{ab}/F)
    \arrow[dashed]{rr}{\lambda} & & \mathsf{Z}(\mathbb{Q})^-
    \backslash \mathsf{G}(\af)
  \end{tikzcd} \]
  so that the definition of a canonical model, e.g.,
  \cite[Definition~3.5]{Del71}, states that $\sigma [x,a] = [x, \lambda(\sigma)
  a]$. It follows that the Galois representation corresponding to $\xi_K
  \vert_{[x,a]_K}$ is $a^{-1}$-conjugate to the composition
  \[
    \rho_x^\mathrm{c} \colon \Gal(F^\mathrm{ab}/F_0) \xrightarrow{\lambda}
    \mathsf{Z}(\mathbb{Q})^- \backslash \mathsf{Z}(\mathbb{Q})^- K \cong (K \cap
    \mathsf{Z}(\mathbb{Q})^-) \backslash K \to G^\mathrm{c}(\qp).
  \]
  On the open subgroup $(i \circ r_{\mu,x})^{-1}(K_p) \subseteq F^\times$, the
  composition $\rho_x^\mathrm{c} \circ \mathrm{Art}_F$ agrees with the
  composition
  \[
    F^\times \xrightarrow{r_{\mu,x}} T(\qp) \hookrightarrow G(\qp) \to
    G^\mathrm{c}(\qp).
  \]
  It now follows from \Cref{Lem:GaloisRepCochar} that the image of
  $\mathrm{DRT}(\xi_K \vert_{[x,a]_K})$ lies in the stratum $\grg\c\m[F]
  \subseteq \grg\c[F]$.
\end{proof}

\paragraph \label{Par:HodgeTatePeriod}
Using the Bia\l{}ynicki-Birula isomorphism, \Cref{Thm:BialynickiBirula},
we obtain an isomorphism
\[
  \grg\m \cong \flgmu \xrightarrow{\cong} \flgcmu \cong \grg\c\m
\]
of v-sheaves over $\Spd E$. We define the \textdef{Hodge--Tate period map}
$\pi_\mathrm{HT}$ to be the composition
\[
  \pi_\mathrm{HT} \colon \sh\d \to \grg\c\m \cong \grg\m.
\]

\begin{proposition}
  Let $\gx \to \gx'$ be a morphism of Shimura data with local reflex fields
  $E^\prime \subseteq E$. Then the induced diagram
  \[ \begin{tikzcd}
    \sh\d \arrow{r} \arrow{d}{\pi_\mathrm{HT}} & \sh\d'
    \arrow{d}{\pi_\mathrm{HT}} \\ \grg\m \arrow{r} & \grg\m'
  \end{tikzcd} \]
  commutes.
\end{proposition}

\begin{proof}
  By definition, it suffices to show that
  \[ \begin{tikzcd}
    \sh\d \arrow{r} \arrow{d}{\pi_\mathrm{HT}^\mathrm{c}} & \sh\d'
    \arrow{d}{\pi_\mathrm{HT}^\mathrm{c}} \\ \grg\c\m \arrow{r} & \grg\c\m'
  \end{tikzcd} \]
  commutes, where we note that there is indeed a map $G^\mathrm{c} \to
  G^{\prime\mathrm{c}}$ by \cite[Lemma~3.8]{IKY23p}. This follows from
  \Cref{Prop:DRTFunctorial} together with the fact that the pushout
  of the $G^\mathrm{c}(\qp)$-torsor on $\sh[K]$ along $G^\mathrm{c}(\qp) \to
  G^{\prime\mathrm{c}}(\qp)$ agrees with the pullback of the
  $G^{\prime\mathrm{c}}(\qp)$-torsor on $\sh'[K^\prime]$ along the map $\sh[K]
  \to \sh'[K^\prime]$.
\end{proof}

}
{\section{Igusa stacks and global uniformizations} \label{Sec:Uniformization}
\def\rrrll{\arrow{r} \arrow[shift left=2]{r} \arrow[shift right=2]{r}%
  \arrow[leftarrow, shift left]{r} \arrow[leftarrow, shift right]{r}}
\def\rrl{\arrow[shift left]{r} \arrow[shift right]{r} \arrow[leftarrow]{r}}

In this section, we give an axiomatic definition of an Igusa stack, and
reinterpret the data of an Igusa stack as a certain uniformization of the
Shimura variety. More precisely, we establish an equivalence of categories
between Igusa stacks and global uniformization maps, which are maps of v-sheaves
\[
  \Theta \colon U \times_{\bung\m} \grg\m \to U
\]
satisfying certain conditions. While the entire section consists mostly of
formalisms, one needs to be careful as we are dealing with v-stacks rather than
v-sheaves.

Let $\gx$ be a Shimura datum with reflex field $\mathsf{E}$, fix a place $v \mid
p$, and write $E = \mathsf{E}_v$ and $G = \mathsf{G}_{\qp}$. Recall from
\Cref{Par:HodgeTatePeriod} the $\underline{\mathsf{G}(\af)}$-equivariant
Hodge--Tate period map
\[
  \pi_\mathrm{HT} \colon \sh\d \to \grg\m,
\]
where $\lbrace \mu \colon \mathbb{G}_{m,\qpbar} \to G_{\qpbar} \rbrace$ is the
geometric conjugacy class of the Hodge cocharacter whose field of definition is
$E \subset \qpbar$.

\paragraph \label{Par:GroupAction}
Let $\mathcal{C}$ be a site, and consider the category
$\mathsf{Shv}(\mathcal{C})$ of sheaves on $\mathcal{C}$ as well as the
$(2,1)$-category $\mathsf{Stk}(\mathcal{C})$ of stacks on $\mathcal{C}$. For a
group object $G \in \mathsf{Shv}(\mathcal{C})$ and a stack $X \in
\mathsf{Stk}(\mathcal{C})$, by a $G$-action on $X$ we mean the data of a
morphism $\alpha \colon G \times X \to X$ together with an isomorphism $\eta$
filling in the diagram
\[ \begin{tikzcd}
  G \times G \times X \arrow{d}{m \times \id} \arrow{r}{\id \times \alpha} & G
  \times X \arrow{d}{\alpha} \\ G \times X \arrow{r}{\alpha} & X,
\end{tikzcd} \]
satisfying the property that the diagram
\[ \begin{tikzcd}[row sep=small, column sep=large]
  G \times G \times G \times X \arrow{rr}{\id \times \id \times \alpha}
  \arrow{dr}{\id \times m \times \id} \arrow{dd}{m \times \id \times \id} & & G
  \times G \times X \arrow{dr}{\id \times \alpha} \arrow{dd}[near start]{m
  \times \id} \\ & G \times G \times X \arrow{rr}[near start]{\id \times \alpha}
  \arrow{dd}[near start]{m \times \id} & & G \times X \arrow{dd}{\alpha} \\ G
  \times G \times X \arrow{rr}[near start]{\id \times \alpha} \arrow{dr}{m
  \times \id} & & G \times X \arrow{dr}{\alpha} \\ & G \times X
  \arrow{rr}{\alpha} & & X
\end{tikzcd} \]
2-commutes, where the top, right, front, bottom faces are $\id \times \eta$,
$\eta$, $\eta$, $\eta$.

\subsection{} \label{Par:EquivariantMorphism}
For two stacks $X, Y \in \mathsf{Stk}(\mathcal{C})$ with $G$-actions, a
$G$-equivariant morphism $X \to Y$ is the data of a map $f \colon X \to Y$
together with an isomorphism $\zeta$ filling in
\[ \begin{tikzcd}
  G \times X \arrow{r}{\id \times f} \arrow{d}{\alpha_X} & G \times Y
  \arrow{d}{\alpha_Y} \\ X \arrow{r}{f} & Y
\end{tikzcd} \]
satisfying the property that the diagram
\[ \begin{tikzcd}[row sep=small, column sep=large]
  G \times G \times X \arrow{rd}{\id \times \alpha_Y} \arrow{rr}{\id \times \id
  \times f} \arrow{dd}{m \times \id} & & G \times G \times Y \arrow{dr}{\id
  \times \alpha_Y} \arrow{dd}[near start]{m \times \id} \\ & G \times X
  \arrow{rr}[near start]{\id \times f} \arrow{dd}[near start]{\alpha_X} & & G
  \times Y \arrow{dd}{\alpha_Y} \\ G \times X \arrow{rr}[near start]{\id \times
  f} \arrow{dr}{\alpha} & & G \times Y \arrow{dr}{\alpha_Y} \\ & X \arrow{rr}{f}
  & & Y
\end{tikzcd} \]
2-commutes, where the six faces of the cube are given by $\eta_X$, $\eta_Y$,
$\zeta$, $\zeta$, $\id \times \zeta$, $\id_m \times \id_f$. The forgetful
functor
\[
  \Hom_G(X, Y) \to \Hom(X, Y)
\]
that forgets the $G$-equivariant structure is faithful, i.e., is a $0$-truncated
map of groupoids, where an isomorphism $\xi$ between two $G$-equivariant maps
$f, g \colon X \to Y$ is $G$-equivariant when $\xi \circ \zeta_f$ agrees with
$\zeta_g \circ (\id_G \times \xi)$.

\paragraph \label{Par:Simplicial}
Let $I \in \mathsf{Stk}(\mathcal{C})$, $U \in \mathsf{Shv}(\mathcal{C})$, and
let $f \colon U \to I$ be a morphism. Consider the associated 2-truncated
\v{C}ech nerve
\[ \begin{tikzcd}
  X_2 = U \times_I U \times_I U \rrrll & X_1 = U \times_I U \rrl & X_0 = U
  \arrow{r} & X_{-1} = I,
\end{tikzcd} \]
where we note that $X_0, X_1, X_2$ are $0$-truncated, i.e., are sheaves. This
satisfies the property that the commutative diagrams
\[ \begin{tikzcd}
  X_2 \arrow{r}{d_{2,0}} \arrow{d}{d_{2,2}} & X_1 \arrow{d}{d_{1,1}} \\ X_1
  \arrow{r}{d_{1,0}} & X_0,
\end{tikzcd} \begin{tikzcd}
  X_2 \arrow{r}{d_{2,1}} \arrow{d}{d_{2,2}} & X_1 \arrow{d}{d_{1,1}} \\ X_1
  \arrow{r}{d_{1,1}} & X_0,
\end{tikzcd} \begin{tikzcd}
  X_2 \arrow{r}{d_{2,1}} \arrow{d}{d_{2,0}} & X_1 \arrow{d}{d_{1,0}} \\ X_1
  \arrow{r}{d_{1,0}} & X_0,
\end{tikzcd} \]
are all Cartesian, and also the associativity property that
\[ \begin{tikzcd}
  X_2 \times_{d_{1,0} \circ d_{2,0},X_0,d_{1,1}} X_1 \arrow{d}{d_{2,1} \times
  \id} & X_2 \times_{d_{2,0},X_1,d_{2,2}} X_2 \arrow{l}{\cong}[']{\id \times
  d_{2,0}} \arrow{r}{d_{2,2} \times \id}[']{\cong} & X_1
  \times_{d_{1,0},X_0,d_{1,1} \circ d_{2,2}} X_2 \arrow{d}{\id \times d_{2,1}}
  \\ X_1 \times_{d_{1,0},X_0,d_{1,1}} X_1 \cong X_2 \arrow{r}{d_{2,1}} & X_1 &
  X_2 \cong X_1 \times_{d_{1,0},X_0,d_{1,1}} X_1 \arrow{l}[']{d_{2,1}}
\end{tikzcd} \]
commutes. Conversely, for any 2-truncated simplicial set satisfying these
properties, the maps $X_1 \rightrightarrows X_0$ together with the composition
rule
\[
  X_1 \times_{d_{1,0},X_0,d_{1,1}} X_1 \cong X_2 \xrightarrow{d_{2,1}} X_1
\]
defines a groupoid object, so that we may take its quotient $X_0 \to \lvert
X_\bullet \rvert = [X_0 / X_1]$. This constructions recovers the essential image
$\im f \subseteq I$.

\paragraph \label{Par:AbsoluteFrob}
For every perfectoid $\fp$-algebra $(R, R^+)$, there is an absolute Frobenius
automorphism $\phi_{(R, R^+)} \colon \Spa(R, R^+) \to \Spa(R, R^+)$. It follows
that for every v-stack $X \in \mathsf{Stk}(\mathsf{Perf}_\mathrm{v})$ there is
an induced morphism of v-stacks
\[
  \phi_X \colon X \to X; \quad x \in X(R, R^+) \mapsto \phi_{(R, R^+)}^\ast x
  \in X(R, R^+).
\]
For a morphism $f \colon X \to Y$ of v-stacks, there is a canonical isomorphism
$f \circ \phi_X \cong \phi_Y \circ f$ coming from the naturality of $f$ for the
morphism $\phi_{(R, R^+)}$. We also note that there is a canonical isomorphism
\[
  \phi_{\bung} \cong \id_{\bung} \colon \bung \to \bung
\]
for every linear algebraic group $G/\qp$, coming from the fact that the map of
adic spaces $\mathcal{X}_\mathrm{FF}(R, R^+) \to \mathcal{X}_\mathrm{FF}(R,
R^+)$ induced by $\phi_{(R, R^+)}$ is the identity map.

\begin{definition} \label{Def:Igusa}
  Let $U \subseteq \sh\d$ be an open subsheaf stable under the action of
  $\underline{\mathsf{G}(\af)}$. An \textdef{Igusa stack} for $U$ is the data
  of
  \begin{itemize}
    \item a v-stack $\igsu$ together with a
      $\underline{\mathsf{G}(\afp)}$-action,
    \item a $\underline{\mathsf{G}(\afp)}$-equivariant map of v-stacks
      $\bar{\pi}_\mathrm{HT} \colon \igsu \to \bung\m$, where the action on
      $\bung\m$ is trivial,
    \item a 2-Cartesian diagram of v-stacks
      \[ \begin{tikzcd}
        U \arrow{r}{\pi_\mathrm{HT}} \arrow{d}{q_\mathrm{Igs}} & \grg\m
        \arrow{d}{\mathrm{BL}} \\ \igsu \arrow{r}{\bar{\pi}_\mathrm{HT}} &
        \bung\m,
      \end{tikzcd} \]
      (including the data of an isomorphism between the two compositions)
  \end{itemize}
  such that
  \begin{enumerate}
    \item the $\underline{\mathsf{G}(\af)}$-action on $U \cong \igsu
      \times_{\bung\m} \grg\m$ induced by the
      $\underline{\mathsf{G}(\afp)}$-action on $\igsu$ and the
      $\underline{G(\qp)}$-action on $\grg\m$ (see \Cref{Par:BLEquivariance})
      recovers the Hecke action on $U$,
    \item the $(\phi, \id)$-action on $\igsu \times_{\bung\m} \grg\m$, induced
      by the canonical isomorphism $\phi_{\bung\m} \cong \id_{\bung\m}$ from
      \Cref{Par:AbsoluteFrob}, recovers the identity map on $U$.
  \end{enumerate}
\end{definition}

\begin{remark} \label{Rem:AbsoluteFrobReform}
  Because $\mathrm{BL} \colon \grg\m \to \bung\m$ is v-surjective, see
  \Cref{Prop:BLSurjective}, the second condition may also be rephrased as the
  existence of an isomorphism $\phi_{\igsu} \cong \id_{\igsu}$ living over the
  canonical isomorphism $\phi_{\bung\m} \cong \id_{\bung\m}$. This indeed is a
  condition rather than extra data because $\igsu \to \bung\m$ is $0$-truncated,
  as guaranteed by the existence of an isomorphism $U \cong \igsu
  \times_{\bung\m} \grg\m$ together with \Cref{Prop:BLSurjective}, and hence
  \[
    \Hom(\igsu, \igsu) \to \Hom(\igsu, \bung\m)
  \]
  is also $0$-truncated.
\end{remark}

\begin{remark}
  The condition on the absolute Frobenius $\phi$ is not strictly necessary for
  our uniqueness and functoriality results in the sense that $(\igsu)_{k_E}$
  will still be unique and functorial. However, we include this axiom as it is
  crucial when comparing the cohomology of Igusa stacks to cohomology of Shimura
  varieties.
\end{remark}

\paragraph
Let $i \colon \gx \to \gx'$ be a morphism of Shimura data, and let $v^\prime
\mid p$ be the restriction of $v$ along $\mathsf{E}^\prime \subseteq
\mathsf{E}$. Let $U \subseteq \sh\d$ and $U^\prime \subseteq \sh\d'$ be
$\mathsf{G}(\af)$-stable and $\mathsf{G}^\prime(\af)$-stable open subsheaves,
satisfying the property that $U \to U^\prime$ under the natural map $\sh\d \to
\sh\d'$. Let $\igsu$ and $\igsu'$ to be Igusa stacks for $U$ and $U^\prime$. We
define a \textdef{morphism of Igusa stacks} from $\igsu$ to $\igsu'$ to be a
$\underline{\mathsf{G}(\afp)}$-equivariant map of v-stacks
\[
  f \colon \igsu \to \igsu'
\]
together with a $\underline{\mathsf{G}(\afp)}$-equivariant isomorphism
$\lambda_f$ filling in the diagram
\[ \begin{tikzcd}
  \igsu \arrow{r}{f} \arrow{d}{\bar{\pi}_\mathrm{HT}} & \igsu'
  \arrow{d}{\bar{\pi}_\mathrm{HT}} \\ \bung\m \arrow{r}{i} & \bung\m',
\end{tikzcd} \]
see \Cref{Par:EquivariantMorphism}, satisfying the property that the induced map
\begin{align*}
  U &\cong \igsu \times_{\bung\m} \grg\m \\ &\to \igsu' \times_{\bung\m'}
  \grg\m' \cong U^\prime
\end{align*}
is the restriction of the natural map $\sh\d \to \sh\d'$. We define an
\textdef{isomorphism between two morphism of Igusa stacks} $f, g \colon \igsu
\to \igsu'$ to be a $\underline{\mathsf{G}(\afp)}$-equivariant isomorphism $\rho
\colon f \cong g$ for which the two isomorphisms
\[
  \bar{\pi}_\mathrm{HT} \circ f \xrightarrow{\rho} \bar{\pi}_\mathrm{HT} \circ g
  \xrightarrow{\lambda_g} i \circ \bar{\pi}_\mathrm{HT}, \quad
  \bar{\pi}_\mathrm{HT} \circ f \xrightarrow{\lambda_f} i \circ
  \bar{\pi}_\mathrm{HT}
\]
agree.

\begin{proposition} \label{Prop:IgusaDiscrete}
  The groupoid of morphisms of Igusa stacks $\igsu \to \igsu'$ is
  $(-1)$-truncated, i.e., is either empty or contractible.
\end{proposition}

\begin{proof}
  We first show that the groupoid of morphisms of Igusa stacks $\igsu \to
  \igsu'$ is $0$-truncated. It suffices to verify that when there exist two
  isomorphism $\rho, \rho^\prime \colon f \cong g$ then $\rho = \rho^\prime$. By
  definition, we have
  \[
    (\bar{\pi}_\mathrm{HT} \circ f \xrightarrow{\rho} \bar{\pi}_\mathrm{HT}
    \circ g) = (\bar{\pi}_\mathrm{HT} \circ f \xrightarrow{\lambda_f} i \circ
    \bar{\pi}_\mathrm{HT} \xrightarrow{\lambda_g^{-1}} \bar{\pi}_\mathrm{HT}
    \circ g) = (\bar{\pi}_\mathrm{HT} \circ f \xrightarrow{\rho^\prime}
    \bar{\pi}_\mathrm{HT} \circ g).
  \]
  As in \Cref{Rem:AbsoluteFrobReform}, the map $\bar{\pi}_\mathrm{HT}
  \colon \igsu' \to \bung\m'$ is $0$-truncated by
  \Cref{Prop:BLSurjective}, and hence so is
  \[
    \bar{\pi}_\mathrm{HT} \circ - \colon \Hom(\igsu, \igsu') \to \Hom(\igsu,
    \bung\m').
  \]
  Thus the above equality implies $\rho = \rho^\prime$ as isomorphisms $f \cong
  g$.

  Next, we show that if $f, g$ are two morphisms are Igusa stacks $\igsu \to
  \igsu'$ then there exists an isomorphism $f \cong g$. Taking the fiber product
  of the commutative diagram $\lambda_f$ with the maps $\grg\m \to \bung\m$ and
  $\grg\m' \to \bung\m'$, we obtain a 2-commutative diagram
  \[ \begin{tikzcd}[row sep=small, column sep=large]
    U \arrow{dr}{\pi_\mathrm{HT}} \arrow{rr}{i} \arrow{dd}{q_\mathrm{Igs}} & &
    U^\prime \arrow{dr}{\pi_\mathrm{HT}} \arrow{dd}[near start]{q_\mathrm{Igs}}
    \\ & \grg\m \arrow{rr}[near start]{i} \arrow{dd}[near start]{\mathrm{BL}} &
    & \grg\m' \arrow{dd}{\mathrm{BL}} \\ \igsu \arrow{rr}[near start]{f}
    \arrow{dr}{\bar{\pi}_\mathrm{HT}} & & \igsu'
    \arrow{dr}{\bar{\pi}_\mathrm{HT}} \\ & \bung\m \arrow{rr}{i} & & \bung\m'
  \end{tikzcd} \]
  where all morphisms have a natural $\underline{\mathsf{G}(\afp)}$-equivariance
  structure and all 2-morphisms on faces of the cube are
  $\underline{\mathsf{G}(\afp)}$-equivariant. Because both maps $\mathrm{BL}$
  are v-surjective and 0-truncated by \Cref{Prop:BLSurjective}, we
  may identify $\igsu$ with the quotient of $U$ along
  $R = U \times_{\igsu} U$ with composition law $\mathrm{pr}_{13} \colon R
  \times_U R \to R$. That is, we may recover $f$ and $\lambda_f$ from the
  diagram $\underline{\mathsf{G}(\afp)}$-equivariant diagram
  \[ \begin{tikzcd}
    U \times_{\igsu} U \times_{\igsu} U \rrrll \arrow{d}{i \times_f i \times_f
    i} & U \times_{\igsu} U \arrow{d}{i \times_f i} \rrl & U \arrow{d}{i} \\
    U^\prime \times_{\igsu'} U^\prime \times_{\igsu'} U^\prime \rrrll & U^\prime
    \times_{\igsu'} U^\prime \rrl  & U^\prime,
  \end{tikzcd} \]
  where all objects are $0$-truncated, i.e., v-sheaves. It now suffices to show
  that
  \[
    i \times_f i = i \times_g i, \quad i \times_f i \times_f i = i \times_g i
    \times_g i.
  \]

  We observe that we have a commutative diagram
  \[ \begin{tikzcd}
    U \times_{\igsu} U \arrow{d}{\id \times \pi_\mathrm{HT}} \arrow{r}{i
    \times_f i} & U^\prime \times_{\igsu'} U^\prime \arrow{d}{\id \times
    \pi_\mathrm{HT}} \\ U \times_{\bung\m} \grg\m \arrow{r}{i \times_i i} &
    U^\prime \times_{\bung\m'} \grg\m'
  \end{tikzcd} \]
  where the vertical maps are isomorphisms and the bottom horizontal map does
  not depend on whether we take $f$ or $g$, and hence $i \times_f i = i \times_g
  i$. We similarly use the isomorphism
  \[
    U \times_{\igsu} U \times_{\igsu} U \cong U \times_{\bung\m} \grg\m
    \times_{\bung\m} \grg\m
  \]
  to check that $i \times_f i \times_f i = i \times_g i \times_g i$.
\end{proof}

\paragraph \label{Par:CechNerve}
We now give a different interpretation of an Igusa stack. Given an Igusa stack
$\igsu \to \bung\m$, we consider the 2-truncated \v{C}ech nerve
\[ \begin{tikzcd}
  U \times_{\igsu} U \times_{\igsu} U \arrow{r} \rrrll & U \times_{\igsu} U \rrl
  & U \arrow{r} & \igsu
\end{tikzcd} \]
living over the \v{C}ech nerve for $\mathrm{BL} \colon \grg\m \to \bung\m$.
Using the isomorphism $U \cong \igsu \times_{\bung\m} \grg\m$, we may identify
\begin{align*}
  \id \times \pi_\mathrm{HT} &\colon U \times_{\igsu} U \cong U \times_{\bung\m}
  \grg\m, \\ \id \times \pi_\mathrm{HT} \times \pi_\mathrm{HT} &\colon U
  \times_{\igsu} U \times_{\igsu} U \\ &\cong U \times_{\bung\m} \grg\m
  \times_{\bung\m} \grg\m,
\end{align*}
and consider the corresponding 2-truncated simplicial diagram
\[ \begin{tikzcd}
  \raisebox{2em}{$X_2 = U \times_{\bung\m} \grg\m \times_{\bung\m} \grg\m$}
  \hspace{-15em} \rrrll & X_1 = U \times_{\bung\m} \grg\m \rrl & X_0 = U.
\end{tikzcd} \]
Denote by
\[
  \Theta = d_{1,0} \colon U \times_{\bung\m} \grg\m \cong U \times_{\igsu} U
  \xrightarrow{\mathrm{pr}_2} U
\]
the face map. We see that the rest of the degeneracy and the face maps can be
described in terms of $\Theta$ as
\begin{align*}
  s_{0,0} &= (\id, \pi_\mathrm{HT}) \colon X_0 \to X_1, \quad d_{1,0} = \Theta
  \colon X_1 \to X_0, \quad d_{1,1} = \mathrm{pr}_1 \colon X_1 \to X_0, \\
  s_{1,0} &= s_{0,0} \times \id \colon X_1 \to X_2, \quad s_{1,1} = \id \times
  \Delta \colon X_1 \to X_2, \\ d_{2,0} &= \Theta \times \id \colon X_2 \to X_1,
  \quad d_{2,1} = \mathrm{pr}_{13} \colon X_2 \to X_1, \quad d_{2,2} =
  \mathrm{pr}_{12} \colon X_2 \to X_1.
\end{align*}
This motivates the following definition.

\begin{definition} \label{Def:GlobalUniformization}
  Let $\gx$ be a Shimura datum and consider an open subsheaf $U \subseteq \sh\d$
  stable under the action of $\underline{G(\qp)}$. A \textit{global
  uniformization} for $U$ is a map of v-sheaves
  \[
    \Theta \colon U \times_{\bung\m} \grg\m \to U
  \]
  such that
  \begin{enumerate}
    \item the map $\Theta$ is $\underline{\mathsf{G}(\af)} \times
      \underline{G(\qp)}$-equivariant where the target $U$ is given the action
      $((g^p, g_p), h_p).u = (g^p, h_p).u$ for $g^p \in
      \underline{\mathsf{G}(\afp)}$ and $g_p, h_p \in \underline{G(\qp)}$,
    \item the map $\Theta$ is intertwines the action of $\phi \times \id$ on the
      source with the action of $\id$ on the target,
    \item \label{Item:GlobalUnifDiagrams} the diagrams
      \[ \begin{tikzcd}
        U \arrow{d}[']{(\id, \pi_\mathrm{HT})} \arrow{r}{\id} & U
        \arrow{d}{\pi_\mathrm{HT}} \\ U \times_{\bung\m} \grg\m
        \arrow{ru}{\Theta} \arrow{r}{\mathrm{pr}_2} & \grg\m,
      \end{tikzcd} \]
      \[ \begin{tikzcd}
        U \times_{\bung\m} \grg\m \times_{\bung\m} \grg\m
        \arrow{d}{\mathrm{pr}_{13}} \arrow{r}{\Theta \times \id} & U
        \times_{\bung\m} \grg\m \arrow{d}{\Theta} \\ U \times_{\bung\m} \grg\m
        \arrow{r}{\Theta} & U,
      \end{tikzcd} \]
      commute.
  \end{enumerate}
\end{definition}

\begin{lemma} \label{Lem:InverseAutomatic}
  Let $\Theta \colon U \times_{\bung\m} \grg\m \to U$ be a global uniformization
  map. Then the commutative diagram
  \[ \begin{tikzcd}
    U \times_{\bung\m} \grg\m \arrow{r}{\Theta} \arrow{d}{\pi_\mathrm{HT} \times
    \id} & U \arrow{d}{\pi_\mathrm{HT}} \\ \grg\m \times_{\bung\m} \grg\m
    \arrow{r}{\mathrm{pr}_2} & \grg\m
  \end{tikzcd} \]
  is Cartesian.
\end{lemma}

\begin{proof}
  The statement is equivalent to showing that the map
  \[
    \alpha = (\pi_\mathrm{HT} \circ \mathrm{pr}_1, \Theta) \colon U
    \times_{\bung\m} \grg\m \to \grg\m \times_{\bung\m} U
  \]
  of v-sheaves is an isomorphism. Write $\beta = \mathrm{swap} \circ \alpha =
  (\Theta, \pi_\mathrm{HT} \circ \mathrm{pr}_1)$, and observe that it suffices
  to prove $\beta^2 = \id$. It is clear from the construction that
  \[ \begin{tikzcd}
    U \times_{\bung\m} \grg\m \arrow{r}{\beta} \arrow{d}{\pi_\mathrm{HT} \times
    \id} & U \times_{\bung\m} \grg\m \arrow{d}{\pi_\mathrm{HT} \times \id} \\
    \grg\m \times_{\bung\m} \grg\m \arrow{r}{\mathrm{swap}} & \grg\m
    \times_{\bung\m} \grg\m
  \end{tikzcd} \]
  commutes, and hence $(\pi_\mathrm{HT} \times \id) \circ \beta^2 =
  (\pi_\mathrm{HT} \times \id)$. Because the map of v-sheaves
  \[
    U \times_{\bung\m} \grg\m \xrightarrow{(\mathrm{pr}_1, \pi_\mathrm{HT}
    \times \id)} U \times (\grg\m \times_{\bung\m} \grg\m)
  \]
  is injective, as $\grg\m$ is a v-sheaf, it remains to show that $\mathrm{pr}_1
  \circ \beta^2 = \mathrm{pr}_1$. We check that $\mathrm{pr}_1 \circ \beta^2$
  can be understood as the composition
  \begin{align*}
    U \times_{\bung\m} &\grg\m \\ &\xrightarrow{(\id, \pi_\mathrm{HT}) \times
    \id} U \times_{\bung\m} \grg\m \times_{\bung\m} \grg\m \\
    &\xrightarrow{\mathrm{swap}_{23}} U \times_{\bung\m} \grg\m \times_{\bung\m}
    \grg\m \\ &\xrightarrow{\Theta \times \id} U \times_{\bung\m} \grg\m
    \xrightarrow{\Theta} U,
  \end{align*}
  which simplifies to $\mathrm{pr}_1$ by the axioms for $\Theta$.
\end{proof}

\begin{proposition} \label{Prop:IgusaGlobalUnif}
  The constructions of \Cref{Par:Simplicial} and
  \Cref{Par:CechNerve} defines a bijection between the set of isomorphism
  classes of Igusa stacks $\igsu \to \bung\m$ and the set of global
  uniformization maps $\Theta \colon U \times_{\bung\m} \grg\m \to U$.
\end{proposition}

\begin{proof}
  Given an Igusa stack $\igsu \to \bung\m$, we take its 2-truncated \v{C}ech
  nerve, identify it with the 2-truncated simplicial v-sheaf
  \[ \begin{tikzcd}
    X_2 \rrrll & X_1 = U \times_{\bung\m} \grg\m \rrl & X_0 = U,
  \end{tikzcd} \]
  and define $\Theta = d_{1,0} \colon X_1 \to X_0$ as in \Cref{Par:CechNerve}.
  Commutativity of the upper triangle in the first diagram of
  (\ref{Item:GlobalUnifDiagrams}) comes from the fact that $d_{1,0} \circ
  s_{0,0} = \id$ on the simplicial v-sheaf $X_\bullet$, commutativity of the
  lower triangle comes from the fact that $s_{0,0} \circ d_{1,0} = d_{2,0} \circ
  s_{1,1}$ as maps $X_1 \to X_1$, and commutativity of the second diagram comes
  from the fact that $d_{1,0} \circ d_{2,0} = d_{1,0} \circ d_{2,1}$ as maps
  $X_2 \to X_1$.

  To see that $\Theta$ is $\underline{\mathsf{G}(\af)} \times
  \underline{G(\qp)}$-equivariant, we note that $\Theta$ can also be described
  as the composition
  \begin{align*}
    U &\times_{\bung\m} \grg\m \\ &\cong \igsu \times_{\bung\m} \grg\m
    \times_{\bung\m} \grg\m \\ &\xrightarrow{\mathrm{pr}_{13}} \igsu
    \times_{\bung\m} \grg\m \cong U
  \end{align*}
  and use the fact that the isomorphism $f \colon U \cong \igsu \times_{\bung\m}
  \grg\m$ is $\underline{\mathsf{G}(\af)}$-equivariant. Similarly, the
  isomorphism $f$ intertwines $\id$ with $\phi \times \id$, and hence $f$
  intertwines $\phi$ with $\id \times \phi$, and it follows that $\Theta$
  intertwines $\phi \times \id$ with $\id$. This verifies the $\Theta$ is indeed
  a global uniformization for $U$.

  In the reverse direction, given $\Theta$ a global uniformization, we construct
  a 2-truncated simplicial diagram $X_\bullet$ as in
  \Cref{Par:CechNerve}. One verifies from the commutativity of the
  diagrams in (\ref{Item:GlobalUnifDiagrams}) of
  \Cref{Def:GlobalUniformization} that this is indeed a 2-truncated
  simplicial diagram, i.e., satisfies all the relations between the face and
  degeneracy maps. It also satsifies the conditions described in
  \Cref{Par:Simplicial}; the only nontrivial condition is the property that the
  second square in (\ref{Item:GlobalUnifDiagrams}) is Cartesian, but this
  follows from \Cref{Lem:InverseAutomatic} because both the lower square and the
  outer square in
  \[ \begin{tikzcd}[row sep=small]
    U \times_{\bung\m} \grg\m \times_{\bung\m} \grg\m
    \arrow{d}{\mathrm{pr}_{13}} \arrow{r}{\Theta \times \id} & U
    \times_{\bung\m} \grg\m \arrow{d}{\Theta} \\ U \times_{\bung\m} \grg\m
    \arrow{r}{\Theta} \arrow{d}{\pi_\mathrm{HT} \times \id} & U
    \arrow{d}{\pi_\mathrm{HT}} \\ \grg\m \times_{\bung\m} \grg\m
    \arrow{r}{\mathrm{pr}_2} & \grg\m
  \end{tikzcd} \]
  are Cartesian.

  The 2-truncated simplicial diagram $X_\bullet$ naturally lives over the
  2-truncated \v{C}ech nerve $Y_\bullet$ of $\mathrm{BL} \colon \grg\m \to
  \bung\m$,
  \[ \begin{tikzcd}
    X_2 \rrrll \arrow{d}{\pi_\mathrm{HT} \times \id \times \id} & X_1 = U
    \times_{\bung\m} \grg\m \rrl \arrow{d}{\pi_\mathrm{HT} \times \id} & X_0 = U
    \arrow{d}{\pi_\mathrm{HT}} \\ Y_2 \rrrll & Y_1 = \grg\m \times_{\bung\m}
    \grg\m \rrl & Y_0 = \grg\m.
  \end{tikzcd} \]
  In this morphism $X_\bullet \to Y_\bullet$ between simplicial diagrams, every
  square corresponding to a morphism from $X_i \to X_j$ to $Y_i \to Y_j$ is
  Cartesian by \Cref{Lem:InverseAutomatic}. Therefore by taking the
  quotient stacks of both simplicial diagrams, we obtain a diagram
  \[ \begin{tikzcd}
    X_0 = U \arrow{r} \arrow{d}{\pi_\mathrm{HT}} & \lbrack X_0/X_1 \rbrack
    \arrow{d} \\ Y_0 = \grg\m \arrow{r}{\mathrm{BL}} & \bung\m
  \end{tikzcd} \]
  together with an isomorphism $\eta$ filling in the square, cf.,
  \cite[Theorem~6.1.3.9]{Lur09}.

  Recall that $\Theta$ is a global uniformization map, hence is
  $\underline{\mathsf{G}(\afp)} \times \underline{G(\qp)}^2$-equivariant. It
  follows that the above diagram of v-sheaves is equivariant for the action of
  $\underline{\mathsf{G}(\afp)}$, where the action on $X_\bullet$ is on $U$ and
  the action on $Y_\bullet$ is trivial. It follows that the v-stack $[X_0/X_1]$
  and the map $[X_0/X_1] \to \bung\m$ has as natural
  $\underline{\mathsf{G}(\afp)}$-equivariant structure, where $\eta$ is also
  equivariant.

  Because $\Theta$ is also $\underline{G(\qp)}^2$-equivariant, there is an
  action of the simplicial group $\underline{G(\qp)}^{\bullet+1}$ on both
  $X_\bullet$ and $Y_\bullet$, under which $X_\bullet \to Y_\bullet$ is
  equivariant. Similarly, there is an action of $(\phi^\mathbb{Z})^{\bullet+1}$
  on both $X_\bullet$ and $Y_\bullet$, under which the map of simplicial objects
  $X_\bullet \to Y_\bullet$ is equivariant. It follows that the
  $\id \times \underline{G(\qp)}$-action and the $\id \times
  \phi^\mathbb{Z}$-action on $[X_0/X_1] \times_{\bung\m} \grg\m$ corresponds to
  the natural $\underline{G(\qp)}$-action and the $\phi^\mathbb{Z}$-action on
  $U$. This shows that $[X_0/X_1] \to \bung\m$ is indeed an Igusa stack.

  It is evident that the two constructions are inverse to each other. Therefore
  they give rise to a bijection between the set of isomorphism classes of Igusa
  stacks and the set of global uniformization maps.
\end{proof}

\begin{proposition} \label{Prop:IgusaThetaFunct}
  Let $\gx \to \gx'$ be a morphism of Shimura data, and let $U \subseteq \sh\d$
  and $U^\prime \subseteq \sh\d'$ be $\mathsf{G}(\af)$-stable and
  $\mathsf{G}^\prime(\af)$-stable open subsheaves. Let $\igsu \to \bung\m$ and
  $\igsu' \to \bung\m'$ be Igusa stacks for $U$ and $U^\prime$, corresponding to
  global uniformization maps $\Theta \colon U \times_{\bung\m} \grg\m \to U$ and
  $\Theta^\prime \colon U^\prime \times_{\bung\m'} \grg\m' \to U^\prime$. Then
  there exists a morphism of Igusa stacks $\igsu \to \igsu'$ if and only if the
  diagram
  \begin{equation} \label{Eq:IgusaThetaFunct} \begin{tikzcd}
    U \times_{\bung\m} \grg\m \arrow{d}{\Theta} \arrow{r} & U^\prime
    \times_{\bung\m'} \grg\m' \arrow{d}{\Theta^\prime} \\ U \arrow{r} & U^\prime
  \end{tikzcd} \tag{$\ast$} \end{equation}
  of v-sheaves commutes.
\end{proposition}

\begin{proof}
  As in the proof of \Cref{Prop:IgusaDiscrete}, a morphism of Igusa
  stacks gives rise to a $\underline{\mathsf{G}(\afp)}$-equivariant diagram
  \[ \begin{tikzcd}[row sep=small]
    U \times_{\igsu} U \times_{\igsu} U \rrrll \arrow{d} & U \times_{\igsu} U
    \arrow{d} \rrl & U \arrow{d} \\ U^\prime \times_{\igsu'} U^\prime
    \times_{\igsu'} U^\prime \rrrll & U^\prime \times_{\igsu'} U^\prime \rrl &
    U^\prime.
  \end{tikzcd} \]
  By identifying $U \times_{\igsu} U \cong U \times_{\bung\m} \grg\m$ and
  similarly for $U^\prime$, the commutativity of the square corresponding to the
  face map $d_{1,0}$ corresponds the commutativity of
  \eqref{Eq:IgusaThetaFunct}.

  Conversely, if \eqref{Eq:IgusaThetaFunct} commutes, then the construction of
  \Cref{Par:CechNerve} gives a $\underline{\mathsf{G}(\afp)}$-equivariant
  diagram
  \[ \begin{tikzcd}[row sep=small]
    X_2 \rrrll \arrow{d} & U \times_{\bung\m} \grg\m \rrl \arrow{d} & U
    \arrow{d} \\ X_2^\prime \rrrll & U^\prime \times_{\bung\m'} \grg\m' \rrl &
    U^\prime.
  \end{tikzcd} \]
  Taking the quotient stack for both rows, we obtain a morphism of Igusa stacks
  $\igsu \to \igsu'$.
\end{proof}

}
{\section{Global uniformization at finite level} \label{Sec:FiniteLevel}
\def\grgmu{\mathrm{Gr}_{G,\{\mu\},E}}
\def\grgmuc{\mathrm{Gr}_{G^\mathrm{c},\{\mu^\mathrm{c}\},E}}

For reasons we will see later, we wish to work with finite level Shimura
varieties $\sh[K]$ instead of its profinite cover $\sh$. For $K \subseteq
\mathsf{G}(\af)$ a neat compact open subgroup, we construct v-sheaves
\begin{align*}
  \mathcal{R} &= U \times_{\bung\m} \grg\m \\ &\to \mathcal{R}_K
  \twoheadrightarrow \mathcal{Q}_K^0 \subseteq \mathcal{Q}_K = U_K
  \times_{\bung\c\m} [\underline{K_p^\mathrm{c}} \backslash \grg\c\m]
\end{align*}
where $\mathcal{R}_K \twoheadrightarrow \mathcal{Q}_K^0$ is a torsor for a
discrete abelian group and $\mathcal{Q}_K^0 \hookrightarrow \mathcal{Q}_K$ is a
open and closed embedding. Much of the complication originates from the fact
that the center $\mathsf{Z}$ of $\mathsf{G}$ need not be cuspidal, and hence we
cannot simply identify $\sh[K]$ with $[\underline{K} \backslash \sh]$; in the
case when $\mathsf{G} = \mathsf{G}^\mathrm{c}$, e.g., when $\gx$ is of Hodge
type, we have $\mathcal{R}_K = \mathcal{Q}_K^0 = \mathcal{Q}_K$.

\paragraph
Throughout this section, we let $\gx$ be a Shimura datum, let $E = \mathsf{E}_v$
be a local reflex field, and let $U \subseteq \sh\d$ be an open subsheaf stable
under the action of $\mathsf{G}(\af)$. Consider the v-sheaf
\[
  \mathcal{R} = U \times_{\bung\m} \grg\m,
\]
which is the source of the global uniformization map. Since $\pi_\mathrm{HT}
\colon U \to \grg\m$ is $\underline{\mathsf{G}(\af)}$-equivariant and there is a
natural $\underline{G(\qp)}$-equivariant structure on $\mathrm{BL} \colon \grg\m
\to \bung\m$, there is a natural $\underline{\mathsf{G}(\af)} \times
\underline{G(\qp)}$-action on $\mathcal{R}$. Similarly, we define
\[
  \mathcal{Q} = U \times_{\bung\c\m} \grg\c\m,
\]
which naturally has a $\underline{\mathsf{G}(\af)} \times
\underline{G^\mathrm{c}(\qp)}$-action.

\begin{lemma} \label{Lem:ActionOnR}
  The closed subgroup
  \[
    (\id_{\mathsf{G}(\afp)} \times \Delta_{G(\qp)})(\mathsf{Z}(\mathbb{Q})^-)
    \subseteq \mathsf{G}(\af) \times G(\qp)
  \]
  acts trivially on $\mathcal{R}$. Similarly, the closed subgroup
  \[
    (\id_{\mathsf{G}(\afp)} \times
    \Delta^\mathrm{c}_{G(\qp)})(\mathsf{Z}(\mathbb{Q})^-) \subseteq
    \mathsf{G}(\af) \times G^\mathrm{c}(\qp)
  \]
  acts trivially on $\mathcal{Q}$, where $\Delta_{G(\qp)}^\mathrm{c} \colon
  G(\qp) \to G(\qp) \times G^\mathrm{c}(\qp)$ is the graph of the natural
  projection $G(\qp) \to G^\mathrm{c}(\qp)$.
\end{lemma}

\begin{proof}
  Write
  \[
    \mathcal{R} = U \times_{\grg\m} (\grg\m \times_{\bung\m} \grg\m).
  \]
  Because $\mathsf{Z}(\mathbb{Q})^- \subseteq \mathsf{G}(\af)$ acts trivially on
  $\sh$, it suffices to show that the diagonal $Z(\qp) \subseteq G(\qp)^2$ acts
  trivially on $\grg\m \times_{\bung\m} \grg\m$. By the description in
  \Cref{Par:BLEquivariance}, this translates to the fact that for two
  untilts $R^{\sharp_1}, R^{\sharp_2}$ of $R$ over $E$, two points $x_i \in
  \grg\m(R^{\sharp_i}, R^{\sharp_i+})$ inducing $G$-torsors $\mathscr{P}_i /
  \mathcal{X}_\mathrm{FF}(R, R^+)$, an isomorphism $\lambda \colon \mathscr{P}_1
  \cong \mathscr{P}_2$, and an element $z \in \underline{Z(\qp)}(R, R^+)$ the
  diagram
  \[ \begin{tikzcd}
    \mathscr{P}_1 \arrow{r}{\lambda} \arrow{d}{r_z} & \mathscr{P}_2
    \arrow{d}{r_z} \\ \mathscr{P}_1 \arrow{r}{\lambda} & \mathscr{P}_2
  \end{tikzcd} \]
  commutes. Here $r_z \colon \mathscr{P}_i \to \mathscr{P}_i$ because $z \in
  \underline{Z(\qp)}$ implies $z x_i = x_i$, and $r_z$ is the unique extension
  right multiplication by $z$ on
  \[
    \mathscr{P}_i \vert_{\mathcal{X}_\mathrm{FF}(R, R^+) \setminus
    \Spa(R^{\sharp_i}, R^{\sharp_i+})} \cong G \times
    (\mathcal{X}_\mathrm{FF}(R, R^+) \setminus \Spa(R^{\sharp_i},
    R^{\sharp_i+})).
  \]
  Because $z$ is central, right multiplication by $z$ is also left
  multiplication by $z$, which extends along the divisor $\Spa(R^{\sharp_i},
  R^{\sharp_i+})$ as well. Therefore $r_z$ is left multiplication by $z$ on all
  of $\mathcal{X}_\mathrm{FF}(R, R^+)$. As $\lambda$ is $G$-equivariant by
  definition, we see that the above diagram commutes.

  Similarly, we write
  \[
    \mathcal{Q} = U \times_{\grg\m} (\grg\m \times_{\bung\c\m} \grg\c\m)
  \]
  and argue as above that $\Delta_{G(\qp)}^\mathrm{c}(Z(\qp)) \subseteq Z(\qp)
  \times Z^\mathrm{c}(\qp)$ acts trivially on the v-sheaf $\grg\m
  \times_{\bung\c\m} \grg\c\m$.
\end{proof}

\paragraph \label{Par:RToQ}
From the diagram
\[ \begin{tikzcd}[column sep=large]
  U \arrow[equals]{d} \arrow{r}{\mathrm{BL} \circ \pi_\mathrm{HT}} & \bung\m
  \arrow{d} & \grg\m \arrow{l}[']{\mathrm{BL}} \arrow{d}{\cong} \\ U
  \arrow{r}{\mathrm{BL} \circ \pi_\mathrm{HT}} & \bung\c\m & \grg\c\m
  \arrow{l}[']{\mathrm{BL}}
\end{tikzcd} \]
we obtain a natural map
\[
  \mathcal{R} = U \times_{\bung\m} \grg\m \to U \times_{\bung\c\m} \grg\c\m =
  \mathcal{Q}.
\]
It is clear from the construction that this map is equivariant with respect to
the actions of $\underline{\mathsf{G}(\af)} \times \underline{G(\qp)} \to
\underline{\mathsf{G}(\af)} \times \underline{G^\mathrm{c}(\qp)}$.

\begin{lemma} \label{Lem:BLCompareCuspidal}
  The commutative diagram
  \[ \begin{tikzcd}[row sep=small]
    \bung\m \arrow{d} & \lbrack \underline{G(\qp)} \backslash \grg\m \rbrack
    \arrow{d} \arrow{l}[']{\mathrm{BL}} \\ \bung\c\m & \lbrack
    \underline{G^\mathrm{c}(\qp)} \backslash \grg\c\m \rbrack
    \arrow{l}[']{\mathrm{BL}}
  \end{tikzcd} \]
  is Cartesian.
\end{lemma}

\begin{proof}
  As in \cite[Proposition~23.3.1]{SW20}, the stack $[\underline{G(\qp)}
  \backslash \grg\m]$ has the moduli interpretation of sending $(R, R^+)$ to an
  untilt $E \to R^\sharp$ together with a meromorphic map
  \[
    \mathscr{P} \vert_{\mathcal{X}_\mathrm{FF}(R, R^+) \setminus \Spa(R^\sharp,
    R^{\sharp+})} \xrightarrow{\cong} \mathscr{Q}
    \vert_{\mathcal{X}_\mathrm{FF}(R, R^+) \setminus \Spa(R^\sharp,
    R^{\sharp+})}
  \]
  of $G$-torsors $\mathscr{P}, \mathscr{Q}$ over $\mathcal{X}_\mathrm{FF}(R,
  R^+)$ with relative position $\{\mu\}$, where $\mathscr{Q}$ is trivial over
  $\mathcal{X}_\mathrm{FF}(C, C^+)$ for all geometric points $R \to C$.
  Therefore it suffices to show that given a $G$-torsor $\mathscr{P} /
  \mathcal{X}_\mathrm{FF}(R, R^+)$, a $G^\mathrm{c}$-torsor
  $\mathscr{Q}^\mathrm{c} / \mathcal{X}_\mathrm{FF}(R, R^+)$ that is trivial
  over all geometric points, and a meromorphic map
  \[
    \alpha^\mathrm{c} \colon G^\mathrm{c} \times^G \mathscr{P}
    \vert_{\mathcal{X}_\mathrm{FF}(R, R^+) \backslash \Spa(R^\sharp,
    R^{\sharp+})} \cong \mathscr{Q}^\mathrm{c} \vert_{\mathcal{X}_\mathrm{FF}(R,
    R^+) \backslash \Spa(R^\sharp, R^{\sharp+})}
  \]
  of relative position $\{\mu^\mathrm{c}\}$, there exists a unique lift of
  $\mathscr{Q}^\mathrm{c}$ to a $G$-torsor $\mathscr{Q}$ together with a lift
  \[
    \alpha \colon \mathscr{P} \vert_{\mathcal{X}_\mathrm{FF}(R, R^+) \backslash
    \Spa(R^\sharp, R^{\sharp+})} \cong \mathscr{Q}
    \vert_{\mathcal{X}_\mathrm{FF}(R, R^+) \backslash \Spa(R^\sharp,
    R^{\sharp+})}
  \]
  of $\alpha^\mathrm{c}$, with the property that $\mathscr{Q}$ is trivial over
  all geometric points.

  Because the natural map $\grgmu \to \grgmuc$ induced by $G \to G^\mathrm{c}$
  is an isomorphism, there exists a unique lift of $\alpha^\mathrm{c} \colon
  G^\mathrm{c} \times^G \mathscr{P} \vert_{\period{B}{dR}(R^\sharp,
  R^{\sharp+})} \cong \mathscr{Q}^\mathrm{c} \vert_{\period{B}{dR}(R^\sharp,
  R^{\sharp+})}$ to a $G$-torsor $\mathscr{Q}^\mathrm{dR}$ on
  $\period{B}{dR}\p(R^\sharp, R^{\sharp+})$ and an isomorphism
  $\alpha^\mathrm{dR} \colon \mathscr{P} \vert_{\period{B}{dR}(R^\sharp,
  R^{\sharp+})} \cong \mathscr{Q}^\mathrm{dR} \vert_{\period{B}{dR}(R^\sharp,
  R^{\sharp+})}$ of relative position $\{\mu\}$. Then we may glue
  $\mathscr{Q}^\mathrm{dR}$ and $\mathscr{P} \vert_{\mathcal{X}_\mathrm{FF}(R,
  R^+) \setminus \Spa(R^\sharp, R^{\sharp+})}$ using \cite[Lemma~5.2.9]{SW20} to
  obtain a $G$-torsor $\mathscr{Q}$ on $X_\mathrm{FF}(R^\sharp, R^{\sharp+})$
  together with
  \[
    \alpha \colon \mathscr{P} \vert_{\mathcal{X}_\mathrm{FF}(R, R^+) \backslash
    \Spa(R^\sharp, R^{\sharp+})} \cong \mathscr{Q}
    \vert_{\mathcal{X}_\mathrm{FF}(R, R^+) \backslash \Spa(R^\sharp,
    R^{\sharp+})}
  \]
  as desired.

  It remains to verify that $\mathscr{Q}$ is trivial over all geometric points
  $R^\sharp \to C^\sharp$. By \cite[Theorem~5.1]{Far20}, isomorphism classes of
  $G$-torsors on $\mathcal{X}_\mathrm{FF}(C, C^+)$ are classified by elements of
  $B(G)$, and say $\mathscr{Q} \vert_{\mathcal{X}_\mathrm{FF}(C, C^+)}$
  corresponds to $[b] \in B(G)$. Its image under $B(G) \to B(G^\mathrm{c})$ is
  trivial, and in particular, the composition $\mathbb{D} \xrightarrow{\nu_b} G
  \to G^\mathrm{c}$ is trivial where $\mathbb{D}$ is as in
  \cite[Section~3]{Kot85}. Thus $\nu_b$ factors through $Z_\mathrm{ac} \subseteq
  G$, which means that $[b]$ is basic. On the other hand, we have by assumption
  that $\mathscr{P} \vert_{\mathcal{X}_\mathrm{FF}(C, C^+)}$ has Kottwitz
  invariant $-\mu^\natural$, and because the rational map $\mathscr{P} \to
  \mathscr{Q}$ has relative position $\{\mu\}$, we see that $\mathscr{Q}
  \vert_{\mathcal{X}_\mathrm{FF}(C, C^+)}$ has trivial Kottwitz invariant.
  Therefore $[b] \in B(G)$ is trivial.
\end{proof}

\paragraph \label{Par:QZero}
Using \Cref{Lem:BLCompareCuspidal} we identify
\[
  \bung\m \times_{\bung\c\m} \grg\c\m \cong \underline{G^\mathrm{c}(\qp)}
  \times^{\underline{G(\qp)}} \grg\m,
\]
and hence
\[
  \mathcal{Q} = U \times_{\bung\c\m} \grg\c\m \cong U \times_{\bung\m}
  (\underline{G^\mathrm{c}(\qp)} \times^{\underline{G(\qp)}} \grg\m).
\]
Because the map of Lie algebras $\operatorname{Lie}(G) \to
\operatorname{Lie}(G^\mathrm{c})$ is surjective, the image $\im(G(\qp) \to
G^\mathrm{c}(\qp))$ is an open subgroup of $G^\mathrm{c}(\qp)$. Hence we may
define the open and closed subsheaf
\[
  \mathcal{Q}^0 = U \times_{\bung\m} (\underline{\im(G(\qp) \to
  G^\mathrm{c}(\qp))} \times^{\underline{G(\qp)}} \grg\m) \subseteq \mathcal{Q}.
\]
This can be identified with the quotient
\[
  \mathcal{Q}^0 \cong (\{1\} \times \underline{Z_\mathrm{ac}(\qp)}) \backslash
  \mathcal{R},
\]
where $Z_\mathrm{ac} = (\mathsf{Z}_\mathrm{ac})_{\qp}$ and the action is via the
natural inclusion $\{1\} \times Z_\mathrm{ac}(\qp) \hookrightarrow
\mathsf{G}(\af) \times G(\qp)$. In particular, $\mathcal{Q}^0 \subseteq
\mathcal{Q}$ can also be thought of as the v-sheaf image of the natural map
$\mathcal{R} \to \mathcal{Q}$. Because $\mathcal{R} \to \mathcal{Q}$ is
equivariant with respect to the group actions of $\underline{\mathsf{G}(\af)}
\times \underline{G(\qp)} \to \underline{\mathsf{G}(\af)} \times
\underline{G^\mathrm{c}(\qp)}$, we see that $\mathcal{Q}^0$ is stable under the
action of
\[
  \underline{\mathsf{G}(\af)} \times \underline{\im(G(\qp) \to
  G^\mathrm{c}(\qp))} \subseteq \underline{\mathsf{G}(\af)} \times
  \underline{G^\mathrm{c}(\qp)}.
\]

\paragraph \label{Par:QFinite}
Let $K = K_p K^p \subseteq \mathsf{G}(\af)$ be a neat compact open subgroup.
Recall from \Cref{Par:LocalSystem} that there is a de Rham
$\underline{G^\mathrm{c}(\qp)}$-torsor $\xi_K$ on $\sh[K]$. Because $\sh\d \to
\sh\d[K]$ is a $\underline{K / (K \cap \mathsf{Z}(\mathbb{Q})^-)}$-torsor, see
\cite[Theorem~5.28]{Mil05} for example, any
open subsheaf $U \subseteq \sh\d$ stable under the $\mathsf{G}(\af)$-action
descends to an open subsheaf $U_K \subseteq \sh\d[K]$, and moreover the natural
projection map $U \to U_K$ is also a $\underline{K / (K \cap
\mathsf{Z}(\mathbb{Q})^-)}$-torsor. We naturally have a commutative diagram
\[ \begin{tikzcd}
  U \arrow{d} \arrow{r}{\mathrm{BL} \circ \pi_\mathrm{HT}} &[4em] \bung\m
  \arrow{d} & \grg\m \arrow{l}[']{\mathrm{BL}} \arrow{d}{\cong} \\ U_K
  \arrow{r}{\mathrm{BL} \circ \mathrm{DRT}(\xi_K)} & \bung\c\m & \grg\c\m.
  \arrow{l}[']{\mathrm{BL}}
\end{tikzcd} \]
Let $K_p^\mathrm{c} \subseteq G^\mathrm{c}(\qp)$ be the image of $K_p \subseteq
G(\qp)$, which is an open compact subgroup since the map on Lie algebras
$\operatorname{Lie}(G) \to \operatorname{Lie}(G^\mathrm{c})$ is surjective. We
define
\[
  \mathcal{Q}_K = U_K \times_{\bung\c\m} [\underline{K_p^\mathrm{c}} \backslash
  \grg\c\m],
\]
so that there is a map
\[
  \mathcal{R} = U \times_{\bung\m} \grg\m \to \mathcal{Q}_K
\]
of v-sheaves. It follows from the construction that once we take the limit over
all $K$, we recover the map
\[
  \mathcal{R} \to \varprojlim_K \mathcal{Q}_K = U \times_{\bung\c\m} \grg\c\m =
  \mathcal{Q}
\]
from \Cref{Par:RToQ}.

\begin{lemma} \label{Lem:QFinQuotOfQ}
  The projection map $\mathcal{Q} \to \mathcal{Q}_K$ identifies $\mathcal{Q}_K$
  as the quotient
  \[
    \mathcal{Q}_K \cong \biggl[ \underline{\frac{K \times K_p^\mathrm{c}}{(\id
    \times \Delta^\mathrm{c})(K \cap \mathsf{Z}(\mathbb{Q})^-)}} \bigg\backslash
    \mathcal{Q} \biggr].
  \]
\end{lemma}

\begin{proof}
  As in the proof of \Cref{Lem:ActionOnR}, we identify $\mathcal{Q}$ with
  the fiber product of
  \[
    U \to \grg\c\m \xleftarrow{\mathrm{pr}_1} \grg\c\m \times_{\bung\c\m}
    \grg\c\m,
  \]
  which is equivariant for the action of
  \[
    \frac{\mathsf{G}(\af)}{\mathsf{Z}(\mathbb{Q})^-} \to
    \frac{G^\mathrm{c}(\qp)}{Z^\mathrm{c}(\qp)} \xleftarrow{\mathrm{pr}_1}
    \frac{G^\mathrm{c}(\qp) \times
    G^\mathrm{c}(\qp)}{\Delta(Z^\mathrm{c}(\qp))},
  \]
  where we write $Z^\mathrm{c} = Z/Z_\mathrm{ac} \subseteq G^\mathrm{c}$.
  We also record the general fact that if morphisms of v-sheaves $X_1 \to X_2
  \leftarrow X_3$ are equivariant for homomorphisms $G_1 \to G_2 \leftarrow G_3$
  of v-sheaves of groups, then
  \[
    [G_1 \backslash X_1] \times_{[G_2 \backslash X_2]} [G_3 \backslash X_3]
    \cong [(G_1 \times_{G_2} G_3) \backslash (X_1 \times_{X_2} X_3)]
  \]
  provided that the multiplication map $G_1 \times G_3 \to G_2$ is surjective.

  We note that the map $U_K \to \bung\c\m$ factors as
  \[
    U_K \to [\underline{K_p^\mathrm{c}} \backslash \xi_K]
    \xrightarrow{\mathrm{DRT}(\xi_K)} [\underline{K_p^\mathrm{c}} \backslash
    \grg\c\m] \xrightarrow{\mathrm{BL}} \bung\c\m
  \]
  because the map $K / (K \cap \mathsf{Z}(\mathbb{Q})^-) \to G^\mathrm{c}(\qp)$
  has image contained in $K_p^\mathrm{c}$, and hence the
  $\underline{G^\mathrm{c}(\qp)}$-torsor $\xi_K$ has a natural
  $K_p^\mathrm{c}$-structure. On the other hand, we have Cartesian squares
  \[ \begin{tikzcd}[row sep=small]
    \Bigl\lbrack \underline{\frac{K_p^\mathrm{c}}{K_p^\mathrm{c} \cap
    Z^\mathrm{c}(\qp)}} \Big\backslash \grg\c\m \Bigr\rbrack & \Bigl\lbrack
    \underline{\frac{K_p^\mathrm{c} \times K_p^\mathrm{c}}{\Delta(K_p^\mathrm{c}
    \cap Z^\mathrm{c}(\qp))}} \Big\backslash \mathcal{F} \Bigr\rbrack
    \arrow{l}{\mathrm{pr}_1} \\ \lbrack \underline{K_p^\mathrm{c}} \backslash
    \grg\c\m \rbrack \arrow{d} \arrow{u} & \lbrack (\underline{K_p^\mathrm{c}
    \times K_p^\mathrm{c}}) \backslash \mathcal{F} \rbrack
    \arrow{d}{\mathrm{pr}_2} \arrow{l}{\mathrm{pr}_1} \arrow{u} \\ \bung\c\m &
    \lbrack \underline{K_p^\mathrm{c}} \backslash \grg\c\m \rbrack, \arrow{l}
  \end{tikzcd} \]
  where we temporarily write
  \[
    \mathcal{F} = \grg\c\m \times_{\bung\c\m} \grg\c\m.
  \]
  It follows that we have isomorphisms
  \begin{align*}
    \mathcal{Q}_K &\cong U_K \times_{[\underline{K_p^\mathrm{c}} \backslash
    \grg\c\m]} [\underline{K_p^\mathrm{c} \times K_p^\mathrm{c}} \backslash
    \mathcal{F}] \\ &\cong U_K \times_{\bigl[
      \underline{\frac{K_p^\mathrm{c}}{K_p^\mathrm{c} \cap Z^\mathrm{c}(\qp)}}
      \big\backslash \grg\c\m \bigr]} \biggl[ \underline{\frac{K_p^\mathrm{c}
    \times K_p^\mathrm{c}}{\Delta(K_p^\mathrm{c} \cap Z^\mathrm{c}(\qp))}}
    \bigg\backslash \mathcal{F} \biggr].
  \end{align*}
  Because we have $U_K = [(\underline{K/(K \cap Z(\mathbb{Q})^-)}) \backslash
  U]$ and the map $K/(K \cap \mathsf{Z}(\mathbb{Q})^-) \to K_p^\mathrm{c} /
  (K_p^\mathrm{c} \cap Z^\mathrm{c}(\qp))$ is surjective we conclude that
  $\mathcal{Q}_K$ is the quotient of $\mathcal{Q}$ by
  \begin{align*}
    \frac{K \times K_p^\mathrm{c}}{(\id \times
    \Delta^\mathrm{c})(K \cap \mathsf{Z}(\mathbb{Q})^-)} &= \frac{K}{K \cap
    \mathsf{Z}(\mathbb{Q})^-} \times_{\frac{K_p^\mathrm{c}}{K_p^\mathrm{c} \cap
    Z^\mathrm{c}(\qp)}} \frac{K_p^\mathrm{c} \times
    K_p^\mathrm{c}}{\Delta(K_p^\mathrm{c} \cap Z^\mathrm{c}(\qp))} \\ &\subseteq
    \frac{\mathsf{G}(\af)}{\mathsf{Z}(\mathbb{Q})^-}
    \times_{\frac{G^\mathrm{c}(\qp)}{Z^\mathrm{c}(\qp)}} \frac{G^\mathrm{c}(\qp)
    \times G^\mathrm{c}(\qp)}{\Delta(Z^\mathrm{c}(\qp))} = \frac{\mathsf{G}(\af)
    \times G^\mathrm{c}(\qp)}{(\id \times
    \Delta^\mathrm{c})(\mathsf{Z}(\mathbb{Q})^-)}
  \end{align*}
  as desired.
\end{proof}

\paragraph
The open substack $\mathcal{Q}^0 \subseteq \mathcal{Q}$ defined in
\Cref{Par:QZero} is stable under the action of $\underline{K} \times
\underline{K_p^\mathrm{c}}$, because $K_p^\mathrm{c}$ is the image of $K_p$. It
follows from \Cref{Lem:QFinQuotOfQ} that $\mathcal{Q}^0$ descends to an
open subsheaf
\[
  \mathcal{Q}_K^0 = \biggl[ \underline{\frac{K \times K_p^\mathrm{c}}{(\id
  \times \Delta^\mathrm{c})(K \cap \mathsf{Z}(\mathbb{Q})^-)}} \bigg\backslash
  \mathcal{Q}^0 \biggr] \subseteq \mathcal{Q}_K.
\]
On the other hand, the action on $\mathcal{Q}^0 \cong [(\{1\} \times
\underline{Z_\mathrm{ac}(\qp)}) \backslash \mathcal{R}]$ is induced by the
action on $\mathcal{R}$ from the short exact sequence
\[
  1 \to \{1\} \times Z_\mathrm{ac}(\qp) \to \frac{\mathsf{G}(\af) \times
  G(\qp)}{(\id \times \Delta)(\mathsf{Z}(\mathbb{Q})^-)} \to
  \frac{\mathsf{G}(\af) \times \im(G(\qp) \to G^\mathrm{c}(\qp))}{(\id \times
  \Delta^\mathrm{c})(\mathsf{Z}(\mathbb{Q})^-)} \to 1.
\]
It follows that we can describe $\mathcal{Q}_K^0$ as the quotient
\[
  \mathcal{Q}_K^0 \cong \biggl[ \underline{\frac{K \times K_p
  Z_\mathrm{ac}(\qp)}{(\id \times \Delta)(K \cap \mathsf{Z}(\mathbb{Q})^-)}}
  \bigg\backslash \mathcal{R} \biggr].
\]

\paragraph \label{Par:RFinite}
For $K = K^p K_p \subseteq \mathsf{G}(\af)$ a neat compact subgroup, we define
\[
  \mathcal{R}_K = \biggl[ \underline{\frac{K \times K_p}{(\id \times \Delta)(K
  \cap \mathsf{Z}(\mathbb{Q})^-)}} \bigg\backslash \mathcal{R} \biggr].
\]
Because we have a short exact sequence
\[
  1 \to \frac{K \times K_p}{(\id \times \Delta)(K \cap
  \mathsf{Z}(\mathbb{Q})^-)} \to \frac{K \times K_p Z_\mathrm{ac}(\qp)}{(\id
  \times \Delta)(K \cap \mathsf{Z}(\mathbb{Q})^-)} \to
  \frac{Z_\mathrm{ac}(\qp)}{K_p \cap Z_\mathrm{ac}(\qp)} \to 1
\]
of profinite groups, the natural projection map $\mathcal{R}_K \to
\mathcal{Q}_K^0$ is a torsor for the discrete abelian group $Z_\mathrm{ac}(\qp)
/ (K_p \cap Z_\mathrm{ac}(\qp))$. It follows that the map $\mathcal{R}_K
\twoheadrightarrow \mathcal{Q}_K^0 \hookrightarrow \mathcal{Q}_K$ is \'{e}tale,
see \cite[Proposition~10.11(iv)]{Sch17p}.

\paragraph \label{Par:ThetaFinite}
Assume from now on that there is a global uniformization map
\[
  \Theta \colon \mathcal{R} = U \times_{\bung\m} \grg\m \to U,
\]
which is by definition equivariant for the action of
\[
  \mathrm{pr}_{13} \colon \frac{\mathsf{G}(\af) \times G(\qp)}{(\id \times
  \Delta)(\mathsf{Z}(\mathbb{Q})^-)} \to
  \frac{\mathsf{G}(\af)}{\mathsf{Z}(\mathbb{Q})^-}.
\]
Then for every neat compact open subgroup $K = K^p K_p \subseteq
\mathsf{G}(\af)$, the composition $\mathcal{R} \to U \to U_K$ is invariant for
the action of the closed subgroup $(K \times K_p) / (\id \times \Delta)(K \cap
\mathsf{Z}(\mathbb{Q})^-)$ hence descends to a map
\[
  \Theta_K \colon \mathcal{R}_K \to U_K.
\]
Conversely, given the collection of the maps $\Theta_K$, we may recover the
global uniformization map $\Theta$ by taking the limit as $K \to \{1\}$.

\paragraph \label{Par:ThetaFiniteAxiom}
Because
\[
  (\id, \pi_\mathrm{HT}) \colon U \to U \times_{\bung\m} \grg\m =
  \mathcal{R}
\]
is equivariant for $(\id, \Delta) \colon \mathsf{G}(\af) /
\mathsf{Z}(\mathbb{Q})^- \to (\mathsf{G}(\af) \times G(\qp)) / (\id \times
\Delta)(\mathsf{Z}(\mathbb{Q})^-)$ we can quotient by the images of $K \to K
\times K_p$ on both sides to obtain a map
\[
  (\id, \pi_\mathrm{HT}) \colon U_K \to \mathcal{R}_K.
\]
Then by \Cref{Def:GlobalUniformization} the diagram
\[ \begin{tikzcd}[column sep=large]
  U_K \arrow[equals]{r} \arrow{d}[']{(\id,\pi_\mathrm{HT})} & U_K
  \arrow{d}{\pi_\mathrm{HT}^\mathrm{c}} \\ \mathcal{R}_K \arrow{ru}{\Theta_K}
  \arrow{r}{\pi_\mathrm{HT}^\mathrm{c}} & \lbrack \underline{K_p^\mathrm{c}}
  \backslash \grg\c\m \rbrack
\end{tikzcd} \]
naturally commutes, where the bottom map $\pi_\mathrm{HT}^\mathrm{c}$ is the
composition
\[
  \mathcal{R}_K \to \mathcal{Q}_K = U_K \times_{\bung\c\m}
  [\underline{K_p^\mathrm{c}} \backslash \grg\c\m] \xrightarrow{\mathrm{pr}_2}
  [\underline{K_p^\mathrm{c}} \backslash \grg\c\m].
\]

\paragraph \label{Par:RFinitePt}
For $y \colon \Spec F \to \sh[K]$ a closed point, where $F/E$ is a finite
extension, there is a corresponding map $y \colon \Spd F \to \sh\d[K]$. Consider
the map
\[
  \mathcal{R}_K \to \mathcal{Q}_K = U_K \times_{\bung\c\m}
  [\underline{K_p^\mathrm{c}} \backslash \grg\c\m] \xrightarrow{\mathrm{pr}_1}
  U_K.
\]
If $y$ is contained in $U_K \subseteq \sh\d[K]$, we may consider its fiber
\[
  \mathcal{R}_{K,y} = \Spd F \times_{U_K} \mathcal{R}_K,
\]
which is \'{e}tale over
\[
  \mathcal{Q}_{K,y} = \Spd F \times_{U_K} \mathcal{Q}_K = \Spd F
  \times_{\bung\c\m} [\underline{K_p^\mathrm{c}} \backslash \grg\c\m]
\]
by the discussion in \Cref{Par:RFinite}. For the corresponding geometric
point $\bar{y} \colon \Spec \cp \to \sh[K]$, we similarly write
\[
  \mathcal{R}_{K,\bar{y}} = \Spd \cp \times_{U_K} \mathcal{R}_K, \quad
  \mathcal{Q}_{K,\bar{y}} = \Spd \cp \times_{U_K} \mathcal{Q}_K.
\]
Given a global uniformization map $\Theta$, we may restrict $\Theta_K$ to
$\mathcal{R}_{K,y}$ and obtain
\[
  \Theta_{K,y} \colon \mathcal{R}_{K,y} \to U_K, \quad \Theta_{K,\bar{y}} \colon
  \mathcal{R}_{K,\bar{y}} \to U_K.
\]

\begin{lemma} \label{Lem:RFiniteGeom}
  Let $x \colon \Spec \cp \to \sh$ be a geometric point lifting
  $\bar{y}$ to infinite level. Then the induced map
  \[
    \Spd \cp \times_{\bung\m} \grg\m \cong \Spd \cp \times_U \mathcal{R} \to
    \Spd \cp \times_{U_K} \mathcal{R}_K = \mathcal{R}_{K,\bar{y}}
  \]
  identifies $\mathcal{R}_{K,\bar{y}}$ with
  \[
    \mathcal{R}_{K,\bar{y}} \cong \Spd \cp \times_{\bung\m} [\underline{K_p}
    \backslash \grg\m].
  \]
\end{lemma}

\begin{proof}
  The maps of v-sheaves
  \[
    \Spd \cp \xrightarrow{x} U \leftarrow \mathcal{R}
  \]
  is equivariant for the action of
  \[
    1 \to \frac{K}{K \cap \mathsf{Z}(\mathbb{Q})^-}
    \xleftarrow{\alpha = \mathrm{pr}_1} \frac{K \times K_p}{(\id \times
    \Delta)(K \cap \mathsf{Z}(\mathbb{Q})^-)},
  \]
  where the map $\alpha$ is surjective. It follows that the fiber product of the
  quotients agrees with quotient of the fiber product, see the proof of
  \Cref{Lem:QFinQuotOfQ}, and hence
  \[
    \mathcal{R}_{K,\bar{y}} = \Spd \cp \times_{U_K} \mathcal{R}_K \cong
    [\underline{\ker \alpha} \backslash (\Spd \cp \times_U \mathcal{R})]
  \]
  by definition in \Cref{Par:RFinite}. On the other hand, we compute that
  $\ker \alpha = \{1\} \times K_p$.
\end{proof}

\begin{definition}
  We say that a point $x \in \sh(\cp)$ is \textdef{classical} when its image in
  $\sh[K](\cp)$ is defined over a finite extension of $E$ for some
  (equivalently, for every) neat compact open subgroup $K \subseteq
  \mathsf{G}(\af)$.
\end{definition}

\paragraph \label{Par:ClassicalPtFinite}
Instead of first passing to finite level and then restricting to a closed point,
we can restrict to a classical $\cp$-point and then pass to finite level. Given
a classical point $x \colon \Spd \cp \to U$, we can restrict a global
uniformization map $\Theta \colon U \times_{\bung\m} \grg\m \to U$ to a
$\underline{G(\qp)}$-equivariant map
\[
  \Theta_x \colon \Spd \cp \times_{\bung\m} \grg\m \to U.
\]
Since the projection map $U \to U_K$ is invariant for the action of $K_p$, its
composition with $\Theta_x$ descends to
\[
  \Theta_{x,K} \colon \Spd \cp \times_{\bung\m} [\underline{K_p} \backslash
  \grg\m] \to U_K.
\]
Let $\bar{y} \in \sh[K](\cp)$ be the image of $x$ under $\sh \to \sh[K]$. Under
the natural isomorphism $\Spd \cp \times_{\bung\m} [\underline{K_p} \backslash
\grg\m] \cong \mathcal{R}_{K,\bar{y}}$ from \Cref{Lem:RFiniteGeom}, this
map is identified with $\Theta_{K,\bar{y}} \colon \mathcal{R}_{K,\bar{y}} \to
U_K$.

}
{\section{Interlude: local Shimura varieties} \label{Sec:LocalShimura}
\def\flgmu{\mathrm{Fl}_{G,\{\mu\},E}}
\def\grgmu{\mathrm{Gr}_{G,\{\mu\},E}}

In this section, we generalize the notion of a local Shimura variety. Instead of
starting with the data of $(G, b, \{\mu\})$ where $b \in G(\qpbr)$ corresponds
to a map $\Spd \fpbar \to \bung$, hence induces a map $\Spd \breve{E} \to \Spd
\fpbar \to \bung$, we allow more general maps $\tau \colon \Spd \breve{E} \to
\bung$ that do not factor through $\Spd \breve{E} \to \Spd \fpbar$. The
resulting local Shimura varieties $\shloc{\tau}$ are isomorphic to classical
local Shimura varieties $\shloc{b}$ upon base changing to $\cp$, but not over a
finite extension of $\breve{E}$.

\begin{definition}
  Let $G/\qp$ be a connected reductive group, let $\lbrace \mu \colon
  \mathbb{G}_{m,\qpbar} \to G_{\qpbar} \rbrace$ be a geometric conjugacy class
  with field of definition $E \subset \qpbar$, let $X \to \Spd E$ be a v-stack,
  and let $f \colon X \to \bung$ be a map. We define \textdef{local Shimura
  variety at infinite level} as the v-stack
  \[
    \shloc{f} = X \times_{\bung \times \Spd E} \grg\m \to X,
  \]
  where $\mathrm{BL} \colon \grg\m \to \bung$ is the Beauville--Laszlo map from
  \Cref{Par:BeauvilleLaszlo}. As $\underline{G(\qp)}$ acts on $\grg\m$,
  it also acts on $\shloc{f}$. For $K_p \subseteq G(\qp)$ an open subgroup, we
  define the \textdef{local Shimura variety} at level $K_p$ as
  \[
    \shloc{f}[K_p] = [\underline{K_p} \backslash \shloc{f}] = X \times_{\bung
    \times \Spd E} [\underline{K_p} \backslash \grg\m] \to X,
  \]
  see \Cref{Par:BLEquivariance}.
\end{definition}

\begin{example} \label{Exam:ClassicalLocShi}
  When $X = \Spd \breve{E}$ and $b \in G(\qpbr)$ and $K_p \subseteq G(\qp)$ is
  compact open, we may consider the composition $\tau \colon X \to \Spd \fpbar
  \xrightarrow{b} \bung$. The corresponding local Shimura variety
  \[
    \shloc{b}[K_p] \to \Spd \breve{E},
  \]
  agrees with the standard definition of a local Shimura variety, e.g., the
  definition in \cite[Section~23.3]{SW20}.
\end{example}

\paragraph \label{Par:TwistorToDeRham}
Let $X \to \Spd \qp$ be a v-stack, and let $f \colon X \to \bung$ be a map of
v-stacks. For every perfectoid test object $x \colon \Spd(R^\sharp, R^{\sharp+})
\to X$ over $\Spd \qp$, the map $f$ induces a $G$-torsor $\mathscr{P}_x$ on the
sousperfectoid space $\mathcal{X}_\mathrm{FF}(R, R^+)$. Restricting to the
Cartier divisor $\Spa(R^\sharp, R^{\sharp+}) \hookrightarrow
\mathcal{X}_\mathrm{FF}(R, R^+)$, we further obtain a $G$-torsor
$\mathscr{P}_{x,\mathrm{dR}}$ on $\Spa(R^\sharp, R^{\sharp+})$. They
collectively define a torsor $D_\mathrm{dR}(f) \to X$ under the v-topology for
the group $G^\lozenge \to \Spd \qp$.

\begin{lemma} \label{Lem:DdRComparison}
  Let $X/K$ be a smooth rigid analytic variety over $K$ a complete discretely
  valued field over $\qp$ with perfect residue field, let $G/\qp$ be a connected
  linear algebraic group, let $\mathbb{P}/X$ be a de Rham
  $\underline{G(\qp)}$-torsor, and consider the composition
  \[
    \tau \colon X^\lozenge \xrightarrow{\mathrm{DRT}(\mathbb{P})}
    [\underline{G(\qp)} \backslash \grg] \xrightarrow{\mathrm{BL}} \bung.
  \]
  Then the $G^\lozenge$-torsor $D_\mathrm{dR}(\tau) \to X^\lozenge$ is
  canonically isomorphic to the v-sheaf associated to the $G$-torsor
  $D_\mathrm{dR}(\mathbb{P}) \to X$.
\end{lemma}

\begin{proof}
  Using the Tannakian formalism, we may reduce to the case when $G =
  \mathrm{GL}_n$, and write $\mathbb{L}$ for the corresponding
  $\underline{\qp}$-local system. Because the affinoid perfectoid spaces $c(U) =
  \Spa(R^\sharp, R^{\sharp+}) \to X^\lozenge$ for $U \in
  X_\mathrm{proet}^\mathrm{ap}$ trivializing $\mathbb{L}$, $\mathbb{M}_0$,
  $D_\mathrm{dR}(\mathbb{L})$ collectively cover $X^\lozenge$, it suffices to
  show that the locally free $R^\sharp$-modules $D_\mathrm{dR}(\mathbb{L})
  \vert_{\Spa(R^\sharp, R^{\sharp+})}$ and $D_\mathrm{dR}(\tau)
  \vert_{\Spa(R^\sharp, R^{\sharp+})}$ are canonically isomorphic for such $U$.
  By \Cref{Prop:DRTOnAffPerf}, the latter is identified with
  \begin{align*}
    \mathbb{M}_0(U) \otimes_{\period{B}{dR}\p(U)} R^\sharp &= (\mathbb{M}_0(U)
    \otimes_{\period{B}{dR}\p(U)} \period\o{B}{dR}\p(U))
    \otimes_{\period\o{B}{dR}\p(U)} R^\sharp \\ &= (\mathbb{M}_0
    \otimes_{\period{B}{dR}\p} \period\o{B}{dR}\p)(U)
    \otimes_{\period\o{B}{dR}\p(U)} R^\sharp \\ &\cong
    (D_\mathrm{dR}(\mathbb{L}) \otimes_{\mathscr{O}_X} \period\o{B}{dR}\p)(U)
    \otimes_{\period\o{B}{dR}\p(U)} R^\sharp \\ &= (D_\mathrm{dR}(\mathbb{L})
    \otimes_{\mathscr{O}_X} \hat{\mathscr{O}}_X)(U),
  \end{align*}
  which is $D_\mathrm{dR}(\mathbb{L}) \vert_{\Spa(R^\sharp, R^{\sharp+})}$.
\end{proof}

\paragraph
Consider the v-sheaves
\[
  L^+G(R, R^+) = \coprod_{R \cong (R^\sharp)^\flat} G(\period{B}{dR}\p(R^\sharp,
  R^{\sharp+})), \quad LG(R, R^+) = \coprod_{R \cong (R^\sharp)^\flat}
  G(\period{B}{dR}(R^\sharp, R^{\sharp+}))
\]
where $L^+G, LG \to \Spd \qp$ are naturally group objects. Associated to the map
$f \colon X \to \bung$ is a $L^+G$-torsor $\mathscr{T}_f \to X^\lozenge$
parametrizing trivializations $G \times \period{B}{dR}\p(R^\sharp, R^{\sharp+})
\cong \mathscr{P} \vert_{\period{B}{dR}\p(R^\sharp, R^{\sharp+})}$, where
$\mathscr{P}$ is the $G$-torsor on $\mathcal{X}_\mathrm{FF}(R, R^+)$
corresponding to $\Spa(R^\sharp, R^{\sharp^+}) \to X^\lozenge \to \bung$, as in
\Cref{Par:TwistorToDeRham}. There
is a natural group homomorphism $L^+G \to G^\lozenge$ induced by $\theta \colon
\period{B}{dR}\p(R^\sharp, R^{\sharp+}) \to R^\sharp$, and by construction we
have an natural isomorphism
\[
  G^\lozenge \times^{L^+G} \mathscr{T}_f \cong D_\mathrm{dR}(f)
\]
of $G^\lozenge$-torsors over $X$.

\paragraph \label{Par:GMPeriod}
We now assume that $\{\mu\}$ is minuscule and let $X \to \Spd E$ be a v-stack.
Using \cite[Theorem~8.5.12]{KL15}, \cite[Theorem~22.5.2]{SW20}, and following
the argument of \cite[Proposition~22.3.1]{SW20}, we identify the $(R,
R^+)$-points of $\shloc{f}[G(\qp)] = [\underline{G(\qp)} \backslash \shloc{f}]$
with the following data:
\begin{itemize}
  \item an untilt $\iota \colon R \cong (R^\sharp)^\flat$,
  \item a map $\Spa(R^\sharp, R^{\sharp+}) \to X$, which gives rise to a
    $G$-torsor $\mathscr{P} / \mathcal{X}_\mathrm{FF}(R, R^+)$,
  \item a $G$-torsor $\mathscr{Q} / \mathcal{X}_\mathrm{FF}(R, R^+)$ that is
    trivial over every geometric point of $\Spa(R^\sharp, R^{\sharp+})$,
  \item a meromorphic isomorphism
    \[
      \alpha \colon \mathscr{P} \vert_{\mathcal{X}_\mathrm{FF}(R, R^+) \setminus
      \Spa(R^\sharp, R^{\sharp+})} \xrightarrow{\cong} \mathscr{Q}
      \vert_{\mathcal{X}_\mathrm{FF}(R, R^+) \setminus \Spa(R^\sharp,
      R^{\sharp+})}
    \]
    with relative position $\{\mu\}$ in the sense of
    \Cref{Par:RelativePosition}.
\end{itemize}
By \cite[Theorem~22.6.2]{SW20}, the natural map
\[
  \shloc{f}[G(\qp)] \hookrightarrow L^+G \backslash (\mathscr{T}_f \times_E
  \grgmu),
\]
which forgets the condition that $\mathscr{Q}$ is trivial at every geometric
point, is an open embedding. On the other hand, the natural action of $L^+G$ on
$\grgmu \cong \flgmu^\lozenge$ defined in \Cref{Par:AffineGrass} factors
through $L^+G \to G^\lozenge$. It follows that the right hand side can be
rewritten as
\[
  [G^\lozenge \backslash ((G^\lozenge \times^{L^+G} \mathscr{T}_f) \times_E
  \flgmu^\lozenge)] = [G^\lozenge \backslash (D_\mathrm{dR}(f) \times_E
  \flgmu^\lozenge)].
\]
Therefore the \textdef{Grothendieck--Messing period map}
\[
  \pi_\mathrm{GM} \colon \shloc{f}[K_p] \to \shloc{f}[G(\qp)] \hookrightarrow
  [G^\lozenge \backslash (D_\mathrm{dR}(f) \times_E \flgmu^\lozenge)]
\]
is a composition of an \'{e}tale map and an open embedding, hence is \'{e}tale.

\begin{lemma} \label{Lem:RelativeFlagVar}
  Let $G/\qp$ be a connected reductive group, let $\{ \mu \colon
  \mathbb{G}_{m,\qpbar} \to G_{\qpbar} \}$ be a geometric conjugacy class of
  minuscule cocharacters with field of definition $E \subset \qpbar$. Let $K
  \supseteq E$ be a complete nonarchimedean field and let $X/K$ be a smooth
  rigid analytic variety. For every $G$-torsor $\mathscr{P} \to X$, the quotient
  \[
    \pi \colon [G^\lozenge \backslash (\mathscr{P}^\lozenge \times_E
    \flgmu^\lozenge)] \to X^\lozenge
  \]
  is representable by a smooth rigid analytic variety $Y/K$, and moreover the
  corresponding map $\pi \colon Y \to X$ of rigid analytic varieties is smooth.
\end{lemma}

\begin{proof} 
  We note that $\flgmu$ has a $G$-equivariant ample line bundle. Indeed, we
  choose a finite Galois extension $\tilde{E}/E$ over which $G$ splits, so that
  $\{\mu\}$ is represented by a cocharacter $\mu \colon \mathbb{G}_{m,\tilde{E}}
  \to G_{\tilde{E}}$. Then we have $\mathrm{Fl}_{G,\{\mu\},\tilde{E}} =
  G_{\tilde{E}} / P_{\mu,\tilde{E}}$, and hence there exists a
  $G_{\tilde{E}}$-equivariant ample line bundle $\tilde{\mathscr{L}}$ on
  $\mathrm{Fl}_{G,\{\mu\},\tilde{E}}$, see \cite[Proposition~4.4]{Jan03} for
  example. Then we may take its norm $\mathscr{L} =
  \operatorname{Nm}_{\tilde{E}/E} \tilde{\mathscr{L}}$, which is a
  $G$-equivariant ample line bundle on $\flgmu$ by \cite[Lemma~0BD0]{Stacks}. It
  follows that the rigid analytic variety
  \[
    Y = \operatorname{Proj}_X^\mathrm{an} \biggl( \bigoplus_{n \ge 0} G
    \backslash (\mathscr{P} \times H^0(\flgmu, \mathscr{L}^{\otimes n})) \biggr)
  \]
  represents the quotient, where the relative Proj construction is discussed in
  \cite[Section~2.3]{Con06}. Smoothness of $Y \to X$ can be checked locally on
  $X$. When $X = \Spa(A, A^+)$, the rigid analytic variety $Y$ is the
  analytification of a smooth projective scheme over $\Spec A$, and hence is
  smooth over $X$.
\end{proof}

\begin{corollary} \label{Cor:LocShiRepresent}
  Let $G/\qp$ be a connected reductive group, let $\{ \mu \colon
  \mathbb{G}_{m,\qpbar} \to G_{\qpbar} \}$ be a geometric conjugacy class of
  minuscule cocharacters with field of definition $E \subset \qpbar$.
  Let $X/K$ be a smooth rigid analytic variety over a complete discretely valued
  nonarchimedean field $K \supseteq E$ with perfect residue field, let
  $\mathbb{P}/X$ be a de Rham $\underline{G(\qp)}$-torsor, and consider the
  composition
  \[
    \tau \colon X^\lozenge \xrightarrow{\mathrm{DRT}(\mathbb{P})}
    [\underline{G(\qp)} \backslash \grg] \xrightarrow{\mathrm{BL}} \bung.
  \]
  Then for every open subgroup $K_p \subseteq G(\qp)$, the local Shimura variety
  $\shloc{\tau}[K_p]$ is represented by a smooth rigid analytic variety over
  $K$, and moreover the corresponding map $\pi \colon \shloc{\tau}[K_p] \to X$
  of rigid analytic varieties is smooth.
\end{corollary}

\begin{proof}
  This follows immediately from the discussion of \Cref{Par:GMPeriod},
  \Cref{Lem:DdRComparison}, \Cref{Lem:RelativeFlagVar}, and
  \cite[Lemma~15.6]{Sch17p}.
\end{proof}

Because $\shloc{\tau}[K_p]$ is itself a smooth rigid analytic space and has a
natural $\underline{K_p}$-torsor over it, we may ask what properties this
pro-\'{e}tale torsor has. We start with a lemma.

\begin{lemma} \label{Lem:LocShiLocSys}
  In the setting of \Cref{Cor:LocShiRepresent}, consider the
  pro-\'{e}tale $K_p$-torsor $\shloc{\tau} \to Y = \shloc{\tau}[K_p]$, and let
  $\pi^\ast \mathbb{P} / Y$ be the pro-\'{e}tale $G(\qp)$-torsor pulled back
  from $X$. There exists a canonical isomorphism of $LG$-torsors
  \[
    LG \times^{\underline{K_p}} \shloc{\tau} \cong LG
    \times^{\underline{G(\qp)}} \pi^\ast \mathbb{P}
  \]
  on the pro-\'{e}tale site $Y_\mathrm{proet}$, where $LG$ denotes the
  pushforward $\gamma_\ast (LG) \in \mathsf{Shv}(Y_\mathrm{proet})$, see
  \Cref{Par:VToProet}.
\end{lemma}

\begin{remark}
  The two $\underline{G(\qp)}$-torsors $\underline{G(\qp)}
  \times^{\underline{K_p}} \shloc{\tau}$ and $\pi^\ast \mathbb{P}$ are almost
  never isomorphic, before pushing out along $\underline{G(\qp)} \to LG$.
\end{remark}

\begin{proof}
  Because $\underline{G(\qp)} \times^{\underline{K_p}} \shloc{\tau}$ is the
  pullback of the pro-\'{e}tale $G(\qp)$-torsor $\shloc{\tau} \to
  \shloc{\tau}[G(\qp)]$ along the natural projection $\shloc{\tau}[K_p] \to
  \shloc{\tau}[G(\qp)]$, it suffices to construct the isomorphism in the case of
  $K_p = G(\qp)$. We also note that both sides can be naturally viewed as the
  pushfoward $\gamma_\ast$ of the $G^\lozenge$-torsors on the v-site
  $Y_\mathrm{v}$, and therefore it suffices to construct a canonical isomorphism
  of $LG$-torsors
  \[
    LG \times^{\underline{G(\qp)}} \shloc{\tau} \cong LG
    \times^{\underline{G(\qp)}} (\mathbb{P}^\lozenge \times_{X^\lozenge}
    \shloc{\tau}[G(\qp)])
  \]
  as v-sheaves over $\shloc{\tau}[G(\qp)]$.

  This follows from unraveling both sides using the moduli interpretation given
  in \Cref{Par:GMPeriod}. Over $\shloc{\tau}[G(\qp)]$, which parametrizes tuples
  $(R^\sharp, \iota \colon R \cong (R^\sharp)^\flat, x \colon \Spa(R^\sharp,
  R^{\sharp+}) \to X, \mathscr{Q}, \alpha \colon \mathscr{P} \dashrightarrow
  \mathscr{Q})$, the left hand side parametrizes trivializations of the
  $G$-torsor $\mathscr{Q} \vert_{\period{B}{dR}(R^\sharp, R^{\sharp+})}$. On the
  other hand $\alpha$ induces an isomorphism
  \[
    \alpha \colon \mathscr{P} \vert_{\period{B}{dR}(R^\sharp, R^{\sharp+})}
    \cong \mathscr{Q} \vert_{\period{B}{dR}(R^\sharp, R^{\sharp+})},
  \]
  and therefore $LG \times^{\underline{G(\qp)}} \shloc{\tau}$ also parametrizes
  trivializations of the $G$-torsor $\mathscr{P} \vert_{\period{B}{dR}(R^\sharp,
  R^{\sharp+})}$. We now claim that, on the right hand side, the $LG$-torsor
  \[
    LG \times^{\underline{G(\qp)}} \mathbb{P}^\lozenge \to X^\lozenge
  \]
  also parametrizes trivializations of $\mathscr{P}
  \vert_{\period{B}{dR}(R^\sharp, R^{\sharp+})}$ for every $\Spa(R^\sharp,
  R^{\sharp+}) \to X^\lozenge$ inducing the $G$-torsor $\mathscr{P} /
  \mathcal{X}_\mathrm{FF}(R, R^+)$ under $\tau$. This follows from the fact that
  $\tau$ was defined as the composition
  \[
    X^\lozenge = [\underline{G(\qp)} \backslash \mathbb{P}^\lozenge] \to
    [\underline{G(\qp)} \backslash \grg] \xrightarrow{\mathrm{BL}} \bung,
  \]
  because $\mathrm{BL}$ is given by modification of the $G$-torsor at the
  divisor $\Spa(R^\sharp, R^{\sharp+})$.
\end{proof}

\begin{proposition} \label{Prop:LocShiDdR}
  In the setting of \Cref{Cor:LocShiRepresent}, consider the
  pro-\'{e}tale $\underline{G(\qp)}$-torsor $\eta / Y = \shloc{\tau}[K_p]$
  corresponding to $\underline{G(\qp)} \times^{\underline{K_p}} \shloc{\tau}$. 
  \begin{enumerate}
    \item The $\underline{G(\qp)}$-torsor $\eta$ is de Rham.
    \item There exists a canonical isomorphism $(D_\mathrm{dR}(\eta), \nabla)
      \cong \pi^\ast(D_\mathrm{dR}(\mathbb{P}), \nabla)$ of $G$-torsors over $Y$
      with integrable connection.
    \item The map $\mathrm{DRT}(\eta) \colon \underline{G(\qp)}
      \times^{\underline{K_p}} \shloc{\tau} \to \grg$ agrees with the
      composition
      \[
        \underline{G(\qp)} \times^{\underline{K_p}} \shloc{\tau}
        \xrightarrow{\id \times \mathrm{pr}_2} \underline{G(\qp)}
        \times^{\underline{K_p}} \grg\m \xrightarrow{\mathrm{act}} \grg\m,
      \]
      where we recall $\shloc{\tau} = X^\lozenge \times_{\bung \times \Spd E}
      \grg\m$.
  \end{enumerate}
\end{proposition}

\begin{proof}
  By functoriality, see \cite[Theorem~3.9(ii)]{LZ17} and
  \Cref{Prop:DRTFunctorial}, it suffices to treat the case of $K_p =
  G(\qp)$. For every representation $\rho \colon G \to \mathrm{GL}(V)$, by
  \Cref{Lem:LocShiLocSys} we have an isomorphism
  \begin{align*}
    (\underline{G(\qp)} \backslash (\shloc{\tau} \times \underline{V}))
    \otimes_{\underline{\qp}} \period\o{B}{dR} &\cong (\underline{G(\qp)}
    \backslash (\pi^\ast \mathbb{P} \times \underline{V}))
    \otimes_{\underline{\qp}} \period\o{B}{dR} \\ &\cong
    D_\mathrm{dR}(\underline{G(\qp)} \backslash (\pi^\ast \mathbb{P} \times
    \underline{V})) \otimes_{\mathscr{O}_Y} \period\o{B}{dR}
  \end{align*}
  of sheaves on $Y_\mathrm{proet}$ compatible the connection (but not
  necessarily with the filtration). We also have from
  \cite[Theorem~3.9(ii)]{LZ17} an isomorphism of vector bundles with connection
  \[
    (D_\mathrm{dR}(\underline{G(\qp)} \backslash (\pi^\ast \mathbb{P} \times
    \underline{V})), \nabla) \cong \pi^\ast (D_\mathrm{dR}(\underline{G(\qp)}
    \backslash (\mathbb{P} \times \underline{V})), \nabla).
  \]

  (1) Because $D_\mathrm{dR}(\underline{G(\qp)} \backslash (\shloc{\tau} \times
  \underline{V})) \cong \pi^\ast D_\mathrm{dR}(\underline{G(\qp)} \backslash
  (\mathbb{P} \times \underline{V}))$ is a vector bundle of rank $\dim_{\qp} V$,
  the $\underline{\qp}$-local system $\underline{G(\qp)} \backslash
  (\shloc{\tau} \times \underline{V})$ is de Rham. This can be seen for example
  by combining parts (ii) and (iv) of \cite[Theorem~3.9]{LZ17}.

  (2) This follows immediately from the isomorphisms above, which are compatible
  with connections.

  (3) Write $\pi_\mathrm{proet} \colon \mathsf{Shv}(Y_\mathrm{proet}) \to
  \mathsf{Shv}(X_\mathrm{proet})$ for morphism of topoi induced by the smooth
  map $\pi \colon Y \to X$. On $X_\mathrm{proet}$ we have the
  $\period{B}{dR}\p$-local system
  \[
    \mathbb{M}_0 = (D_\mathrm{dR}(\underline{G(\qp)} \backslash (\mathbb{P}
    \times \underline{V})) \otimes_{\mathscr{O}_X}
    \period\o{B}{dR}\p)^{\nabla=0} \subseteq (\underline{G(\qp)} \backslash
    (\mathbb{P} \times \underline{V})) \otimes_{\underline{\qp}} \period{B}{dR}.
  \]
  For any vector bundle with integrable connection $(\mathscr{V}, \nabla)$,
  pulling back along $\pi$ gives an identification
  \[
    \pi_\mathrm{proet}^{-1}(\mathscr{V} \otimes_{\mathscr{O}_X}
    \period\o{B}{dR}\p) \otimes_{\pi_\mathrm{proet}^{-1} \period\o{B}{dR}\p}
    \period\o{B}{dR}\p = \pi_\mathrm{proet}^{-1} \mathscr{V}
    \otimes_{\pi_\mathrm{proet}^{-1} \mathscr{O}_ X} \period\o{B}{dR}\p =
    \pi^\ast \mathscr{V} \otimes_{\mathscr{O}_Y} \period\o{B}{dR}\p,
  \]
  see \cite[Lemma~03EL]{Stacks}, which preserves the connection. It follows from
  \cite[Theorem~7.2]{Sch13} that
  \[
    \pi_\mathrm{proet}^{-1} \mathbb{M}_0 \otimes_{\pi_\mathrm{proet}^{-1}
    \period{B}{dR}\p} \period{B}{dR}\p = (D_\mathrm{dR}(\underline{G(\qp)}
    \backslash (\pi^\ast \mathbb{P} \times V)) \otimes_{\mathscr{O}_Y}
    \period\o{B}{dR}\p)^{\nabla=0}
  \]
  as subsheaves of $(\underline{G(\qp)} \backslash (\pi^\ast \mathbb{P} \times
  \underline{V})) \otimes_{\underline{\qp}} \period{B}{dR}$. Following the
  isomorphism constructed in \Cref{Lem:LocShiLocSys}, we see that the inclusion
  \[
    \pi_\mathrm{proet}^{-1} \mathbb{M}_0 \otimes_{\pi_\mathrm{proet}^{-1}
    \period{B}{dR}\p} \period{B}{dR}\p \hookrightarrow (\underline{G(\qp)}
    \backslash (\shloc{\tau} \times \underline{V})) \otimes_{\underline{\qp}}
    \period{B}{dR}
  \]
  is the tautological $\period{B}{dR}\p$-lattice. Here, if $Y =
  \shloc{\tau}[G(\qp)]$ parametrizes rational maps $\mathscr{P} \dashrightarrow
  \mathscr{Q}$ where $\mathscr{Q}$ is trivial at geometric points, by the
  tautological $\period{B}{dR}\p$-lattice we mean
  \[
    (G \backslash (\mathscr{P} \times V)) \vert_{\period{B}{dR}\p} \subseteq (G
    \backslash (\mathscr{P} \times V)) \vert_{\period{B}{dR}} \cong (G
    \backslash (\mathscr{Q} \times V)) \vert_{\period{B}{dR}},
  \]
  where trivializations of $\mathscr{Q}$ correspond to the
  $\underline{G(\qp)}$-local system. It now follows from the construction in
  \Cref{Par:LocSysToGr} that $\mathrm{DRT}(\eta)$ agrees with the natural
  projection $\shloc{\tau} \to \grg\m$.
\end{proof}

}
{\section{Functoriality over the diagonal} \label{Sec:SameUntilt}
\def\flgmu{\mathrm{Fl}_{G,\{\mu\},E}}
\def\flgmuc{\mathrm{Fl}_{G^\mathrm{c},\{\mu^\mathrm{c}\},E}}
\def\flgmucp{\mathrm{Fl}_{G^{\prime\mathrm{c}},\{\mu^{\prime\mathrm{c}}\},E^\prime}}

For $E/\qp$ a finite extension, there is a closed embedding of v-sheaves
\[
  \Spd E \xrightarrow{\Delta} \Spd E \times_{\Spd \fp} \Spd E.
\]
In this section, we study the local structure of a global uniformization map
$\Theta$ restricted to such a locus, i.e., the composition
\[
  \Theta^E \colon U \times_{\bung\m \times \Spd E} \grg\m \hookrightarrow U
  \times_{\bung\m} \grg\m \xrightarrow{\Theta} U,
\]
and show that $\Theta^E$ is automatically functorial with respect to morphisms
of Shimura data.

\begin{lemma} \label{Lem:TorsorConnection}
  Let $K \supseteq \qp$ be a complete nonarchimedean field and let $X/K$ be a
  smooth rigid analytic variety with a rational point $x \in X(K)$. Let $G/\qp$
  be a connected reductive group, and let $\mathscr{P}/X$ be a $G$-torsor
  together with an integrable connection $\nabla$. Then there exists an open
  neighborhood $x \in U \subseteq X$ and an isomorphism $(\mathscr{P}, \nabla)
  \cong (\mathscr{P}_x \times U, d)$ of $G$-torsors with integrable connection,
  which restricts to the identity at $x$.
\end{lemma}

\begin{proof}
  When $G = \mathrm{GL}_n$, the data $(\mathscr{P}, \nabla)$ corresponds to a
  rank $n$ vector bundle with integrable connection, and the result follows from
  integrating the connection,
  see \cite[Theorem~9.7]{Shi22} for example. We now treat the case of a general
  connected reductive group $G$. Fix a faithful representation $\rho \colon G
  \to \operatorname{GL}(V)$ together with a collection of tensors $\lbrace
  s_\alpha \rbrace \subseteq \bigoplus_{n,m} V^{\otimes n} \otimes
  (V^\vee)^{\otimes m}$ whose pointwise stabilizer is $G$, using
  \cite[Proposition~I.3.1(c)]{DMOS82}. Choose $U$ small enough so that
  $(\mathscr{V} = G \backslash (\mathscr{P} \times V), \nabla)$ is trivial on
  $U$, and moreover $U$ is connected. Fixing the unique isomorphism
  \[
    \varphi \colon (\mathscr{V}, \nabla) \xrightarrow{\cong} (U \times
    \mathscr{V}_x, d)
  \]
  extending the identity over $x$, we see that $\varphi$ maps the flat sections
  $s_{\alpha,U} \in H^0(U, \mathscr{V}^\otimes)$ to the constant sections
  $s_{\alpha,x} \in \mathscr{V}_x^\otimes$. It follows that $\mathscr{P} =
  \operatorname{Isom}((\mathscr{V}, \{s_{\alpha,U}\}), (V, \{s_\alpha\}))$
  together with $\nabla$ is identified with $\mathscr{P}_x \times U$ together
  with the constant connection $d$, as every finite-dimensional representation
  of $G$ is a direct summand of some $V^{\otimes n} \otimes (V^\vee)^{\otimes
  m}$.
\end{proof}

\paragraph \label{Par:PsiSetup}
Let $K$ be a complete discretely valued nonarchimedean field over $\qp$ with
perfect residue field, and let $X/K$ be a smooth rigid analytic variety. Let
$G/\qp$ be a connected reductive group, let $\{\mu \colon \mathbb{G}_{m,\qpbar}
\to G_{\qpbar}\}$ be a conjugacy class of a minuscule geometric cocharacter with
field of definition $E \subseteq \qpbar$, and fix an embedding $E
\hookrightarrow K$. Let $\mathbb{P}/X$ be a de Rham $G(\qp)$-torsor, and assume
that $\mathrm{DRT}(\mathbb{P}) \colon \mathbb{P} \to \grg\m \subseteq \grg$.
Consider the composition
\[
  \tau \colon X^\lozenge \xrightarrow{\mathrm{DRT}(\mathbb{P})}
  [\underline{G(\qp)} \backslash \grg\m] \xrightarrow{\mathrm{BL}} \bung\m.
\]
We then have a graph
\begin{align*}
  s \colon X^\lozenge &\to X^\lozenge \times_{\bung \times \Spd E}
  [\underline{G(\qp)} \backslash \grg\m] \\ &\qquad = \shloc{\tau}[G(\qp)]
  \subseteq [G^\lozenge \backslash (D_\mathrm{dR}(\mathbb{P}) \times
  \flgmu^\lozenge)],
\end{align*}
which is a section of the natural projection map. Note that both v-sheaves are
represented by smooth rigid analytic varieties over $K$ by
\Cref{Cor:LocShiRepresent}, hence $s$ corresponds to a map $s \colon X \to [G
\backslash (D_\mathrm{dR}(\mathbb{P}) \times \flgmu)]$ of rigid analytic
varieties by \cite[Proposition~10.2.3]{SW20}.

\paragraph \label{Par:PsiMap}
Let $x \in X(K)$ be a rational point. Using \Cref{Lem:TorsorConnection}, we
find a connected open neighborhood $x \in U \subseteq X$ together with an
isomorphism $\alpha_x \colon (D_\mathrm{dR}(\mathbb{P}), \nabla) \cong
(D_\mathrm{dR}(\mathbb{P})_x \times_K U, d)$. Then $s$ defines a section
\[
  U \to [G \backslash (D_\mathrm{dR}(\mathbb{P}) \vert_U \times \flgmu)] \cong
  [G \backslash (D_\mathrm{dR}(\mathbb{P})_x \times \flgmu)] \times_K U,
\]
which corresponds to a map
\[
  \Psi_x \colon U \to [G \backslash (D_\mathrm{dR}(\mathbb{P})_x \times \flgmu)]
\]
of smooth rigid analytic varieties. We record a functoriality result.

\begin{proposition} \label{Prop:PsiFunctoriality}
  Assume we are in the setting of \Cref{Par:PsiSetup}. Let $f \colon Y
  \to X$ be a map of smooth rigid analytic varieties over $K$, and let $y \in
  Y(K)$ be a rational point mapping to $f(y) = x \in X(K)$. Choose a connected
  open $x \in U \subseteq X$ as in \Cref{Par:PsiMap}, and then a
  connected open $y \in V \subseteq Y$ small enough as in
  \Cref{Par:PsiMap} so that $f \colon V \to U$. Then the diagram
  \[ \begin{tikzcd}
    V \arrow{r}{f} \arrow{d}{\Psi_y} & U \arrow{d}{\Psi_x} \\ \lbrack G \backslash
    (D_\mathrm{dR}(f^\ast \mathbb{P})_y \times \flgmu) \rbrack \arrow{r}{\cong}
    & \lbrack G \backslash (D_\mathrm{dR}(\mathbb{P})_x \times \flgmu) \rbrack
  \end{tikzcd} \]
  commutes.
\end{proposition}

\begin{proof}
  We have $f^\ast(D_\mathrm{dR}(\mathbb{P}), \nabla) \cong (D_\mathrm{dR}(f^\ast
  \mathbb{P}), \nabla)$ by \cite[Theorem~3.9(ii)]{LZ17}. We now consider the
  diagram
  \[ \begin{tikzcd}
    f^\ast(D_\mathrm{dR}(\mathbb{P}) \vert_U, \nabla) \vert_V \arrow{r}{\cong}
    \arrow{d}{f^\ast \alpha_x}[']{\cong} & (D_\mathrm{dR}(f^\ast \mathbb{P})
    \vert_V, \nabla) \arrow{d}{\alpha_y}[']{\cong} \\
    f^\ast(D_\mathrm{dR}(\mathbb{P})_x \times_K U, d) \vert_V \arrow{r}{\cong} &
    (D_\mathrm{dR}(f^\ast \mathbb{P})_y \times_K V, d).
  \end{tikzcd} \]
  This is a diagram of isomorphisms of $G$-torsors with integrable connection
  that commutes over $y$. Because $V$ is connected, it also commutes over $V$.
  The statement now follows from the functoriality of $\mathrm{DRT}$,
  \Cref{Prop:DRTFunctorial}.
\end{proof}

\paragraph \label{Par:ThetaEFinPt}
Throughout the rest of this section, we assume that we are in the setting of
\Cref{Def:GlobalUniformization} and that we are given a global
uniformization map
\[
  \Theta \colon \mathcal{R} = U \times_{\bung\m} \grg\m \to U
\]
for $U$. We can restrict it to a map
\[
  \Theta^E \colon \mathcal{R}^E = U \times_{\bung\m \times \Spd E} \grg\m \to U,
\]
of v-sheaves over $\Spd E$. As in \Cref{Par:RFinitePt}, for a neat open
compact subgroup $K \subseteq \mathsf{G}(\af)$, a finite extension $F/E$, and a
classical point $y \in U_K(F)$, we may descend $\Theta^E$ to $\mathcal{R}^E_K =
\mathcal{R}_K \times_{\Spd E \times \Spd E} \Spd E$ and restrict to
$\mathcal{R}^E_{K,y} = \mathcal{R}_{K,y} \times_{\Spd F \times \Spd E} \Spd F$
to obtain a morphism
\[
  \Theta^E_{K,y} \colon \mathcal{R}^E_{K,y} \to U_{K,F}
\]
of v-sheaves over $\Spd F$. As in the discussion of
\Cref{Par:ThetaFiniteAxiom}, the diagram
\[ \begin{tikzcd}[column sep=large]
  \Spd F \arrow{r}{y} \arrow{d}{y_0} & U_{K,F}
  \arrow{d}{\pi_\mathrm{HT}^\mathrm{c}} \\ \mathcal{R}^E_{K,y}
  \arrow{ru}{\Theta^E_{K,y}} \arrow{r}{\pi_\mathrm{HT}^\mathrm{c}} & \lbrack
  \underline{K_p^\mathrm{c}} \backslash \grg\c\m[F] \rbrack
\end{tikzcd} \]
naturally commutes. Here $y_0$ is the natural base point induced by $(\id,
\pi_\mathrm{HT}) \colon U_K \to \mathcal{R}^E_K$, and the bottom map
$\pi_\mathrm{HT}^\mathrm{c}$ is the composition
\begin{align*}
  \mathcal{R}^E_{K,y} \to \mathcal{Q}^E_{K,y} &= \Spd F \times_{\bung\c\m \times
  \Spd E} [\underline{K_p^\mathrm{c}} \backslash \grg\c\m[E]] \\ &=
  \shloc\c{\tau_y^\mathrm{c}}[K_p^\mathrm{c}] \xrightarrow{\mathrm{pr}_ 2}
  [\underline{K_p^\mathrm{c}} \backslash \grg\c\m[F]],
\end{align*}
where $\tau_y^\mathrm{c} \colon \Spd F \to \bung\c$ is the natural map
\[
  \tau \colon \sh\d[K] \xrightarrow{\pi_\mathrm{HT}} [\underline{K_p} \backslash
  \grg\m] \xrightarrow{\mathrm{BL}} \bung
\]
restricted to $y$ and mapped under $\bung \to \bung\c$.

\begin{remark} \label{Rem:ShtukaAutomatic}
  Using \Cref{Prop:LocShiDdR}, we see that commutativity of the lower triangle
  is actually automatic. Indeed, as $\Theta^E$ is $\mathsf{G}(\af) \times
  G(\qp)$-equivariant, the natural $\underline{K_p^\mathrm{c}}$-torsor on
  $U_{K,F}$ pulls back to the $\underline{K_p^\mathrm{c}}$-torsor
  \[
    \underline{K_p^\mathrm{c}} \times^{\underline{(K \times K_p) / (\id \times
    \Delta)(K \cap \mathsf{Z}(\mathbb{Q})^-)}} \mathcal{R}^E \to \mathcal{R}_K^E
  \]
  restricted to $\mathcal{R}_{K,y}^E$, which agrees with the pullback of the
  natural $\underline{K_p^\mathrm{c}}$-torsor on $\mathcal{Q}_{K,y}^E =
  \shloc\c{\tau_y^\mathrm{c}}[K_p^\mathrm{c}]$. Then functoriality of the
  Hodge--Tate period map, \Cref{Prop:DRTFunctorial}, implies the
  desired commutativity.
\end{remark}

\paragraph \label{Par:ThetaERepresent}
By \Cref{Cor:LocShiRepresent}, the v-sheaf $\mathcal{Q}^E_{K,y}
= \shloc\c{\tau_y^\mathrm{c}}[K_p^\mathrm{c}]$ is representable by a smooth
rigid analytic variety over $F$. Because $\mathcal{R}^E_{K,y} \to
\mathcal{Q}^E_{K,y}$ is \'{e}tale as shown in \Cref{Par:RFinite}, the
v-sheaf $\mathcal{R}^E_{K,y}$ is also representable by a smooth rigid analytic
variety over $F$ by \cite[Lemma~15.6]{Sch17p}. It follows from
\cite[Proposition~10.2.3]{SW20} that $\Theta^E_{K,y} \colon \mathcal{R}^E_{K,y}
\to U_{K,F}$ corresponds uniquely to a map between smooth rigid analytic
varieties over $F$.

\paragraph \label{Par:PsiForShimura} Recall from \Cref{Par:HodgeTatePeriod} that
the composition
\[
  \tau^\mathrm{c} \colon \sh\d[K] \xrightarrow{\tau} \bung \to \bung\c
\]
agrees with the map
\[
  \sh\d[K] \xrightarrow{\mathrm{DRT}(\xi_K)} [\underline{K_p^\mathrm{c}}
  \backslash \grg\c] \xrightarrow{\mathrm{BL}} \bung\c
\]
coming from the de Rham $\underline{G^\mathrm{c}(\qp)}$-torsor $\xi_K$ on
$\sh[K]$. As in \Cref{Par:PsiMap}, we find a connected open neighborhood 
\[
  y \in V_{K,y} \subseteq U_{K,F} \subseteq \sh\ad[K][F]
\]
together with a map
\[
  \Psi_{K,y} \colon V_{K,y} \to [G^\mathrm{c} \backslash (D_\mathrm{dR}(\xi_K)_y
  \times_E \flgmuc)].
\]

\begin{proposition} \label{Prop:LocalUniformization}
  Let $W_{K,y} \subseteq \mathcal{R}^E_{K,y}$ be a connected open neighborhood
  of the base point $y_0$, small enough so that $\Theta^E_{K,y}$ maps $W_{K,y}$
  into $V_{K,y}$. Then the diagram
  \[ \begin{tikzcd} \label{Eq:LocalUniformization}
    W_{K,y} \arrow{r}{\Theta^E_{K,y}} \arrow{d} & V_{K,y} \arrow{d}{\Psi_{K,y}}
    \\ \mathcal{Q}^E_{K,y} = \shloc\c{\tau_y^\mathrm{c}}[K_p^\mathrm{c}]
    \arrow{r}{\pi_\mathrm{GM}} & \lbrack G^\mathrm{c} \backslash
    (D_\mathrm{dR}(\xi_K)_y \times_E \flgmuc) \rbrack
  \end{tikzcd} \tag{$\clubsuit$} \]
  commutes, where the bottom map $\pi_\mathrm{GM}$ is constructed in
  \Cref{Par:GMPeriod}. Moreover, upon shrinking $V_{K,y}$ and $W_{K,y}$,
  we may arrange $\Theta^E_{K,y} \colon W_{K,y} \to V_{K,y}$ to be an
  isomorphism and $\Psi_{K,y}$ to be an open embedding.
\end{proposition}

\begin{proof}
  We first verify the commutativity of the diagram. Let $\eta_K /
  \mathcal{Q}_{K,y}^E$ be the pro-\'{e}tale $\underline{K_p^\mathrm{c}}$-torsor
  corresponding to $\shloc\c{\tau_y^\mathrm{c}} \to
  \shloc\c{\tau_y^\mathrm{c}}[K_p^\mathrm{c}]$. From
  \Cref{Prop:LocShiDdR}(2) we see that there is a canonical
  trivialization
  \[
    (D_\mathrm{dR}(\eta_K), \nabla) \cong (D_\mathrm{dR}(\xi_K)_y \times_F
    \mathcal{Q}_{K,y}^E, d).
  \]
  Using this trivialization, we may carry out the construction in
  \Cref{Par:PsiMap} to obtain a map
  \[
    \Psi \colon \mathcal{Q}_{K,y}^E =
    \shloc\c{\tau_y^\mathrm{c}}[K_p^\mathrm{c}] \to [G^\mathrm{c} \backslash
    (D_\mathrm{dR}(\xi_K)_y \times_E \flgmuc)].
  \]
  On the other hand, by \Cref{Prop:LocShiDdR} and the construction in
  \Cref{Par:PsiSetup}, this is identified with the Grothendieck--Messing
  period map $\pi_\mathrm{GM}$.

  By \Cref{Prop:PsiFunctoriality}, it suffices to show that there is
  an isomorphism of $\underline{G^\mathrm{c}(\qp)}$-torsors
  \[
    (\mathcal{R}_{K,y}^E \to U_{K,F})^\ast \xi_K \cong (\mathcal{R}_{K,y}^E \to
    \mathcal{Q}_{K,y}^E)^\ast \eta_K.
  \]
  Because $\Theta^E$ is assumed to be
  $G(\qp)$-equivariant, both torsors are naturally identified with the pushout
  of $\mathcal{R}^E \to \mathcal{R}_K^E$ along
  \[
    \underline{\frac{K \times K_p}{(\id \times \Delta)(K \cap
    \mathsf{Z}(\mathbb{Q})^-)}} \to \underline{K_p^\mathrm{c}},
  \]
  restricted to $W_{K,y} \subseteq \mathcal{R}_{K,y}^E \to \mathcal{R}_K^E$,
  see \Cref{Rem:ShtukaAutomatic}.

  We now verify that by shrinking $V_{K,y}$ and $W_{K,y}$, we may arrange
  $\Theta^E_{K,y} \colon W_{K,y} \cong V_{K,y}$ to be an isomorphism and
  $\Psi_{K,y} \colon V_{K,y} \to \shloc\c{\tau_y^\mathrm{c}}[K_p^\mathrm{c}]$ to
  be an open embedding. Since all objects are smooth rigid analytic varieties
  over $F$, we may consider the map on tangent spaces
  \[
    T_{y_0} \mathcal{R}^E_{K,y} \xrightarrow{d\Theta^E_{K,y}} T_y V_{K,y}
    \xrightarrow{d\Psi_{K,y}} T_{y_0} [G^\mathrm{c} \backslash
    (D_\mathrm{dR}(\xi_K)_y \times \flgmuc)].
  \]
  Because $\mathcal{R}^E_{K,y} \to \mathcal{Q}^E_{K,y} \cong
  \shloc\c{\tau_y^\mathrm{c}}[K_p^\mathrm{c}]$ is \'{e}tale, see
  \Cref{Par:RFinite}, and $\pi_\mathrm{GM} \colon \mathcal{Q}^E_{K,y} \to
  [G^\mathrm{c} \backslash (D_\mathrm{dR}(\xi_K)_y \times \flgmuc)]$ is also
  \'{e}tale, see \Cref{Par:GMPeriod}, this composition is an isomorphism.
  Moreover, we have $\dim V_{K,y} = \dim \sh[K] = \dim
  \shloc\c{\tau_y^\mathrm{c}}[K_p]$ and hence both $d\Theta^E_{K,y}$ and
  $d\Psi_{K,y}$ are isomorphisms. Then \cite[Proposition~1.6.9(iii)]{Hub96}
  implies that $\Theta^E_{K,y}$ is \'{e}tale at $y_0$ and $\Psi_{K,y}$ is
  \'{e}tale at $y$. As in the proof of \cite[Proposition~9.6]{Shi22}, \'{e}tale
  morphisms at rational points are local isomorphisms. Therefore we may shrink
  $V_{K,y}$ and $W_{K,y}$ so that both $\Theta^E_{K,y} \colon W_{K,y} \to
  V_{K,y}$ is an isomorphism and $\Psi_{K,y} \colon V_{K,y} \to
  \shloc\c{\tau_y^\mathrm{c}}[K_p^\mathrm{c}]$ is an isomorphism onto an open
  subset.
\end{proof}

\begin{remark}
  It is unconditionally (i.e., without assuming the existence of $\Theta$)
  proven in \cite[Corollary~5.3.3]{DLLZ23} that the map $\Psi_{K,y}$ is
  \'{e}tale. This also guarantees the local existence of the map
  $\Theta_{K,y}^E$.
\end{remark}

We record some topological facts that we need to prove the functoriality result.

\begin{lemma} \label{Lem:DensityClassical}
  Let $U \subseteq \sh\d$ be an open subsheaf, and let 
  \[
    \mathcal{R}^E = U \times_{\bung\m \times \Spd E} \grg\m.
  \]
  Then the subset
  \[
    S = \bigcup_f \im(\lvert \Spd \cp \times_{\bung\m \times \Spd E} \grg\m
    \rvert \xrightarrow{f \times \id} \lvert \mathcal{R}^E \rvert) \subseteq
    \lvert \mathcal{R}^E \rvert
  \]
  is dense, where $f$ ranges over classical $\cp$-points of $U$.
\end{lemma}

\begin{proof}
  Fix a prime $\ell \neq p$. The proof of \cite[Theorem~IV.1.19]{FS21p}
  shows that the map
  \[
    [\underline{G(\qp)} \backslash \grg\m] \to \bung\m \times_{\Spd \fp} \Spd E
  \]
  is separated and $\ell$-cohomologically smooth. We now apply
  \cite[Proposition~23.11]{Sch17p} to conclude that
  \[
    Y = [(\{1\} \times \underline{G(\qp)}) \backslash \mathcal{R}^E] = U
    \times_{\bung\m \times \Spd E} [\underline{G(\qp)} \backslash \grg\m] \to U,
  \]
  is open.

  On the other hand, since $U$ is open in a limit of rigid analytic spaces,
  \cite[Lemma~11.22]{Sch17p} implies that $U$ is a locally spatial diamond. As
  $[\underline{G(\qp)} \backslash \grg\m] \to \bung\m \times \Spd E$ is
  $\ell$-cohomologically smooth, it is also representable by locally spatial
  diamonds, and hence $Y$ is a locally spatial diamond. Upon choosing a cofinal
  sequence of open compact subgroups $K_i \subseteq G(\qp)$, we observe that $U
  \times_{\bung\m \times \Spd E} [\underline{K_i} \backslash \grg\m]$ is
  \'{e}tale over $Y$ by \cite[Lemma~10.13]{Sch17p}, and hence is a locally
  spatial diamond. Applying \cite[Lemma~11.22]{Sch17p} again, we obtain a
  homeomorphism
  \begin{align*}
    \lvert \mathcal{R}^E \rvert &= \Bigl\lvert \varprojlim_i \bigl( U
    \times_{\bung\m \times \Spd E} [\underline{K_i} \backslash \grg\m] \bigr)
    \Bigr\rvert \\ &\cong \varprojlim_i \bigl\lvert U \times_{\bung\m \times
    \Spd E} [\underline{K_i} \backslash \grg\m] \bigr\rvert.
  \end{align*}
  As \'{e}tale maps are open, each map $\lvert U \times_{\bung\m \times \Spd E}
  [\underline{K_i} \backslash \grg\m] \rvert \to \lvert Y \rvert$ is open. Since
  the projection maps $\lvert \mathcal{R}^E \rvert \to \lvert U \times_{\bung\m}
  [\underline{K_i} \backslash \grg\m] \rvert$ are surjective, we conclude that
  $\lvert \mathcal{R}^E \rvert \to \lvert Y \rvert$ is open. This shows that the
  map
  \[
    \mathrm{pr}_1 \colon \mathcal{R}^E = U \times_{\bung\m \times \Spd E} \grg\m
    \to U
  \]
  is open.

  We finally show that $S \subseteq \lvert \mathcal{R}^E \rvert$ is dense. Let
  $V \subseteq \lvert \mathcal{R}^E \rvert$ be any nonempty open. As $\lvert
  \mathcal{R}^E \rvert \to \lvert U \rvert$ is open, the image $W =
  \mathrm{pr}_1(V) \subseteq U$ is nonempty open. Because $\sh\d = \varprojlim_K
  \sh\d[K]$, by \cite[Lemma~11.22]{Sch17p} there exists a level $K$ and a
  nonempty open $W_K \subseteq \sh\d[K]$ for which $W$ contains the inverse
  image of $W_K$. Since $\lvert \sh\d[K] \rvert \cong \lvert \sh\ad[K] \rvert$
  by \cite[Lemma~15.6]{Sch17p} and $\sh\ad[K]$ is a smooth rigid analytic
  variety over $E$, there exists a classical $\cp$-point in $W_K$, and we may
  lift it to a classical $\cp$-point $f \colon \Spd \cp \to W \subseteq U$. By
  \cite[Proposition~12.10]{Sch17p}, the map
  \[
    \lvert \Spd \cp \times_{\bung\m} \grg\m \rvert \to \lvert \mathrm{pr}_1
    \rvert^{-1}(\im \lvert f \rvert) \subseteq \lvert \mathcal{R}^E \rvert
  \]
  is surjective. Since $\lvert \mathrm{pr}_1 \rvert^{-1}(\im \lvert f \rvert)$
  intersects nontrivially with $V$, this gives an point in $S \cap V$.
\end{proof}

\begin{lemma} \label{Lem:TorsorOverConnected}
  Let $X / \Spa \zp$ be a strongly Noetherian connected analytic adic space, let
  $G$ be a locally profinite group, and let $\mathbb{P}/X$ be a pro-\'{e}tale
  $\underline{G}$-torsor. Then $G$ acts transitively on
  \[
    \varprojlim_{H \to \{1\}} \pi_0([\underline{H} \backslash \mathbb{P}]),
  \]
  where the limit is over compact open subgroups $H \subseteq G$.
\end{lemma}

\begin{remark}
  As $X$ is not necessarily quasi-compact and $G$ is not necessarily profinite,
  it is unclear whether $\pi_0(\mathbb{P}^\lozenge) \twoheadrightarrow
  \varprojlim_{H \to \{1\}} \pi_0([\underline{H} \backslash \mathbb{P}])$ is
  injective. Hence it is also unclear whether $G$ acts transitively on
  $\pi_0(\mathbb{P}^\lozenge)$.
\end{remark}

\begin{proof}
  We first note that connected components of strongly Noetherian analytic adic
  spaces are open, because Noetherian rings have finitely many idempotents, see
  \cite[Lemma~04MF]{Stacks} for example. As finite \'{e}tale morphisms between
  strongly Noetherian adic spaces are both open and closed by
  \cite[Lemma~1.4.5(ii)]{Hub96} and \cite[Proposition~1.7.8]{Hub96}, every
  connected component maps surjectively onto a connected component under a
  finite \'{e}tale map.

  By \cite[Lemma~10.13]{Sch17p} and \cite[Lemma~15.6]{Sch17p}, each
  $\mathbb{P}_H = [\underline{H} \backslash \mathbb{P}]$ corresponds to a
  strongly Noetherian analytic adic space, \'{e}tale over $X$. Therefore each
  connected component of $\mathbb{P}_H$ is open. We now claim that for every
  connected component $Y \subseteq \mathbb{P}_H$, the projection $Y \to X$ is
  surjective. We consider its base change
  \[
    f \colon Y \times_X \mathbb{P}_H \to \mathbb{P}_H
  \]
  along the natural projection $\mathbb{P}_H \to X$. Here $Y \times_X
  \mathbb{P}_H \subseteq \mathbb{P}_H \times_X \mathbb{P}_H$ is an open and
  closed subspace, and moreover the projection $\mathbb{P}_H \times_X
  \mathbb{P}_H \to \mathbb{P}_H$ is a disjoint union of finite \'{e}tale maps
  because each double coset $H g H$ contains $H$ with finite index. Hence $f$ is
  also a disjoint union of finite \'{e}tale maps, and so its image $\im(f)$ is a
  union of connected components. On the other hand, this image $\im(f)$ is the
  preimage of $\im(Y \to X) \subseteq X$ under the projection map $\mathbb{P}_H
  \to X$ by \cite[Proposition~12.10]{Sch17p}. Because the surjective map
  $\mathbb{P}_H \to X$ is a topological quotient map as it is \'{e}tale hence
  open, we conclude that the image $\im(Y \to X)$ is a nonempty open and closed
  subset. Therefore $Y \to X$ is surjective as $X$ is connected.

  Fix a geometric point $x \colon \Spa(C^\sharp, C^{\sharp+}) \to X$ and a lift
  $\tilde{x} \colon \Spd(C^\sharp, C^{\sharp+}) \to \mathbb{P}$. Its orbit
  defines $G$-equivariant map
  \[
    G \to \varprojlim_{H \to \{1\}} \pi_0(\mathbb{P}_H).
  \]
  It remains to see that this map is surjective. Fix a family of connected
  components $Y_H \subseteq \mathbb{P}_H$ with the property that $H^\prime
  \subseteq H$ then $Y_{H^\prime} \to Y_H$. Let $S_H \subseteq H \backslash G$
  be the subset with the property that $g \tilde{x}$ maps to $Y_H \subseteq
  \mathbb{P}_H$ if and only if $g \in H S_H$. Because $Y_H \twoheadrightarrow X$
  is surjective, we see that $S_H$ is nonempty for all $H \subseteq G$. On the
  other hand, because $\mathbb{P}_{H^\prime} \to \mathbb{P}_H$ is finite
  \'{e}tale, the map $Y_{H^\prime} \to Y_H$ is surjective, and hence
  $S_{H^\prime} \to S_H$ is also surjective with finite fibers. Therefore when
  we take the limit, the subset
  \[
    \varprojlim_{H \to \{1\}} S_H \subseteq \varprojlim_{H \to \{1\}} H
    \backslash G
  \]
  is nonempty, where a cofiltered limit of nonempty sets with finite
  surjective transition maps is nonempty by Tychonoff's theorem. Taking $g$ in
  the intersection, we see that $g$ maps to the desired element $(Y_H)_H \in
  \varprojlim_H \pi_0(\mathbb{P}_H)$.
\end{proof}

We are now ready to prove functoriality of global uniformization maps $\Theta^E$
over the diagonal $\Spd E \hookrightarrow \Spd E \times \Spd E$.

\begin{proposition} \label{Prop:SameUntilt}
  Let $\gx \to \gx'$ be a morphism of Shimura data, inducing a morphism of
  Shimura varieties $\sh \to \sh'$. Let $U \subseteq \sh\d$ and $U^\prime
  \subseteq \sh\d'$ be open subsheaves stable under
  $\underline{\mathsf{G}(\af)}$ and $\underline{\mathsf{G}^\prime(\af)}$, where
  $U$ maps to $U^\prime$. Let
  \[
    \Theta \colon U \times_{\bung\m} \grg\m \to U, \quad \Theta^\prime \colon
    U^\prime \times_{\bung\m'} \grg\m' \to U^\prime
  \]
  be global uniformization maps for $U$ and $U^\prime$. Then the diagram
  \[ \begin{tikzcd}
    U \times_{\bung\m \times \Spd E} \grg\m \arrow{r} \arrow{d}{\Theta^E} &
    U^\prime \times_{\bung\m' \times \Spd E^\prime} \grg\m'
    \arrow{d}{\Theta^{\prime E^\prime}} \\ U \arrow{r} & U^\prime
  \end{tikzcd} \]
  commutes.
\end{proposition}

\begin{proof}
  Because $U^\prime$ is a separated v-sheaf, the locus on which the two
  compositions agree is a closed v-subsheaf of $U \times_{\bung\m \times \Spd E}
  \grg\m$. By \Cref{Lem:DensityClassical}, it is enough to show for all
  classical points $x \in U(\cp)$ mapping to $x^\prime \in U^\prime(\cp)$ that
  \[ \begin{tikzcd} \label{Eq:ClasicalPt}
    \shloc{\tau_x} \arrow{r} \arrow{d}{\Theta^E_x} & \shloc'{\tau_{x^\prime}}
    \arrow{d}{\Theta^{\prime E^\prime}_{x^\prime}} \\ U_{\cp} \arrow{r} &
    U^\prime_{\cp}
  \end{tikzcd} \tag{$\dagger$} \]
  commutes. For $K \to K^\prime$ neat open compact subgroups of $\mathsf{G}(\af)
  \to \mathsf{G}^\prime(\af)$, let $\bar{y} \in U_K(\cp)$ and $\bar{y}^\prime
  \in U^\prime_{K^\prime}(\cp)$ be the images of $x$ and $x^\prime$. Writing $F$
  for the residue field of $y$, we have a diagram
  \[ \begin{tikzcd} \label{Eq:FiniteLevel}
    \mathcal{R}_{K,y} \arrow{r} \arrow{d}{\Theta^E_{K,y}} &
    \mathcal{R}_{K^\prime,y^\prime} \arrow{d}{\Theta^{\prime
    E^\prime}_{K^\prime,y^\prime}} \\ U_{K,F} \arrow{r} & U^\prime_{K^\prime,F}
  \end{tikzcd} \tag{$\ddagger$} \]
  over $\Spd F$. As discussed in \Cref{Par:ClassicalPtFinite}, the
  diagram \eqref{Eq:ClasicalPt} can be obtain by base changing
  \eqref{Eq:FiniteLevel} to $\cp$ and taking the limit as $K, K^\prime \to
  \{1\}$. Because all v-sheaves in \eqref{Eq:FiniteLevel} are representable by
  smooth rigid analytic varieties over $F$, we may use
  \cite[Proposition~10.2.3]{SW20} to view it as a diagram of smooth rigid
  analytic varieties.

  We now use \Cref{Prop:LocalUniformization} to find connected open
  neighborhoods
  \begin{align*}
    y_0 &\in W_{K,y} \subseteq \mathcal{R}^E_{K,y}, & y_0^\prime &\in
    W_{K^\prime,y^\prime} \subseteq \mathcal{R}^{E^\prime}_{K^\prime,y^\prime},
    \\ y &\in V_{K,y} \subseteq U_{K,F} \subseteq \sh\ad[K][F], & y^\prime &\in
    V_{K^\prime,y^\prime} \subseteq U_{K^\prime,F} \subseteq
    \sh\ad'[K^\prime][F]
  \end{align*}
  and shrink $W_{K,y}$ and $V_{K,y}$ if necessary so that we have a diagram
  \[ \begin{tikzcd} \label{Eq:LocalFunct}
    W_{K,y} \arrow{rrr} \arrow{dd} \arrow{dr}{\Theta^E_{K,y}}[']{\cong}
    &[-4.5em] &[-1em] &[-4.5em] W_{K^\prime,y^\prime} \arrow{dd}
    \arrow{dl}[']{\Theta^{\prime E^\prime}_{K^\prime,y^\prime}}{\cong}
    \arrow{dd} \\ & V_{K,y} \arrow{r} \arrow[hook]{dl}[']{\Psi_{K,y}} &
    V_{K^\prime,y^\prime} \arrow[hook]{dr}{\Psi_{K^\prime,y^\prime}} \\ \lbrack
    G^\mathrm{c} \backslash (D_\mathrm{dR}(\xi_K)_y \times \flgmuc) \rbrack
    \arrow{rrr} & & & \lbrack G^{\prime\mathrm{c}} \backslash
    (D_\mathrm{dR}(\xi_{K^\prime}^\prime)_{y^\prime} \times \flgmucp) \rbrack
  \end{tikzcd} \tag{$\heartsuit$} \]
  where the triangles commute by \Cref{Prop:LocalUniformization} and
  both $\Psi_{K,y}$ and $\Psi_{K^\prime,y^\prime}$ are open embeddings. By
  \Cref{Prop:PsiFunctoriality} the lower trapezoid commutes, and so
  does the outer rectangle. Because $\Psi_{K^\prime,y^\prime}$ is an injection,
  it follows that the upper trapezoid commutes as well.

  Because \eqref{Eq:FiniteLevel} is a diagram of separated smooth rigid analytic
  varieties over $\cp$, the locus 
  \[
    \mathcal{E}_{x,K} = \operatorname{eq}(\shloc{\tau_x}[K_p] \cong
    \mathcal{R}^E_{K,y} \times_{\Spa F} \Spa \cp \rightrightarrows
    U^\prime_{K^\prime,\cp}) \subseteq \shloc{\tau_x}[K_p]
  \]
  on which \eqref{Eq:FiniteLevel} commutes is a closed rigid analytic subvariety
  of $\shloc{\tau_x}[K_p]$. On the other hand, the commutativity of the
  upper trapezoid of \eqref{Eq:LocalFunct} shows that this locus contains
  $W_{K,y}$, which is an open neighborhood of $x_0$. It follows that
  $\mathcal{E}_{x,K}$ contains the connected component $Y_{K_p} \subseteq
  \shloc{\tau_x}[K_p]$ of $x_0$ by \cite[Lemma~2.1.4]{Con99}. Taking the limit,
  we see that the closed v-subsheaf
  \[
    \mathcal{E}_x = \operatorname{eq}(\shloc{\tau_x} \rightrightarrows
    U^\prime_{\cp}) \subseteq \shloc{\tau_x}
  \]
  on which \eqref{Eq:ClasicalPt} commutes contains $\varprojlim_{K_p \to \{1\}}
  Y_{K_p} \subseteq \shloc{\tau_x}$.

  On the other hand, note that rigid analytic varieties are strongly Noetherian
  and the quotient
  \[
    [\underline{G(\qp)} \backslash \shloc{\tau_x}] = \shloc{\tau_x}[G(\qp)]
    \subseteq \grg\m[\cp]
  \]
  is a connected rigid analytic variety over $\cp$ by \cite[Theorem~1.1]{GL22p}.
  Therefore by \Cref{Lem:TorsorOverConnected} the $G(\qp)$-translates of
  $\varprojlim_{K_p} Y_{K_p}$ cover all of $\shloc{\tau_x}$. Because both maps
  $\shloc{\tau_x} \to U^\prime_{\cp}$ are $\underline{G(\qp)}$-equivariant, the
  locus $\mathcal{E}_x$ is stable under $\underline{G(\qp)}$, and therefore
  $\mathcal{E}_x = \shloc{\tau_x}$ as desired.
\end{proof}

}
{\section{Interlude: stably v-complete spaces} \label{Sec:StablyVComp}
\def\hotimes{\operatorname{\hat{\otimes}}}

Given two adic spaces $X, Y \to \Spa \zp$, there is a natural map
\[
  \Hom_{\Spa \zp}(X, Y) \to \Hom_{\Spd \zp}(X^\lozenge, Y^\lozenge).
\]
The goal of this section is to provide a criterion for this map to be a
bijection. This result will be used in the following section to prove
functoriality of the global uniformization map $\Theta$ on all of $\Spd E \times
\Spd E$.

\begin{definition}[{\cite[Definition~9.6]{HK25p}}]
  Let $A$ be a Tate ring in which $p$ is topologically nilpotent. We say that
  $A$ is \textdef{v-complete} when the natural map
  \[
    A \to \varprojlim_{A \to R^\sharp} R^\sharp
  \]
  is a bijection, where the limit is over all continuous ring homomorphisms $A
  \to R^\sharp$ where $R^\sharp$ is a perfectoid Tate ring.
\end{definition}

The following lemma is implicit in the discussion of \cite[Section~9]{HK25p}.

\begin{lemma} \label{Lem:TopLimitVComp}
  Let $(A, A^+)$ be a Tate--Huber pair in which $p$ is topological nilpotent,
  and assume that $A$ is v-complete. Then $A$ is uniform, and moreover the
  natural maps
  \[
    A \to \varprojlim_{(A,A^+) \to (R^\sharp,R^{\sharp+})} R^{\sharp}, \quad A^+
    \to \varprojlim_{(A,A^+) \to (R^\sharp,R^{\sharp+})} R^{\sharp+}
  \]
  are isomorphisms of topological rings, where the right hand sides have the
  inverse limit topology.
\end{lemma}

\begin{proof}
  Uniformity of $A$ is \cite[Lemma~9.10(a)]{HK25p}. We observe that there is a
  forgetful functor $\lbrace (A, A^+) \to (R^\sharp, R^{\sharp+}) \rbrace \to
  \lbrace A \to R^\sharp \rbrace$. This functor is initial in the sense of
  \cite[Definition~09WP]{Stacks} because given $A \to R^\sharp$ we obtain $(A,
  A^+) \to (R^\sharp, R^{\sharp\circ})$, and similarly given $(A, A^+) \to
  (R_1^\sharp, R_1^{\sharp+}), (R_2^\sharp, R_2^{\sharp+})$ and $R_1^\sharp,
  R_2^\sharp \to S^\sharp$ we obtain $(R_1^\sharp, R_1^{\sharp+}), (R_2^\sharp,
  R_2^{\sharp+}) \to (S^\sharp, S^{\sharp\circ})$. This shows that
  \[
    f \colon A \xrightarrow{\cong} \varprojlim_{(A,A^+) \to
    (R^\sharp,R^{\sharp+})} R^\sharp
  \]
  as abstract rings. By \cite[Remark~9.2, Lemma~9.7]{HK25p}, there exists a
  continuous injection $A \hookrightarrow R^\sharp$ for which $A$ has the
  subspace topology, and it follows that $f$ is a homeomorphism.

  To see that $A^+ \to \varprojlim_{(A,A^+) \to (R^\sharp,R^{\sharp+})}
  R^{\sharp+}$ is bijective, it suffices to show that an element $f \in A$
  satisfies $\varphi(f) \in R^{\sharp+}$ for all $\varphi \colon (A, A^+) \to
  (R^\sharp, R^{\sharp+})$ if and only if $f \in A^+$. This can be seen for
  example by testing against all algebraic completions of residue fields of
  $\Spa(A, A^+)$ using \cite[Lemma~3.3(i)]{Hub93}. On the other hand, both sides
  have the subspace topology from $A \cong \varprojlim_{(A,A^+) \to
  (R^\sharp,R^{\sharp+})} R^\sharp$, and hence it is moreover a homeomorphism.
\end{proof}

We provide a criterion for verifying that a Tate ring is v-complete, which is
moreover stable under rational localizations, finite \'{e}tale maps, and passing
from $A$ to $A\langle T^{\pm 1} \rangle$.

\begin{proposition} \label{Prop:VCompCriter}
  Let $A$ be a Tate ring, and assume that there exists a continuous ring
  homomorphism $\qp\langle T_1^{\pm 1}, \dotsc, T_n^{\pm 1} \rangle \to A$ for
  which the completed tensor product
  \[
    \tilde{A} = A \hotimes_{\qp\langle T_1^{\pm 1}, \dotsc, T_n^{\pm 1} \rangle}
    \cp\langle T_1^{\pm 1/p^\infty}, \dotsc, T_n^{\pm 1/p^\infty}
    \rangle
  \]
  is perfectoid. Then $A$ is v-complete.
\end{proposition}

\begin{remark}
  This criterion applies to $A = \qp$ (by taking $n = 0$) but does not apply to
  $A = \qp^\mathrm{cyc}$ as $\tilde{A} = \qp^\mathrm{cyc} \hotimes_{\qp} \cp$ is
  not uniform.
\end{remark}

\begin{proof}
  Fix a sequence of finite Galois extensions $\qp \subseteq E_1 \subseteq E_2
  \subseteq \dotsb$ such that $\bigcup_i E_i = \qpbar$ and $E_k \supseteq
  \qp(\zeta_{p^k})$. We have a tower of finite \'{e}tale Galois covers
  \begin{align*}
    A &\hookrightarrow A \otimes_{\qp\langle T_1^{\pm 1}, \dotsc, T_n^{\pm 1}
    \rangle} E_1\langle T_1^{\pm 1/p}, \dotsc, T_n^{\pm 1/p} \rangle \\
    &\hookrightarrow A \otimes_{\mathbb{Q}_ p\langle T_1^{\pm 1}, \dotsc,
    T_n^{\pm 1} \rangle} E_2\langle T_1^{\pm 1/p^2}, \dotsc, T_n^{\pm 1/p^2}
    \rangle \hookrightarrow \dotsb.
  \end{align*}
  For $A_0 \subseteq A$ a ring of definition, the elements of
  \[
    A_0 \hotimes_{\zp\langle T_1^{\pm 1}, \dotsc, T_n^{\pm 1} \rangle}
    \mathcal{O}_{E_k}\langle T_1^{\pm 1/p^k}, \dotsc, T_n^{\pm 1/p^k} \rangle
  \]
  are power bounded, and because the ring $\tilde{A}$ is already uniform, the
  uniform completion of the tower is $\tilde{A}$. Hence by
  \cite[Lemma~9.7]{HK25p} it suffices to check that
  \[
    A = (\tilde{A})^{\zp^n \rtimes \Gal(\qpbar/\qp)},
  \]
  where $\zp^n \rtimes \Gal(\qpbar/\qp)$ acts on $\cp\langle T_1^{\pm
  1/p^\infty}, \dotsc, T_n^{\pm 1/p^\infty} \rangle$.

  We first note that $\tilde{A}$ can be written as a completed direct sum
  \[
    \tilde{A} = \mathop{\widehat{\bigoplus}}_{d_1, \dotsc, d_n \in
    \mathbb{Z}[p^{-1}] \cap [0,1)} (A \hotimes_{\qp\langle T_1^{\pm 1}, \dotsc,
    T_n^{\pm 1} \rangle} \cp\langle T_1^{\pm 1}, \dotsc, T_n^{\pm 1} \rangle)
    T_1^{d_1} \dotsm T_n^{d_n}.
  \]
  Since $\zp^n$ acts on each basis element $T_1^{c_1/p^k} \dotsm T_n^{c_n/p^k}$
  via the character $(\gamma_1, \dotsc, \gamma_n) \mapsto \zeta_{p^k}^{c_1
  \gamma_1 + \dotsb + c_n \gamma_n}$, an element of $\tilde{A}$ that is
  fixed under the $\zp^n$-action must have $(d_1, \dotsc, d_n)$-component equal
  to zero, unless $d_1 = \dotsb = d_n = 0$. That is, we have
  \[
    (\tilde{A})^{\zp^n} = A \hotimes_{\qp\langle T_1^{\pm 1}, \dotsc, T_n^{\pm
    1} \rangle} \cp\langle T_1^{\pm 1}, \dotsc, T_n^{\pm 1} \rangle = A
    \hotimes_{\qp} \cp.
  \]

  It now remains to show that
  \[
    (A \hotimes_{\qp} \cp)^{\Gal(\qpbar/\qp)} = A
  \]
  for every Banach $\qp$-vector space $A$. It is clear that $A$ is contained in
  the $\Gal(\qpbar/\qp)$-invariants. In the other direction, given any
  $\Gal(\qpbar/\qp)$-invariant $\alpha \in A \hotimes_{\qp} \cp$, there exists a
  topologically countably generated closed subspace $B \subseteq A$ such that
  $\alpha \in B \hotimes_{\qp} \cp$. By \cite[Proposition~2.7.2.8]{BGR84}, we
  may write $B \cong \mathop{\widehat{\bigoplus}} \qp$, so that
  \[
    (B \hotimes_{\qp} \cp)^{\Gal(\qpbar/\qp)} = \Bigl(
    \mathop{\widehat{\bigoplus}} \cp \Bigr)^{\Gal(\qpbar/\qp)} =
    \mathop{\widehat{\bigoplus}} \qp = B.
  \]
  This shows that $\alpha \in B \subseteq A$ as desired.
\end{proof}

\begin{definition}
  We say that an analytic adic space $X \to \Spa \zp$ is \textdef{stably
  v-complete} if there exists a topological basis of affinoid opens $U_i =
  \Spa(A_i, A_i^+) \subseteq X$ with the property that each ring $A_i$ is
  sheafy, Tate, and v-complete.
\end{definition}

\begin{remark}
  For $(A, A^+)$ a sheafy Tate--Huber pair, the affinoid adic space $\Spa(A,
  A^+)$ is stably v-complete if and only if all rational localizations of $A$
  are v-complete. If we have
  \[
    \Spa(B, B^+) = \bigcup_{i=1}^n \Spa(B_i, B_i^+), \quad \Spa(B_i, B_i^+) \cap
    \Spa(B_j, B_j^+) = \bigcup_{k=1}^{n_{ij}} \Spa(B_{ijk}, B_{ijk}^+)
  \]
  where all $B, B_i, B_{ijk}$ are sheafy, and $B_i, B_{ijk}$ are v-complete,
  then
  \begin{align*}
    H^0(\Spa(B, B^+)_\mathrm{v}, \mathscr{O}) &= \operatorname{eq}\biggl(
    \prod_i H^0(\Spa(B_i, B_i^+)_\mathrm{v}, \mathscr{O}) \rightrightarrows
    \prod_{i,j,k} H^0(\Spa(B_{ijk}, B_{ijk}^+)_\mathrm{v}, \mathscr{O}) \biggr)
    \\ &\cong \operatorname{eq}\biggl( \prod_i B_i \rightrightarrows
    \prod_{i,j,k} B_{ijk} \biggr) = B,
  \end{align*}
  and hence $B$ is also v-complete.
\end{remark}

\begin{remark} \label{Rem:Diamantine}
  Hansen--Kedlaya \cite[Section~11]{HK25p} introduces a notion of a diamantine
  Tate ring, which is stronger but more robust than that of a stably v-complete
  ring.
\end{remark}

\begin{lemma} \label{Lem:StaVCompFullFaith}
  Let $X, Y \to \zp$ be analytic adic spaces, where $X$ is stably v-complete.
  Then the natural map
  \[
    \Hom_{\Spa \zp}(X, Y) \to \Hom_{\Spd \zp}(X^\lozenge, Y^\lozenge)
  \]
  is a bijection.
\end{lemma}

\begin{proof}
  Because $\lvert X \rvert \cong \lvert X^\lozenge \rvert$ and $\lvert Y \rvert
  \cong \lvert Y^\lozenge \rvert$ by \cite[Lemma~15.6]{Sch17p}, the map on
  topological spaces can be recovered from $X^\lozenge \to Y^\lozenge$.
  Therefore we can reduce to the case when $Y$ is affinoid, and then
  to when $X$ is v-complete affinoid.

  Write $X = \Spa(A, A^+)$ and $Y = \Spa(B, B^+)$, where both $A$ and $B$ are
  sheafy and $A$ is moreover v-complete. Suppose we are given a map $X^\lozenge
  \to Y^\lozenge$. This is the data of associating to every map $(A, A^+) \to
  (R^\sharp, R^{\sharp+})$ to a perfectoid Huber pair, a corresponding map $(B,
  B^+) \to (R^\sharp, R^{\sharp+})$. This is then equivalent to the data of a
  continuous ring homomorphism
  \[
    B \to \varprojlim_{(A, A^+) \to (R^\sharp, R^{\sharp+})} R^\sharp,
  \]
  where the right hand side has the inverse limit topology, which moreover sends
  $B^+$ to $\varprojlim R^{\sharp+}$. By \Cref{Lem:TopLimitVComp}, this is
  the data of a map of Huber pairs
  \[
    (B, B^+) \to (A, A^+),
  \]
  which by \cite[Proposition~2.1(i)]{Hub94} corresponds to a map of adic spaces
  $\Spa(A, A^+) \to \Spa(B, B^+)$.
\end{proof}

\begin{proposition} \label{Prop:YIsStaVComp}
  Let $(R, R^+)$ be a perfectoid Tate--Huber pair in characteristic $p$, and let
  $X$ be a sousperfectoid adic space admitting a smooth morphism $X \to
  \mathcal{Y}_{(0,\infty)}(R, R^+)$ in the sense of
  \cite[Definition~IV.4.8]{FS21p}. Then $X$ is stably v-complete.
\end{proposition}

\begin{proof}
  Note that the criterion in \Cref{Prop:VCompCriter} is stable under
  rational localizations and finite \'{e}tale algebras, because completed tensor
  products distribute under rational localizations of finite \'{e}tale algebras,
  and perfectoid spaces are stable under such operations. We also observe that
  if $A$ satisfies the criterion, so does $A\langle T^{\pm 1} \rangle$, because
  if $\tilde{A}$ is perfectoid then so is $\tilde{A}\langle T^{\pm 1/p^\infty}
  \rangle$. Therefore it is enough to verify the assumptions of
  \Cref{Prop:VCompCriter} on an affinoid cover of
  $\mathcal{Y}_{(0,\infty)}(R, R^+)$ (where we note that $\Spa A\langle T
  \rangle$ can be covered by two copies of $\Spa A\langle T^{\pm 1} \rangle$).
  It follows from the proof of \cite[Proposition~II.1.1]{FS21p} that there
  exists an affinoid open cover of $\mathcal{Y}_{(0,\infty)}(R, R^+)$ whose base
  change to $\qp(p^{1/p^\infty})^\wedge$ is a perfectoid affinoid open cover of
  \[
    \mathcal{Y}_{(0,\infty)}(R, R^+) \times_{\Spa \qp} \Spa
    \qp(p^{1/p^\infty})^\wedge.
  \]
  On the other hand, because $\qp(p^{1/p^\infty})^\wedge$ is a perfectoid field,
  their base change to $\cp$ are also perfectoid by
  \cite[Proposition~3.6.11]{KL15}.
\end{proof}

\begin{remark}
  Alternatively, when $(R, R^+) = (C, \mathcal{O}_C)$ is a rank $1$ perfectoid
  field, it was pointed out by Kiran Kedlaya that the Fargues--Fontaine curve
  $\mathcal{Y}_{(0,\infty)}(C, \mathcal{O}_C)$ is covered by diamantine
  affinoids in the sense of \cite[Definition~11.1]{HK25p}, see
  \Cref{Rem:Diamantine}. For rational numbers
  $0 < r < s < \infty$, write $Y = \mathcal{Y}_{[r,s]}(C, \mathcal{O}_C)$ and
  $\period{B}{[r,s]}(C, \mathcal{O}_C) = H^0(Y, \mathscr{O}_Y)$. Then because
  $\period{B}{[r,s]}(C, \mathcal{O}_C) \hotimes_{\qp} \cp$ is perfectoid,
  $\period{B}{[r,s]}(C, \mathcal{O}_C)$ is v-complete by the argument above and
  moreover $1$-strict in the sense of \cite[Definition~9.11]{HK25p}. It follows
  that the $p$-torsion subgroup of $H^1(Y, \mathscr{O}_Y^+)$ is uniformly
  torsion by \cite[Lemma~9.19]{HK25p}. On the other hand, from the explicit
  topological description of $\lvert Y \rvert$ we see that every open cover
  has a refinement in which any nonempty triple intersection is also a double
  intersection, and thus $H^i(Y, \mathscr{O}_Y^+) = 0$ for $i \ge 2$. It
  now follows from the discussion after \cite[Definition~6.8]{HK25p} that
  $\period{B}{[r,s]}(C, \mathcal{O}_C)$ is plus-sheafy, and hence diamantine. By
  \cite[Theorem~11.18]{HK25p}, smooth algebras over $\period{B}{[r,s]}(C,
  \mathcal{O}_C)$ are diamantine as well.
\end{remark}

}
{\section{Functoriality of Igusa stacks} \label{Sec:DifferentUntilt}
\def\ebr{\breve{E}}
\def\binterval{\mathbf{B}_{[1/r,r]}}
\def\hotimes{\operatorname{\hat{\otimes}}}
\def\rkeybar{\mathcal{R}_{K,\bar{y}}^{k_E}}
\def\reybar{\mathcal{R}_{K,\bar{y}}^E}

In this section, we discuss some techniques to pass from the diagonal $\Spd E
\to \Spd E \times \Spd E$ to its connected component $\Spd E \times_{\Spd k_E}
\Spd E$. We prove a rather general fact, which states that for any smooth rigid
analytic variety $X/E$, a map of v-sheaves
\[
  \mathcal{R}_{K,y}^{k_E} \to X
\]
that lives over the projection $\mathrm{pr}_2 \colon \Spd E \times_{\Spd k_E}
\Spd E \to \Spd E$ is controlled by its restriction to $\mathcal{R}_{K,y}^E$.
This allows us to deduce the functoriality result for global uniformizations
$\Theta$ from the functoriality of $\Theta^E$.

\paragraph
Let $K = K^p K_p \subseteq \mathsf{G}(\af)$ be a neat compact open subgroup. For
a closed point $y \colon \Spd F \to U_K$ with $F/E$ a finite extension, recall
from \Cref{Par:RFinitePt} the v-sheaf
\[
  \mathcal{R}_{K,y} \to \mathcal{Q}_{K,y} = \Spd F \times_{\bung\c\m}
  [\underline{K_p^\mathrm{c}} \backslash \grg\c\m] \to \Spd F \times \Spd E.
\]
We have closed embeddings
\[
  \Spd F = \Spd F \times_{\Spd E} \Spd E \hookrightarrow \Spd F \times_{\Spd
  k_E} \Spd E \hookrightarrow \Spd F \times \Spd E,
\]
where $k_E$ is the residue field of $E$ and the second map is moreover an
inclusion of a connected component. We base change the map $\mathcal{R}_{K,y}
\to \Spd F \times \Spd E$ to these subsheaves and define $\mathcal{R}_{K,y}^E
\hookrightarrow \mathcal{R}_{K,y}^{k_E} \hookrightarrow \mathcal{R}_{K,y}$.
Similarly, we base change
\[
  \mathcal{R}_{K,\bar{y}} = \Spd \cp \times_{\Spd F} \mathcal{R}_{K,y} \to \Spd
  \cp \times_{\Spd F} (\Spd F \times \Spd E) = \Spd \cp \times \Spd E
\]
to these subsheaves to define closed subsheaves $\reybar
\hookrightarrow \rkeybar \hookrightarrow
\mathcal{R}_{K,\bar{y}}$. To summarize, we have the commutative diagram
\[ \begin{tikzcd}[row sep=small]
  \reybar \arrow[hook]{r} \arrow{d} & \rkeybar \arrow[hook]{r} \arrow{d} &
  \mathcal{R}_{K,\bar{y}} \arrow{d} \\ \mathcal{R}_{K,y}^E \arrow[hook]{r}
  \arrow{d} & \mathcal{R}_{K,y}^{k_E} \arrow[hook]{r} \arrow{d} &
  \mathcal{R}_{K,y} \arrow{d} \\ \Spd F \arrow[hook]{r} & \Spd F \times_{\Spd
  k_E} \Spd E \arrow[hook]{r} & \Spd F \times \Spd E,
\end{tikzcd} \]
of v-sheaves, where all squares are Cartesian, all horizontal arrows are closed
embeddings. We shall consider all of them as v-sheaves over $\Spd E$ using the
structure map $\mathrm{pr}_2 \colon \Spd F \times \Spd E \to \Spd E$.

\begin{proposition} \label{Prop:RFinPtRepresent}
  The v-sheaves $\mathrm{pr}_2 \colon \reybar, \rkeybar, \mathcal{R}_{K,\bar{y}}
  \to \Spd E$ are representable by strongly Noetherian sousperfectoid stably
  v-complete adic spaces over $\Spa E$. Moreover, $\mathcal{R}_{K,y}^E \to \Spd
  F$ is representable by a smooth rigid analytic variety over $F$, and hence
  $\reybar \cong \Spd \cp \times_{\Spd F} \mathcal{R}_{K,y}^E$ is representable
  by a smooth rigid analytic variety over $\cp$.
\end{proposition}

\begin{remark}
  In general, the v-sheaves $\mathcal{R}_{K,y}$ and $\mathcal{R}_{K,y}^{k_E}$
  are not representable by adic spaces.
\end{remark}

\begin{proof}
  Recall from \Cref{Lem:RFiniteGeom} that upon choosing a lift $x \in
  U(\cp)$ of $\bar{y} \in U_K(\cp)$, we obtain an isomorphism
  \[
    \mathcal{R}_{K,\bar{y}} \cong \Spd \cp \times_{\bung\m} [\underline{K_p}
    \backslash \grg\m].
  \]
  Furthermore, upon choosing a $b \in G(\qpbr)$ together with an isomorphism
  between $\tau_x \colon \Spd \cp \to \bung$ and $\Spd \cp \to \Spd \fpbar
  \xrightarrow{b} \bung$ using \cite[Theorem~5.1]{Far20}, we obtain
  \begin{align*}
    \mathcal{R}_{K,\bar{y}} \cong \coprod_{k_E \hookrightarrow \fpbar}
    \rkeybar, \quad \rkeybar &\cong \Spd \cp \times_{\bung\m \times \Spd k_E}
    [\underline{K_p} \backslash \grg\m] \\ &\cong \Spd \cp \times_{\Spd \fpbar}
    \shloc{b}[K_p],
  \end{align*}
  where $\shloc{b}[K_p]$ as in \Cref{Exam:ClassicalLocShi} is a smooth
  rigid analytic variety over $\ebr$ by \cite[Section~24.1]{SW20}.

  On the other hand, the v-sheaf $\mathrm{pr}_2 \colon \Spd \cp \times_{\Spd
  \fpbar} \Spd \qpbr \to \Spd \qpbr$ is represented by the strongly Noetherian
  sousperfectoid adic space
  \[
    \mathcal{Y}_{(0,\infty)}(\cp, \mathcal{O}_{\cp}) \to \Spa \qpbr
  \]
  by \cite[Proposition~II.1.17]{FS21p} and \cite[Theorem~4.10]{Ked16}. Because
  $\shloc{b}[K_p]$ is a smooth rigid analytic variety over $\qpbr$ (as it is
  over $\ebr$), it follows from \Cref{Prop:YIsStaVComp} that the
  fiber product
  \[
    \Spd \cp \times_{\Spd \fpbar} \shloc{b}[K_p] \cong
    \mathcal{Y}_{(0,\infty)}(\cp, \mathcal{O}_{\cp}) \times_{\Spd \qpbr}
    \shloc{b}[K_p]
  \]
  is represented by a strongly Noetherian sousperfectoid stably v-complete adic
  space. It follows that $\rkeybar, \mathcal{R}_{K,\bar{y}} \to \Spd E$ are
  representable by strongly Noetherian sousperfectoid stably v-complete adic
  spaces.

  The fact that $\mathcal{R}_{K,y}^E \to \Spd F$ is representable by a smooth
  rigid analytic variety was shown in \Cref{Par:ThetaERepresent}.
\end{proof}

\paragraph
Let $E_0 = W(k_E)[p^{-1}] \subseteq E$ be the maximal subfield of $E$ unramified
over $\qp$. In view of \Cref{Lem:StaVCompFullFaith}, we regard the
v-sheaves $\reybar, \rkeybar$ as the stably v-complete adic spaces they are
represented by from \Cref{Prop:RFinPtRepresent}. We then have a
commutative diagram of adic spaces
\[ \begin{tikzcd} \label{Eq:RFinPtGalois}
  \reybar \arrow[hook]{r}{i} \arrow{d}{q} & \rkeybar \arrow{d}{p} \\ \Spa \cp
  \arrow[hook]{r}{j} & \mathcal{Y}_{(0,\infty)}(\cp, \mathcal{O}_{\cp})
  \times_{\Spa E_0} \Spa E
\end{tikzcd} \tag{$\ast$} \]
which becomes Cartesian once we pass to associated v-sheaves. On the other hand,
at the level of v-sheaves, the above diagram is the base change of the Cartesian
square
\[ \begin{tikzcd}[row sep=small]
  \mathcal{R}_{K,y}^E \arrow[hook]{r} \arrow{d} & \mathcal{R}_{K,y}^{k_E}
  \arrow{d} \\ \Spd F \arrow[hook]{r} & \Spd F \times_{\Spd k_E} \Spd E
\end{tikzcd} \]
of v-sheaves along $\Spd \cp \to \Spd F$. By applying
\Cref{Prop:RFinPtRepresent}, we see that there is a corresponding
$\Gal(\cp/F)$-action on the adic spaces $\reybar, \rkeybar$ making the diagram
\eqref{Eq:RFinPtGalois} equivariant with respect to the $\Gal(\cp/F)$-actions.

\paragraph \label{Par:DivisorXi}
There is a natural ring homomorphism
\[
  \theta_E \colon W(\mathcal{O}_{\cp^\flat}) \otimes_{\mathcal{O}_{E_0}}
  \mathcal{O}_E \to \mathcal{O}_{\cp}
\]
that $\mathcal{O}_E$-linearly extends the map $\theta \colon
W(\mathcal{O}_{\cp^\flat}) \to \mathcal{O}_{\cp}$ from \Cref{Par:BdR}.
Recall from \cite[Proposition~2.4]{Fon82} that the kernel of $\theta_E$ is
principal, and choose a generator $\xi_E \in \ker(\theta_E)$. This defines a
global section
\[
  \xi_E \in H^0(\mathcal{Y}_{(0,\infty)}(\cp, \mathcal{O}_{\cp}) \times_{\Spa
  E_0} \Spa E, \mathscr{O}_{\mathcal{Y}_{(0,\infty)}(\cp, \mathcal{O}_{\cp})
  \times_{\Spa E_0} \Spa E}).
\]

\begin{lemma} \label{Lem:CartierDiv}
  The inclusion $i \colon \reybar \to \rkeybar$ induces a short exact sequence
  \[
    0 \to \mathscr{O}_{\rkeybar} \xrightarrow{\xi_E} \mathscr{O}_{\rkeybar} \to
    i_\ast \mathscr{O}_{\reybar} \to 0
  \]
  of sheaves on $\lvert \rkeybar \rvert$, where $\xi_E$ is as in
  \Cref{Par:DivisorXi}.
\end{lemma}

\begin{proof}
  Recall from \Cref{Prop:RFinPtRepresent} that $\rkeybar$ is strongly
  Noetherian. It now suffices to find an open cover $\rkeybar = \bigcup_\alpha
  W_\alpha$ by strongly Noetherian affinoids for which $i^{-1}(W_\alpha)$ is
  also affinoid and
  \[
    0 \to \mathscr{O}_{\rkeybar}(W_\alpha) \xrightarrow{\xi_E}
    \mathscr{O}_{\rkeybar}(W_\alpha) \to \mathscr{O}_{\reybar} (i^{-1}
    (W_\alpha)) \to 0
  \]
  is exact. Indeed, by \cite[Theorem~2.5, II.1(iv)]{Hub94} every rational
  localization $\mathscr{O}_{\rkeybar}(W_\alpha) \to A$ is flat so that the base
  change
  \[
    0 \to A \xrightarrow{\xi_E} A \to \mathscr{O}_{\reybar}(i^{-1}(W_\alpha))
    \otimes_{\mathscr{O}_{\rkeybar}(W_\alpha)} A \to 0
  \]
  is exact, and it follows from \cite[Lemma~2.4]{Hub94} that the topology on the
  last term agrees with the natural topology as a finitely generated $A$-module,
  hence already complete.

  As in the proof of \Cref{Prop:RFinPtRepresent}, we choose a lift $x
  \in U(\cp)$ of $\bar{y} \in U_K(\cp)$ and an element $b \in G(\qpbr)$ together
  with an isomorphism $\mathscr{P}_b \cong \tau_x \colon \Spd \cp \to \bung$.
  Then we may identify
  \begin{align*}
    \rkeybar &\cong (\mathcal{Y}_{(0,\infty)}(\cp, \mathcal{O}_{\cp})
    \times_{\Spa E_0} \Spa E) \times_{\Spa \ebr} \shloc{b}[K_p], \\ \reybar
    &\cong \Spa \cp \times_{\Spa \ebr} \shloc{b}[K_p].
  \end{align*}
  We consider a cover $\shloc{b}[K_p] = \bigcup_\alpha W_\alpha$ for $W_\alpha =
  \Spa(A_\alpha, A_\alpha^+)$ affinoid rigid analytic varieties over $\ebr$.
  Then we have an affinoid cover
  \[
    \rkeybar = \bigcup_{\alpha, r} (\mathcal{Y}_{[1/r,r]}(\cp,
    \mathcal{O}_{\cp}) \times_{\Spa E_0} \Spa E) \times_{\Spa \ebr} W_\alpha.
  \]
  Writing $\binterval$ for the Tate ring corresponding to
  $\mathcal{Y}_{[1/r,r]}(\cp, \mathcal{O}_{\cp}) \otimes_{E_0} E$, we see that
  it now suffices to prove exactness of
  \[
    0 \to \binterval \hotimes_{\ebr} A_\alpha \xrightarrow{\xi_E} \binterval
    \hotimes_{\ebr} A_\alpha \to \cp \hotimes_{\ebr} A_\alpha \to 0.
  \]
  Because $A_\alpha$ is an affinoid $\ebr$-algebra, it is a Banach $\ebr$-vector
  space of countable type, and hence has an orthonormal Schauder basis with
  respect to some norm by \cite[Proposition~2.7.2.8]{BGR84}. As $\binterval$ is
  Noetherian and
  \[
    0 \to \binterval \xrightarrow{\xi_E} \binterval \to \cp \to 0
  \]
  is exact, the map $\binterval \to \cp$ is strict, and the result follows.
\end{proof}

\begin{lemma} \label{Lem:SheafGalInv}
  Denote by $\pi \colon \lvert \reybar \rvert \to \lvert \mathcal{R}_{K,y}^E
  \rvert$ the $\Gal(\cp/F)$-invariant continuous map. Then the natural map
  \[
    \mathscr{O}_{\mathcal{R}_{K,y}^E} \to (\pi_\ast (\mathscr{O}_{\rkeybar} /
    \xi_E \mathscr{O}_{\rkeybar}))^{\Gal(\cp/F)}
  \]
  is an isomorphism, and for every integer $n \ge 1$ the natural map
  \[
    (\pi_\ast (\mathscr{O}_{\rkeybar} / \xi_E^n
    \mathscr{O}_{\rkeybar}))^{\Gal(\cp/F)} \to (\pi_\ast (\mathscr{O}_{\rkeybar}
    / \xi_E \mathscr{O}_{\rkeybar}))^{\Gal(\cp/F)} \cong
    \mathscr{O}_{\mathcal{R}_{K,y}^E}
  \]
  is injective.
\end{lemma}

\begin{proof}
  For the first part, we see from \Cref{Lem:CartierDiv} that we can
  identify
  \[
    \mathscr{O}_{\rkeybar} / \xi_E \mathscr{O}_{\rkeybar} \cong i_\ast
    \mathscr{O}_{\reybar}.
  \]
  Because $\mathcal{R}_{K,y}^E$ is a rigid analytic variety over $F$, it
  suffices to show for every affinoid $F$-algebra $A$ of finite type that
  \[
    A \to (A \hotimes_F \cp)^{\Gal(\cp/F)}
  \]
  is an isomorphism. This is clear upon choosing an orthonormal Schauder basis
  of $A$ over $F$ using \cite[Proposition~2.7.2.8]{BGR84}, because
  $\cp^{\Gal(\cp/F)} = F$.

  Before moving on to the second statement, we note that the natural action of
  $\Gal(\cp/F)$ on $\rkeybar$ preserves the ideal sheaf $\xi_E
  \mathscr{O}_{\rkeybar}$ by \Cref{Lem:CartierDiv}, and hence also its
  powers $\xi_E^n \mathscr{O}_{\rkeybar}$. There is a short exact sequence
  \[
    0 \to \xi^{n-1} \mathscr{O}_{\rkeybar} / \xi^n \mathscr{O}_{\rkeybar} \to
    \mathscr{O}_{\rkeybar} / \xi_E^n \mathscr{O}_{\rkeybar} \to
    \mathscr{O}_{\rkeybar} / \xi_E^{n-1} \mathscr{O}_{\rkeybar} \to 0,
  \]
  and after taking $\pi_\ast$ and $\Gal(\cp/F)$-invariants we obtain a left
  exact sequence
  \begin{align*}
    0 &\to \pi_\ast(\xi_E^{n-1} \mathscr{O}_{\rkeybar} / \xi_E^n
    \mathscr{O}_{\rkeybar})^{\Gal(\cp/F)} \\ &\to
    \pi_\ast(\mathscr{O}_{\rkeybar} / \xi_E^n
    \mathscr{O}_{\rkeybar})^{\Gal(\cp/F)} \to \pi_\ast(\mathscr{O}_{\rkeybar} /
    \xi_E^{n-1} \mathscr{O}_{\rkeybar})^{\Gal(\cp/F)}.
  \end{align*}
  Therefore it suffices to show that the first term vanishes for $n \ge 2$.

  We consider the multiplication map
  \[
    i_\ast \mathscr{O}_{\reybar} \cong \mathscr{O}_{\rkeybar} /
    \xi_E \mathscr{O}_{\rkeybar} \xrightarrow{\xi_E^{n-1}} \xi_E^{n-1}
    \mathscr{O}_{\rkeybar} / \xi_E^n \mathscr{O}_{\rkeybar},
  \]
  which is an isomorphism. Because the natural map
  \[
    p \colon \rkeybar \to \mathcal{Y}_{(0,\infty)}(\cp, \mathcal{O}_{\cp})
    \times_{\Spa E_0} \Spa E
  \]
  of adic spaces is $\Gal(\cp/F)$-equivariant, the sheaf $\xi_E^{n-1}
  \mathscr{O}_{\rkeybar} / \xi_E^n \mathscr{O}_{\rkeybar}$ with the
  $\Gal(\cp/F)$-action is identified with the sheaf $\mathscr{O}_{\reybar}$ with
  the natural $\Gal(\cp/F)$-action twisted by $\chi^{n-1}$, where $\chi$ is the
  1-cocycle
  \[
    \chi \colon \Gal(\cp/F) \to \cp^\times; \quad \chi(\sigma) =
    \theta(\sigma(\xi_E) / \xi_E).
  \]
  Because $\xi_E$ and $\log([\epsilon])$ generate the same maximal ideal in
  $\period{B}{dR}\p(\cp)$, this 1-cocycle $\chi$ is cohomologous to the
  cyclotomic character. It follows that the $\Gal(\cp/F)$-action on
  \[
    \pi_\ast(\xi_E^{n-1} \mathscr{O}_{\rkeybar} / \xi_E^n
    \mathscr{O}_{\rkeybar})
  \]
  is identified with the $\Gal(\cp/F)$-action on $i_\ast \mathscr{O}_{\reybar}$
  twisted by $\chi^{n-1}$. Similarly as before, for each affinoid $\Spa(A, A^+)
  \subseteq \mathcal{R}_{K,y}^E$ we see that
  \[
    (A \hotimes_F \cp(n-1))^{\Gal(\cp/F)} = 0
  \]
  by picking an orthonormal Schauder basis for $A/F$, since
  $\cp(n-1)^{\Gal(\cp/F)} = 0$.
\end{proof}

\begin{proposition} \label{Prop:RFinFactorPt}
  Let $X$ be a smooth rigid analytic variety over $E$, and let $Z \subseteq X$
  be a closed subvariety. Let $K = K_p K^p \subseteq \mathsf{G}(\af)$ be a neat
  compact open subgroup, let $y \in U_K(F)$ be a classical point, and let $f
  \colon \mathcal{R}_{K,y}^{k_E} \to X^\lozenge$ be a map of v-sheaves over
  \[ \begin{tikzcd}
    \mathcal{R}_{K,y}^{k_E} \arrow{r}{f} \arrow{d} & X^\lozenge \arrow{d} \\
    \Spd F \times_{\Spd k_E} \Spd E \arrow{r}{\mathrm{pr}_2} & \Spd E.
  \end{tikzcd} \]
  If the restriction of $f$ to the closed subsheaf $\mathcal{R}_{K,y}^E
  \subseteq \mathcal{R}_{K,y}^{k_E}$ factors through $Z^\lozenge \hookrightarrow
  X^\lozenge$, then $f$ factors through $Z^\lozenge \hookrightarrow X^\lozenge$.
\end{proposition}

\begin{proof}
  We note that for a stably uniform Huber--Tate pair $(A, A^+)$ (for example, a
  sousperfectoid Huber pair), an element $f \in \bigcap_x \ker(\lvert -
  \rvert_x) \subseteq A$ is necessarily zero, because $\varpi^{-n} f \in A^+$
  for all $n$ by \cite[Lemma~3.3(i)]{Hub93}. Hence a map $\Spa(A, A^+) \to X$
  factors through $Z$ if and only if its image lies in $\lvert Z \rvert
  \subseteq \lvert X \rvert$. In particular, we have an isomorphism $Z^\lozenge
  \cong X^\lozenge \times_{\underline{\lvert X \rvert}} \underline{\lvert Z
  \rvert}$ of v-sheaves.

  By \Cref{Lem:StaVCompFullFaith} and
  \Cref{Prop:RFinPtRepresent}, the map $f$ between v-sheaves
  corresponds to a $\Gal(\cp/F)$-invariant map
  \[
    g \colon \rkeybar \to X
  \]
  between strongly Noetherian sousperfectoid stably v-complete adic spaces. The
  assumption that $f$ sends $\mathcal{R}_{K,y}^E$ to $Z^\lozenge$ corresponds to
  the fact that $g$ sends $\reybar$ to $Z$. That is, letting $\mathscr{I}_Z
  \subseteq \mathscr{O}_X$ be the sheaf of ideals cutting out $Z$, the
  composition
  \[
    \mathscr{I}_Z \hookrightarrow \mathscr{O}_X \to g_\ast (i_\ast
    \mathscr{O}_{\reybar})
  \]
  is zero, where $i \colon \reybar \hookrightarrow \rkeybar$ is the natural
  inclusion. We decompose the map $g \circ i$ as
  \[
    \lvert \reybar \rvert \xrightarrow{\pi} \lvert \mathcal{R}_{K,y}^E \rvert
    \xrightarrow{f} \lvert X \rvert,
  \]
  so that the corresponding composition
  \[
    f^{-1} \mathscr{I}_Z \hookrightarrow f^{-1} \mathscr{O}_X \xrightarrow{g}
    \pi_\ast \mathscr{O}_{\reybar}
  \]
  obtained by adjunction is zero.

  For an integer $n \ge 1$, we consider the map
  \[
    f^{-1} \mathscr{O}_X \to \pi_\ast (\mathscr{O}_{\rkeybar} / \xi_E^n
    \mathscr{O}_{\rkeybar})
  \]
  induced by $g$. Because $g$ is $\Gal(\cp/F)$-equivariant, its image lies in
  the Galois invariants
  \[
    (\pi_\ast(\mathscr{O}_{\rkeybar} / \xi_E^n
    \mathscr{O}_{\rkeybar}))^{\Gal(\cp/F)} \subseteq
    \pi_\ast(\mathscr{O}_{\rkeybar} / \xi_E^n \mathscr{O}_{\rkeybar}).
  \]
  On the other hand, from above we know that the composition
  \[
    f^{-1} \mathscr{I}_Z \hookrightarrow f^{-1} \mathscr{O}_X \to
    \pi_\ast(\mathscr{O}_{\rkeybar} / \xi_E^n
    \mathscr{O}_{\rkeybar})^{\Gal(\cp/F)} \hookrightarrow
    \pi_\ast(\mathscr{O}_{\reybar})^{\Gal(\cp/F)}
  \]
  is zero, where the last map is injective by \Cref{Lem:SheafGalInv}. It
  follows that $f^{-1} \mathscr{I}_Z$ maps to zero under
  $\pi_\ast(\mathscr{O}_{\rkeybar} / \xi_E^n \mathscr{O}_{\rkeybar})$, and
  therefore the composition
  \[
    \mathscr{I}_Z \hookrightarrow \mathscr{O}_X \to
    g_\ast(\mathscr{O}_{\rkeybar} / \xi_E^n \mathscr{O}_{\rkeybar})
  \]
  is zero.

  Because $\rkeybar$ is strongly Noetherian, there is a well-defined coherent
  ideal sheaf
  \[
    (\pi^{-1} f^{-1} \mathscr{I}_Z) \mathscr{O}_{\rkeybar} \subseteq
    \mathscr{O}_{\rkeybar}.
  \]
  We have seen above that this is contained in the coherent ideal sheaf $\xi_E^n
  \mathscr{O}_{\rkeybar} \subseteq \mathscr{O}_{\rkeybar}$
  for all integers $n \ge 1$. We now claim that the only coherent ideal sheaf
  contained in $\xi_E^n \mathscr{O}_{\rkeybar}$ for all $n \ge 1$ is the zero
  ideal sheaf. By \cite[Theorem~2.3.3]{KL19p}, coherent ideal sheaves correspond
  to ideals on strongly Noetherian affinoids. Hence it suffices to find a
  strongly Noetherian affinoid open cover
  \[
    \rkeybar = \bigcup_{\beta} \Spa(B_\beta, B_\beta^+)
  \]
  with the property that $\bigcap_{n \ge 1} \xi_E^n B_\beta= 0$ for all $\beta$.

  As in the proof of \Cref{Lem:CartierDiv}, we choose a lift $x \in U(\cp)$
  of $\bar{y}$ and $b \in G(\qpbr)$ so that
  \[
    \rkeybar \cong (\mathcal{Y}_{(0,\infty)}(\cp, \mathcal{O}_{\cp}) \times_{\Spa
    E_0} \Spa E) \times_{\Spa \ebr} \shloc{b}[K_p].
  \]
  Fixing an affinoid cover $\shloc{b}[K_p] = \bigcup_\alpha \Spa(A_\alpha,
  A_\alpha^+)$ and writing $\binterval$ for the ring of global sections of
  $\mathcal{Y}_{[1/r,r]}(\cp, \mathcal{O}_{\cp}) \times_{\Spa E_0} \Spa E$, we
  obtain an open cover
  \[
    \rkeybar = \bigcup_{\alpha,r} \Spa (\binterval \hotimes_{\ebr} A_\alpha).
  \]
  Because $\binterval$ is a Noetherian integral domain (even a principal ideal
  domain by \cite[Proposition~2.6.8]{Ked05}), we have $\bigcap_{n \ge 1} \xi_E^n
  \binterval = 0$ by Nakayama's lemma. It follows from choosing a Schauder basis
  of $A_\alpha$ over $\ebr$ that $\bigcap_{n \ge 1} \xi_E^n \binterval
  \hotimes_{\ebr} A_\alpha = 0$ as desired.

  This shows that $\mathscr{I}_Z \mathscr{O}_{\rkeybar} = 0$, and hence the
  composition
  \[
    \mathscr{I}_Z \hookrightarrow \mathscr{O}_X \xrightarrow{g} g_\ast
    \mathscr{O}_{\rkeybar}
  \]
  is zero. It follows that the map $g \colon \rkeybar \to X$ factors through the
  vanishing locus of $\mathscr{I}_Z$, which is $Z$. Thus $f \colon
  \mathcal{R}_{K,y}^{k_E} \to X^\lozenge$ factors through $Z^\lozenge$.
\end{proof}

\begin{remark}
  The assumption that $X$ is smooth can be relaxed to that $X$ is seminormal,
  because seminormal rigid analytic varieties are stably v-complete, see
  \cite[Theorem~8.2.3]{KL19p} or \cite[Theorem~10.3]{HK25p}.
\end{remark}

\paragraph
Recall from \Cref{Par:RFinite} the v-sheaf $\mathcal{R}_K \to \Spd E \times \Spd
E$. We denote by
\[
  \mathcal{R}_K^E \hookrightarrow \mathcal{R}_K^{k_E} \hookrightarrow
  \mathcal{R}_K
\]
its base change to $\Spd E \hookrightarrow \Spd E \times_{\Spd k_E} \Spd E
\hookrightarrow \Spd E \times \Spd E$.

\begin{corollary} \label{Cor:RFinFactor}
  Let $X$ be a smooth rigid analytic variety over $E$, and let $Z \subseteq X$
  be a closed subvariety. Let $K = K_p K^p \subseteq \mathsf{G}(\af)$ be a neat
  compact open subgroup, and let $f \colon \mathcal{R}_K^{k_E} \to X^\lozenge$
  be a map of v-sheaves over
  \[ \begin{tikzcd}
    \mathcal{R}_K^{k_E} \arrow{r}{f} \arrow{d} & X^\lozenge \arrow{d} \\ \Spd E
    \times_{\Spd k_E} \Spd E \arrow{r}{\mathrm{pr}_2} & \Spd E.
  \end{tikzcd} \]
  If the restriction of $f$ to the closed subsheaf $\mathcal{R}_K^E \subseteq
  \mathcal{R}_K^{k_E}$ factors through $Z^\lozenge \hookrightarrow X^\lozenge$,
  then $f$ factors through $Z^\lozenge \hookrightarrow X^\lozenge$.
\end{corollary}

\begin{proof}
  As in the beginning of the proof of \Cref{Prop:RFinFactorPt}, it
  suffices to show that the closed subset $f^{-1}(\lvert Z \rvert) \subseteq
  \lvert \mathcal{R}_K^{k_E} \rvert$ is the entire space. On the other hand, for
  each classical point $y \colon \Spa F \to U_K$ for $F/E$ a finite extension,
  we know from \Cref{Prop:RFinFactorPt} that
  \[
    \im(\lvert \mathcal{R}_{K,y}^{k_E} \rvert \to \lvert \mathcal{R}_K^{k_E}
    \rvert) \subseteq f^{-1}(\lvert Z \rvert).
  \]
  Let $\pi \colon \mathcal{R}_K^{k_E} \to U_K$ be the natural map. By
  \cite[Proposition~12.10]{Sch17p}, we may identify the image with $\im(\lvert
  \mathcal{R}_{K,y}^{k_E} \rvert \to \lvert \mathcal{R}_K^{k_E} \rvert) =
  \pi^{-1}(\{y\})$.

  We now argue as in \Cref{Lem:DensityClassical}. The map $\pi$ is the
  composition
  \begin{align*}
    \mathcal{R}_K^{k_E} &\to \mathcal{Q}_K^{k_E} \cong U_K \times_{\bung\c\m
    \times \Spd k_E} [\underline{K_p^\mathrm{c}} \backslash \grg\c\m] \\ &\to
    U_K \times_{\bung\c\m \times \Spd k_E} [\underline{G^\mathrm{c}(\qp)}
    \backslash \grg\c\m] \xrightarrow{\mathrm{pr}_1} U_K.
  \end{align*}
  The first two maps are \'{e}tale by the discussion in
  \Cref{Par:RFinite}, and hence open. The last map is the base change of
  a separated and $\ell$-cohomologically smooth morphism
  \[
    [\underline{G^\mathrm{c}(\qp)} \backslash \grg\c\m] \to \bung\c\m \times
    \Spd E \to \bung\c\m \times \Spd k_E,
  \]
  see \cite[Theorem~IV.1.19]{FS21p} and \cite[Proposition~24.5]{Sch17p}, and
  hence is also open by \cite[Proposition~23.11]{Sch17p}. It follows that $\pi$
  is an open morphism, so that the preimage of a dense subset is dense. Finally,
  because $U_K$ is a rigid analytic variety, its classical points are dense, and
  thus the union $\bigcup_y \pi^{-1}(\{y\}) \subseteq \lvert \mathcal{R}_K^{k_E}
  \rvert$ is dense. This shows that $f^{-1}(\lvert Z \rvert) = \lvert
  \mathcal{R}_K^{k_E} \rvert$.
\end{proof}

\begin{proposition} \label{Prop:DifferentUntilt}
  Let $\gx \to \gx'$ be a morphism of Shimura data, inducing a morphism of
  Shimura varieties $\sh \to \sh'$. Let $U \subseteq \sh\d$ and $U^\prime
  \subseteq \sh\d'$ be open subsheaves stable under
  $\underline{\mathsf{G}(\af)}$ and $\underline{\mathsf{G}^\prime(\af)}$, where
  $U$ maps to $U^\prime$. Let
  \[
    \Theta \colon U \times_{\bung\m} \grg\m \to U, \quad \Theta^\prime \colon
    U^\prime \times_{\bung\m'} \grg\m' \to U^\prime
  \]
  be global uniformization maps for $U$ and $U^\prime$. Then the diagram
  \[ \begin{tikzcd}
    U \times_{\bung\m} \grg\m \arrow{r} \arrow{d}{\Theta} & U^\prime
    \times_{\bung\m'} \grg\m' \arrow{d}{\Theta^\prime} \\ U \arrow{r} & U^\prime
  \end{tikzcd} \]
  commutes.
\end{proposition}

\begin{proof}
  Write $\mathcal{R} = U \times_{\bung\m} \grg\m$ and consider the base changes
  \[
    \mathcal{R}^E \hookrightarrow \mathcal{R}^{k_E} \hookrightarrow \mathcal{R}
  \]
  to $\Spd E \hookrightarrow \Spd E \times_{\Spd k_E} \Spd E \hookrightarrow
  \Spd E \times \Spd E$. From \Cref{Prop:SameUntilt} we see that the
  two $\underline{\mathsf{G}(\af)} \times \underline{G(\qp)}$-equivariant maps
  $\alpha, \beta \colon \mathcal{R} = U \times_{\bung\m} \grg\m \to U^\prime$ of
  the commutative diagram agree after restricting to $\mathcal{R}^E$.

  For neat compact open subgroups $K \subseteq \mathsf{G}(\af)$ and $K^\prime
  \subseteq \mathsf{G}^\prime(\af)$ satisfying $K \to K^\prime$, we have by
  \Cref{Par:ThetaFinite} a diagram
  \[ \begin{tikzcd}
    \mathcal{R}_K \arrow{r} \arrow{d}{\Theta_K} & \mathcal{R}_{K^\prime}^\prime
    \arrow{d}{\Theta_{K^\prime}^\prime} \\ U_K \arrow{r} & U_{K^\prime}^\prime
  \end{tikzcd} \]
  that commutes after restriction to $\mathcal{R}_K^E \hookrightarrow
  \mathcal{R}_K$. Because $U_K$ is a smooth separated rigid analytic variety, we
  may apply \Cref{Cor:RFinFactor} to the diagonal $U_K \hookrightarrow
  U_K \times_E U_K$ and conclude that the diagram commutes after restriction to
  $\mathcal{R}_K^{k_E} \subseteq \mathcal{R}_K$. On the other hand, both
  compositions $\alpha_K, \beta_K \colon \mathcal{R}_K \to U_{K^\prime}^\prime$
  are invariant for the action of $\phi \times \id$ on $\mathcal{R}_K$. Because
  we have
  \begin{align*}
    \mathcal{R} &= U \times_{\bung\m} \grg\m \\ &= \coprod_{n=0}^{[k_E:\fp]-1}
    (\phi \times \id)^n (U \times_{\bung\m \times \Spd k_E} \grg\m) =
    \coprod_{n=0}^{[k_E:\fp]-1} (\phi \times \id)^n \mathcal{R}^{k_E},
  \end{align*}
  taking an appropriate quotient as in \Cref{Par:RFinite} we also have
  \[
    \mathcal{R}_K = \coprod_{n=0}^{[k_E:\fp]-1} (\phi \times
    \id)^n(\mathcal{R}_K^{k_E}).
  \]
  Therefore the two maps $\alpha_K, \beta_K$ agree. Taking the limit as $K \to
  \{1\}$ and $K^\prime \to \{1\}$, we conclude that $\alpha = \beta$.
\end{proof}

\begin{theorem} \label{Thm:Functoriality}
  Let $\gx \to \gx'$ be a morphism of Shimura data, inducing a morphism of
  Shimura varieties $\sh \to \sh'$. Let $\mathsf{E} \supseteq \mathsf{E}^\prime$
  be the inclusion of reflex fields, choose a place $v \mid p$ of $\mathsf{E}$
  that induces a place $v^\prime \mid p$ of $\mathsf{E}^\prime$, and write $E =
  \mathsf{E}_v$ and $E^\prime = \mathsf{E}^\prime_{v^\prime}$. Let $U \subseteq
  \sh\d$ and $U^\prime \subseteq \sh\d'$ be open subsheaves stable under
  $\underline{\mathsf{G}(\af)}$ and $\underline{\mathsf{G}^\prime(\af)}$, where
  $U \to U^\prime$. Let
  \[
    \igsu \to \bung\m, \quad \igsu' \to \bung\m'
  \]
  be Igusa stacks for $U$ and $U^\prime$. Then there exists, uniquely up to
  unique isomorphism, a morphism of Igusa stacks
  \[
    \igsu \to \igsu'.
  \]
\end{theorem}

\begin{proof}
  This is a direct consequence of \Cref{Prop:IgusaDiscrete},
  \Cref{Prop:IgusaThetaFunct}, and
  \Cref{Prop:DifferentUntilt}.
\end{proof}

}
{\section{Igusa stacks for Shimura subdata} \label{Sec:Subdata}
\def\flgmud{\mathrm{Fl}_{G,\{\mu^{-1}\},E}^\lozenge}
\def\flgmudp{\mathrm{Fl}_{G^\prime,\{\mu^{\prime-1}\},E}^\lozenge}
\def\flgmuc{\mathrm{Fl}_{G^\mathrm{c},\{\mu^\mathrm{c}\},E}}
\def\flgmucp{\mathrm{Fl}_{G^{\prime\mathrm{c}},\{\mu^{\prime\mathrm{c}}\},E^\prime}}

Throughout this section, we let $\gx \hookrightarrow \gx'$ be an embedding of
Shimura data. Using the techniques developed for proving functoriality of Igusa
stacks, we prove that the existence of an Igusa stack for $\gx'$ implies the
existence of an Igusa stack for $\gx$. A corollary is that the Igusa stack for
Hodge type Shimura data exists not only on the good reduction locus but on the
entire Shimura variety.

\paragraph
As before, we denote by $\mathsf{E}$ and $\mathsf{E}^\prime$ the reflex fields
of $\gx$ and $\gx'$, fix a place $v \mid p$ of $\mathsf{E}$ inducing $v^\prime
\mid p$ of $\mathsf{E}^\prime$, and write $E = \mathsf{E}_v$ and $E^\prime =
\mathsf{E}^\prime_{v^\prime}$. For each compact open subgroup $K \subseteq
\mathsf{G}(\af)$, there exists an open compact subgroup $K^\prime \subseteq
\mathsf{G}^\prime(\af)$ such that $\sh[K] \to \sh'[K^\prime][E]$ is a closed
embedding by \cite[Proposition~1.15]{Del71}. Taking the limit over all $K$ and
$K^\prime$, we see that
\[
  i \colon \sh\d \to \sh\d'[][E]
\]
is a closed embedding of v-sheaves.

\begin{proposition} \label{Prop:SubdataSame}
  Let $U \subseteq \sh\d$ and $U^\prime \subseteq \sh\d'$ be open v-subsheaves
  stable under the actions of $\mathsf{G}(\af)$ and $\mathsf{G}^\prime(\af)$,
  satisfying the property that $i(U) \subseteq U_E^\prime$.
  Let $\Theta^\prime \colon U^\prime \times_{\bung\m'} \grg\m' \to U^\prime$ be
  a global uniformization for $U^\prime$. Then the composition
  \begin{align*}
    U \times_{\bung\m \times \Spd E} \grg\m &\to U^\prime \times_{\bung\m'
    \times \Spd E^\prime} \grg\m'[E] \\ &\xrightarrow{\Theta^{\prime E^\prime}}
    U^\prime_E \subseteq \sh\d'[][E]
  \end{align*}
  factors through the closed subsheaf $\sh\d$.
\end{proposition}

\begin{proof}
  We follow the proof of \Cref{Prop:SameUntilt}. Because $\sh\d \to
  \sh\d'[][E]$ is a closed embedding of v-sheaves, by
  \Cref{Lem:DensityClassical} it suffices to show the map
  \[
    \shloc{\tau_x} \to \shloc'{\tau_{x^\prime}} \xrightarrow{\Theta^{\prime
    E^\prime}_{x^\prime}} U^\prime_{\cp} \subseteq \sh\d'[][\cp]
  \]
  factors through $\sh\d[][\cp]$ for every classical point $x \in U(\cp)$
  mapping to $x^\prime \in U^\prime(\cp)$. Let
  $K \subseteq \mathsf{G}(\af)$ be a neat open compact subgroup and choose a
  neat open compact subgroup $K \subseteq K^\prime \subseteq
  \mathsf{G}^\prime(\af)$ for which $i \colon \sh[K] \to \sh'[K^\prime][E]$ is a
  closed embedding. We then obtain a map
  \[
    \shloc{\tau_x}[K_p] \to \shloc'{\tau_{x^\prime}}[K_p^\prime]
    \xrightarrow{\Theta^{\prime E^\prime}_{x^\prime,K^\prime}}
    U^\prime_{K^\prime,\cp},
  \]
  and writing $y \in U_K(F)$ for the image of $x$, it can be identified with the
  base change of
  \[
    \mathcal{R}^E_{K,y} \to \mathcal{R}^{E^\prime}_{K^\prime,y^\prime}
    \xrightarrow{\Theta^{\prime E^\prime}_{K^\prime,y^\prime}}
    U^\prime_{K^\prime,F}
  \]
  from $\Spd F$ to $\Spd \cp$ by the discussion in
  \Cref{Par:ClassicalPtFinite}.

  Using \Cref{Prop:LocalUniformization}, we find connected open
  neighborhoods
  \begin{align*}
    y_0 &\in W_{K,y} \subseteq \mathcal{R}^E_{K,y}, & y_0^\prime &\in
    W_{K^\prime,y^\prime} \subseteq \mathcal{R}^{E^\prime}_{K^\prime,y^\prime},
    \\ y &\in V_{K,y} \subseteq U_{K,F} \subseteq \sh\ad[K][F], & y^\prime &\in
    V_{K^\prime,y^\prime} \subseteq U_{K^\prime,F} \subseteq
    \sh\ad'[K^\prime][F]
  \end{align*}
  and shrink $V_{K,y}$ if necessary so that we have a diagram
  \[ \begin{tikzcd} \label{Eq:SubdataSame}
    W_{K,y} \arrow{rrr} \arrow{dd} \arrow[dashed]{rd} &[-4.5em] &[-1em]
    &[-4.5em] W_{K^\prime,y^\prime} \arrow{dd} \arrow{dl}[']{\Theta^{\prime
    E^\prime}_{K^\prime,y^\prime}}{\cong} \arrow{dd} \\ & V_{K,y}
    \arrow[hook]{r}{i} \arrow{dl}[']{\Psi_{K,y}} & V_{K^\prime,y^\prime}
    \arrow[hook]{dr}{\Psi_{K^\prime,y^\prime}} \\ \lbrack G^\mathrm{c}
    \backslash (D_\mathrm{dR}(\xi_K)_y \times \flgmuc) \rbrack \arrow{rrr} & & &
    \lbrack G^{\prime\mathrm{c}} \backslash
    (D_\mathrm{dR}(\xi_{K^\prime}^\prime)_{y^\prime} \times \flgmucp) \rbrack
  \end{tikzcd} \tag{$\diamondsuit$} \]
  where the triangle on the right commutes, $\Psi_{K^\prime,y^\prime}$ is an
  open embedding, and the outer rectangle commutes by construction.

  We now claim that the map $\Psi_{K,y}$ is \'{e}tale at $y \in V_{K,y}(F)$. We
  have a commutative diagram
  \[ \begin{tikzcd}
    T_y V_{K,y} \arrow[hook]{r}{di} \arrow{d}{d\Psi_{K,y}} & T_{y^\prime}
    V_{K^\prime,y^\prime} \arrow{d}{d\Psi_{K^\prime,y^\prime}}[']{\cong} \\
    T_{y_0} \lbrack G^\mathrm{c} \backslash (D_\mathrm{dR}(\xi_K)_y \times
    \flgmuc) \rbrack \arrow{r} & T_{y^\prime_0} \lbrack G^{\prime\mathrm{c}}
    \backslash (D_\mathrm{dR}(\xi_{K^\prime}^\prime)_{y^\prime} \times \flgmucp)
    \rbrack
  \end{tikzcd} \]
  of tangent spaces, where $\dim V_{K,y} = \dim \sh[K] = \dim \flgmuc$. Because
  $d\Psi_{K,y}$ is injective, it is an isomorphism. As both $V_{K,y}$ and
  $\lbrack G^\mathrm{c} \backslash (D_\mathrm{dR}(\xi_K)_y \times \flgmuc)
  \rbrack$ are smooth rigid analytic varieties over $F$, the map $\Psi_{K,y}$ is
  \'{e}tale at $y$ by \cite[Proposition~1.6.9(iii)]{Hub96}.

  We obtain from \eqref{Eq:SubdataSame} two pointed maps
  \[
    W_{K,y}, V_{K,y} \to Y = V_{K^\prime,y^\prime} \times_{[G^{\prime\mathrm{c}}
    \backslash (D_\mathrm{dR}(\xi_{K^\prime}^\prime)_{y^\prime} \times
    \flgmucp)]} [G^\mathrm{c} \backslash (D_\mathrm{dR}(\xi_K)_y \times
    \flgmuc)]
  \]
  of smooth rigid analytic varieties over $F$, where both $y_0, y$ are sent to
  $(y^\prime, y_0)$. As we have seen above, $V_{K,y} \to Y$ is \'{e}tale at $y$,
  and the proof of \cite[Proposition~9.6]{Shi22} shows that pointed \'{e}tale
  maps between rigid analytic varieties are local isomorphisms. Therefore, upon
  shrinking $W_{K,y}$, we may lift $W_{K,y} \to Y$ to a map $W_{K,y} \to
  V_{K,y}$ sending $y_0$ to $y$. This fills in the dashed arrow in
  \eqref{Eq:SubdataSame} to make the entire diagram commutative, and it follows
  the composition
  \[
    W_{K,y} \subseteq \mathcal{R}^E_{K,y} \to
    \mathcal{R}^{E^\prime}_{K^\prime,y^\prime}
    \xrightarrow{\Theta^{E^\prime}_{K^\prime,y^\prime}} U_{K^\prime,F} \subseteq
    \sh\ad'[K^\prime][F]
  \]
  factors through $\sh\ad[K][F]$.

  The preimage
  \[
    \mathcal{I}_{x,K} = (\shloc{\tau_x}[K_p] \cong \mathcal{R}^E_{K,y}
    \times_{\Spa F} \Spa \cp) \times_{\sh\ad'[K^\prime][\cp]}
    \sh\ad[K][\cp]
  \]
  is a closed subvariety of $\shloc{\tau_x}[K_p]$ containing an open
  neighborhood of $x_0$. It follows that $\mathcal{I}_{x,K}$ contains the entire
  connected component of $Y_{K_p} \subseteq \shloc{\tau_x}[K_p]$ by
  \cite[Lemma~2.1.4]{Con99}. Taking the limit as $K, K^\prime \to \{1\}$, we see
  that the closed subsheaf
  \[
    \mathcal{I}_x = \shloc{\tau_x} \times_{\sh\d'[][\cp]} \sh\d[][\cp] \subseteq
    \shloc{\tau_x}
  \]
  both contains the limit $\varprojlim_{K_p \to \{1\}} Y_{K_p}$ and is stable
  under the action of $G(\qp)$, because both $\Theta^{\prime E^\prime}_x$ and
  $i$ are $\underline{G(\qp)}$-equivariant. Using
  \Cref{Lem:TorsorOverConnected}, we conclude that $\mathcal{I}_x =
  \shloc{\tau_x}$ as in the proof of \Cref{Prop:SameUntilt}.
\end{proof}

\begin{proposition} \label{Prop:SubdataDifferent}
  Let $U \subseteq \sh\d$ and $U^\prime \subseteq \sh\d'$ be open v-subsheaves
  stable under the actions of $\mathsf{G}(\af)$ and $\mathsf{G}^\prime(\af)$,
  satisfying the property that $i(U) \subseteq U_E^\prime$.
  Let $\Theta^\prime \colon U^\prime \times_{\bung\m'} \grg\m' \to U^\prime$ be a
  global uniformization for $U^\prime$. Then the composition
  \begin{align*}
    U \times_{\bung\m} \grg\m &\to U^\prime \times_{\bung\m'} \grg\m'[E] \\
    &\xrightarrow{\Theta^\prime} U^\prime_E \subseteq \sh\d'[][E]
  \end{align*}
  factors through the closed subsheaf $\sh\d$.
\end{proposition}

\begin{proof}
  For $K \subseteq \mathsf{G}(\af)$ and $K^\prime \subseteq
  \mathsf{G}^\prime(\af)$ neat compact open subgroups satisfying the property
  that $i \colon \sh[K] \hookrightarrow \sh'[K^\prime][E]$ is a closed
  embedding, we have a map
  \[
    f_{K,K^\prime} \colon \mathcal{R}_K \to \mathcal{R}_{K^\prime}^\prime
    \times_{\Spd E^\prime \times \Spd E^\prime} (\Spd E^\prime \times \Spd E)
    \xrightarrow{\Theta^\prime_{K^\prime}} U^\prime_{K^\prime,E}
  \]
  We consider the base changes
  \[
    \mathcal{R}_K^E \hookrightarrow \mathcal{R}_K^{k_E} \hookrightarrow
    \mathcal{R}_K
  \]
  along $\Spd E \hookrightarrow \Spd E \times_{\Spd k_E} \Spd E \hookrightarrow
  \Spd E \times \Spd E$. By \Cref{Prop:SubdataSame}, the map
  $f_{K,K^\prime}$ sends $\mathcal{R}^E_K$ to $\sh\d[K] \subseteq
  \sh\d'[K^\prime][E]$. Then by \Cref{Cor:RFinFactor}, the map
  $f_{K,K^\prime}$ sends $\mathcal{R}^{k_E}_K$ to $\sh\d[K]$. Taking the limit
  as $K, K^\prime \to \{1\}$, the composition
  \[
    f \colon \mathcal{R} = U \times_{\bung\m} \grg\m \to \mathcal{R}^\prime
    \times_{\Spd E^\prime \times \Spd E^\prime} (\Spd E^\prime \times \Spd E)
    \xrightarrow{\Theta^\prime} U^\prime_E
  \]
  sends $\mathcal{R}^{k_E} = U \times_{\bung\m \times \Spd k_E} \grg\m$ to
  $\sh\d$. Finally, we note that $f$ is invariant for the action of $\phi \times
  \id$ because $\Theta^\prime$ is, and therefore $f$ factors through $\sh\d
  \hookrightarrow \sh\d'[][E]$.
\end{proof}

\begin{theorem} \label{Thm:IgusaSubdata}
  Let $\gx \hookrightarrow \gx'$ be an embedding of Shimura data with reflex
  fields $\mathsf{E} \supseteq \mathsf{E}^\prime$. Let $v^\prime \mid p$ be a
  place of $\mathsf{E}^\prime$, write $E^\prime = \mathsf{E}^\prime_{v^\prime}$,
  and let $U^\prime \subseteq \sh\d'$ be an open subsheaf that is
  $\underline{\mathsf{G}^\prime(\af)}$-stable. Assume that there exists a Igusa
  stack
  \[
    \igsu' \to \bung\m'
  \]
  for $U^\prime$. Then for every place $v \mid v^\prime$ of $\mathsf{E}$, there
  exists an Igusa stack
  \[
    \igsu \to \bung\m
  \]
  for $U = i^{-1}(U^\prime)$, where $E = \mathsf{E}_v$ and $i \colon \sh\d \to
  \sh\d'$ is the natural map of Shimura varieties.
\end{theorem}

\begin{proof}
  In view of \Cref{Prop:IgusaGlobalUnif}, it suffices to prove that
  there exists a global uniformization map
  \[
    \Theta \colon U \times_{\bung\m} \grg\m \to U.
  \]
  Writing $\Theta^\prime \colon U^\prime \times_{\bung\m'} \grg\m' \to U^\prime$
  for the global uniformization for $U^\prime$, we see from
  \Cref{Prop:SubdataDifferent} that $\Theta^\prime$ restricts to
  \[ \begin{tikzcd}
    U \times_{\bung\m} \grg\m \arrow{r} \arrow{d}{\Theta} & U^\prime
    \times_{\bung\m'} \grg\m'[E] \arrow{d}{\Theta^\prime} \\ U
    \arrow[hook]{r}{i} & U^\prime_E.
  \end{tikzcd} \]
  It remains to verify that $\Theta$ indeed is a global uniformization map.

  Because $i \colon U \to U^\prime$ is injective, to show that $\Theta$ is
  $\underline{\mathsf{G}(\af)} \times \underline{G(\qp)}$-equivariant and $\phi
  \times \id$-invariant, we may instead check that the composition
  \[
    U \times_{\bung\m} \grg\m \to U^\prime \times_{\bung\m'} \grg\m'[E]
    \xrightarrow{\Theta^\prime} U^\prime_E
  \]
  is equivariant and invariant. This immediately follows from the equivariance
  of the first map $U \times_{\bung\m} \grg\m \to U^\prime \times_{\bung\m'}
  \grg\m'$ together with the corresponding properties of $\Theta^\prime$.

  For the upper triangle of the first diagram and the second diagram in
  (\ref{Item:GlobalUnifDiagrams}) of \Cref{Def:GlobalUniformization},
  we again may check commutativity after post-composition with the closed
  embedding $U \hookrightarrow U^\prime$. Similarly, for the lower triangle of
  the first diagram, we use the fact that
  \[
    \grg\m \cong \flgmud \hookrightarrow \flgmudp \cong \grg\m'[E]
  \]
  is a closed embedding, so that we may check commutativity after
  post-composition with $\grg\m \hookrightarrow \grg\m'$. All of them follow
  from the corresponding properties of $\Theta^\prime$.
\end{proof}

As an application, we extend the result of \cite{DvHKZ24p} beyond the good
reduction locus.

\begin{corollary} \label{Cor:HodgeTypeBadReduction}
  Let $\gx$ be of Hodge type. Then there exist Igusa stacks
  \[
    \mathrm{Igs}^\circ\gx \subseteq \mathrm{Igs}\gx
  \]
  corresponding to $\mathbf{Sh}^\circ\gx_E^\lozenge \subseteq \sh\d$, where the
  inclusion of Igusa stacks is an open embedding.
\end{corollary}

\begin{proof}
  Once the corresponding Igusa stacks exist, it will be clear from
  \Cref{Thm:Functoriality} that there is a map $\mathrm{Igs}^\circ\gx \to
  \mathrm{Igs}\gx$, and then it follows from \cite[Proposition~10.11(i)]{Sch17p}
  that the induced map is an open embedding.

  Using \Cref{Thm:IgusaSubdata}, we immediately reduce to the case of $\gx =
  (\mathrm{GSp}_{2g}, \mathcal{H}_g^\pm)$. When $g \ge 2$, this is done in
  \Cref{Cor:IgusaSiegel}. For $g = 1$, we consider the diagonal embedding
  \[
    \mathrm{GL}_2 \hookrightarrow \mathrm{GSp}_4; \quad \begin{pmatrix} a & b \\
    c & d \end{pmatrix} \mapsto \begin{pmatrix} a & & & b \\ & a & b \\ & c &
    d \\ c & & & d \end{pmatrix},
  \]
  which induces a morphism of Shimura data, and then apply
  \Cref{Thm:IgusaSubdata} again. (We use the symplectic form $e_1^\vee \wedge
  e_4^\vee + e_2^\vee \wedge e_3^\vee$ for $\mathrm{GSp}_4$, following
  \Cref{Par:BailyBorel}.)
\end{proof}

}

\appendix
{\section{Minimally compactified Siegel modular varieties} \label{Sec:Boundary}
\def\GSp{\operatorname{GSp}}
\def\GL{\operatorname{GL}}
\def\SL{\operatorname{SL}}
\def\Gstd{\operatorname{GSp}_{2g}}
\def\Hpm{\mathcal{H}^\pm}
\def\Hstd{\mathcal{H}^\pm_g}
\def\gx{(\Gstd,\Hstd)}
\def\bung{\mathrm{Bun}_{\Gstd}}
\def\grgmu{\mathrm{Gr}_{\Gstd,\{\mu^{-1}\},\qp}}

In this appendix, we study the minimally compactified Siegel modular variety at
infinite level and compute the fixed points under a certain prime-to-$p$
unipotent Hecke action. We use this to deduce that for $g \ge 2$ the global
uniformization map
\[
  \Theta^\ast \colon \mathbf{Sh}^\ast\gx_{\qp}^\lozenge \times_{\bung} \grgmu
  \to \mathbf{Sh}^\ast\gx_{\qp}^\lozenge
\]
corresponding to minimally compactified Igusa stacks constructed in
\cite[Section~9.3]{Zha23p} sends boundary strata to boundary strata. As a
consequence, the map $\Theta^\ast$ restricts to a global uniformization
\[
  \Theta \colon \mathbf{Sh}\gx_{\qp}^\lozenge \times_{\bung} \grgmu \to
  \mathbf{Sh}\gx_{\qp}^\lozenge
\]
and hence defines an Igusa stack for the entire Shimura variety.

\paragraph
Let $(V, \psi)$ be a nonzero symplectic $\mathbb{Q}$-vector space of dimension
$2g$. We consider the general symplectic group $\GSp(V, \psi)$ of pairs
$(\gamma, \lambda)$ where $\gamma \in \Aut(V)$ satisfies $\psi(\gamma x, \gamma
y) = \lambda \psi(x, y)$. We also have the space $\Hpm(V, \psi)$ of
homomorphisms
\[
  h \colon \operatorname{Res}_{\mathbb{C}/\mathbb{R}} \mathbb{G}_{m,\mathbb{C}}
  \to \GSp(V, \psi)_\mathbb{R}
\]
satisfying the property that $h$ defines complex structure on $V_\mathbb{R}$ and
$\psi(x, h(i)y)$ is a definite symmetric pairing on $V_\mathbb{R}$. The Lie
group $\GSp(V, \psi)(\mathbb{R})$ acts transitively on $\Hpm(V, \psi)$ by
conjugation, and the pair $(\GSp(V, \psi), \Hpm(V, \psi))$ defines a Shimura
datum. We shall also write $\Gstd = \GSp(V, \psi)$ and $\Hstd = \Hpm(V, \psi)$
when $(V, \psi)$ is clear from the context.

\begin{remark}
  When $g = 0$, we follow the convention of \cite[Example~2.8]{Pin90}. Namely,
  we define $\Gstd = \mathbb{G}_{m,\mathbb{Q}}$ and $\Hstd = \lbrace \pm
  \sqrt{-1} \rbrace$ as a homogeneous space for $\mathbb{R}^\times$. They do not
  form a Shimura datum in the sense of \cite{Del71} or \cite{Mil05} but only in
  the slightly more general setup of \cite[Definition~2.1]{Pin90}. This
  modification is necessary to obtain the correct adelic description of the
  Satake compactification.
\end{remark}

\paragraph
For each (not necessarily maximal) isotropic subspace $W \subseteq V$, there is
a corresponding rational boundary component given as follows, cf.\
\cite[Example~4.25]{Pin90}. We have the maximal parabolic subgroup and its
canonical normal subgroup
\begin{align*}
  Q_W &= \lbrace \gamma \in \Gstd : \gamma(W) = W, \gamma(W^\perp) =
  W^\perp \rbrace, \\ P_W &= \lbrace \gamma \in Q_W : \gamma \equiv \id
  \pmod{W^\perp} \rbrace,
\end{align*}
where $W^\perp$ is the orthogonal complement of $W$. The maximal reductive
quotient of $P_W$ is canonically identified with the general symplectic group of
$W^\perp/W$,
\[
  P_W \twoheadrightarrow G_W = \GSp(W^\perp/W, \psi), \quad \gamma \mapsto
  \gamma \vert_{W^\perp}.
\]

\paragraph
There is a topological space $(\Hstd)^\ast$ that is set-theoretically
described as
\[
  (\Hstd)^\ast = \coprod_{W \subseteq V} \Hpm(W^\perp/W, \psi),
\]
where the union is over all isotropic subspaces $W \subseteq V$. This carries a
Satake topology, see \cite[Section~III.3]{BJ05} and the modification in
\cite[Section~6.2]{Pin90} to finite disjoint unions of symmetric spaces. There
is moreover a group action of $\Gstd(\mathbb{Q})$ on $(\Hstd)^\ast$ by
homeomorphisms. We remark that $(\Hstd)^\ast$ is not compact as a topological
space.

\paragraph \label{Par:BailyBorel}
Let $K \subseteq \Gstd(\af)$ be a neat compact open subgroup. We obtain the
\textdef{Baily--Borel compactification} or the \textdef{minimal
compactification} of the Shimura variety $\mathbf{Sh}_K\gx_\mathbb{C}$ as
\[
  \mathbf{Sh}_K^\ast\gx(\mathbb{C}) = \Gstd(\mathbb{Q}) \backslash (\Hstd)^\ast
  \times \Gstd(\af) / K.
\]
We note that for each $0 \le h \le g$, there exists a unique
$\Gstd(\mathbb{Q})$-orbit of an isotropic subspace $W \subseteq V$. Fix a basis
$e_1, \dotsc, e_{2g} \in V$ so that
\[
  \psi(e_i, e_j) = \delta_{j,2g+1-i}
\]
for $1 \le i < j \le 2g$. For each $0 \le h \le g$, we choose the isotropic
subspace $W_h = \operatorname{span}(e_1, \dotsc, e_h)$ so that the corresponding
parabolic subgroup and its normal subgroup are
\[
  Q_h = \Gstd \cap \left\lbrace \begin{pmatrix} \ast & \ast & \ast \\ 0 & \ast &
  \ast \\ 0 & 0 & \ast \end{pmatrix} \right\rbrace, \quad P_h = \Gstd \cap
  \left\lbrace \begin{pmatrix} \ast & \ast & \ast \\ 0 & \ast & \ast \\ 0 & 0 &
  1 \end{pmatrix} \right\rbrace,
\]
where the blocks have size $h, 2(g-h), h$. The reductive quotient of $P_h$ is
$G_h \cong \GSp_{2(g-h)}$ and $Q_h(\mathbb{Q})$ is the
stabilizer of $W_h$, and hence we obtain a stratification
\[
  \mathbf{Sh}_K^\ast\gx(\mathbb{C}) = \coprod_{h=0}^g Q_h(\mathbb{Q})
  \backslash \Hpm_{g-h} \times \Gstd(\af) / K.
\]

\paragraph \label{Par:CuspFinLevel}
For each $0 \le h \le g$ and a representative $c \in \Gstd(\af)$ of the double
coset
\[
  [c] \in Q_h(\mathbb{Q}) P_h(\af) \backslash \Gstd(\af) / K,
\]
the contribution of $[c]$ to the minimal compactification be computed as
follows, see \cite[Section~6.3]{Pin90}. Write $\pi \colon Q_h \twoheadrightarrow
M_h \cong \GSp_{2(g-h)} \times \GL_h$ for the reductive quotient, and consider
the groups
\begin{align*}
  H_{c,K} &= cKc^{-1} \cap Q_h(\af) \subseteq Q_h(\af), \\ L_{c,K} &=
  \pi(H_{c,K}) = \pi(cKc^{-1} \cap Q_h(\af)) \subseteq M_h(\af), \\ K_{c,K} &=
  L_{c,K} \cap G_h(\af) = \pi(cKc^{-1} \cap P_h(\af)) \subseteq G_h(\af), \\
  \Delta_{c,K} &= (M_h(\mathbb{Q}) \cap G_h(\af)L_{c,K}) / G_h(\mathbb{Q}).
\end{align*}
Then we see that
\begin{align*}
  Q_h(\mathbb{Q}) \backslash \Hpm_{g-h} &\times Q_h(\mathbb{Q}) P_h(\af) c K / K
  \\ &\xleftarrow{\cong} (Q_h(\mathbb{Q}) \cap P_h(\af) H_{c,K}) \backslash
  \Hpm_{g-h} \times P_h(\af) H_{c,K} / H_{c,K}
\end{align*} 
sending $[x, \gamma H_{c,K}]$ to $[x, \gamma c K]$ is an isomorphism, where we
note that $P_h(\af) H_{c,K} \subseteq Q_h(\af)$ is a subgroup. Using the fact
that for the unipotent radical $R_\mathrm{u}(Q_h) \subseteq Q_h$ the subgroup
$R_\mathrm{u}(Q_h)(\mathbb{Q}) \subseteq R_\mathrm{u}(Q_h)(\af)$ is dense, we
quotient both by $R_\mathrm{u}(Q_h)(\mathbb{Q})$ to further simplify it to
\begin{align*}
  (Q_h(\mathbb{Q}) \cap P_h(\af) H_{c,K}) &\backslash \Hpm_{g-h} \times P_h(\af)
  H_{c,K} / H_{c,K} \\ &\xrightarrow{\cong} (M_h(\mathbb{Q}) \cap G_h(\af)
  L_{c,K}) \backslash \Hpm_{g-h} \times G_h(\af) L_{c,K} / L_{c,K} \\ &\cong
  \Delta_{c,K} \backslash \mathbf{Sh}_{K_{c,K}}(\GSp_{2(g-h)},
  \Hpm_{g-h})(\mathbb{C}).
\end{align*}
When $K$ satisfies the property that $L_{c,K}$ factors as a direct product of open
compact subgroups under $M_h(\af) \cong \GSp_{2(g-h)}(\af) \times \GL_h(\af)$,
the group $\Delta_{c,K}$ acts trivially on $\mathbf{Sh}_{K_{c,K}}(\GSp_{2(g-h)},
\Hpm_{g-h})(\mathbb{C})$.

\paragraph
By \cite[Section~12.3]{Pin90}, the minimal compactification
$\mathbf{Sh}^\ast_K\gx$ has a natural structure of a projective variety over
$\mathbb{Q}$, and moreover the inclusion
\[
  \Delta_{c,K} \backslash \mathbf{Sh}_{K_{c,K}}(\GSp_{2(g-h)}, \Hpm_{g-h})
  \hookrightarrow \mathbf{Sh}^\ast_K\gx
\]
defines a locally closed embedding of quasi-projective varieties over
$\mathbb{Q}$. We note that the collection of $K$ for which $L_{c,K}$ splits as a
direct product form a neighborhood basis of $1$. Taking the limit as $K \to
\{1\}$, we obtain for each
\[
  [c] \in Q_h(\mathbb{Q}) P_h(\af) \backslash \Gstd(\af)
\]
a locally closed embedding of $\mathbb{Q}$-schemes
\[
  \mathbf{Sh}(\GSp_{2(g-h)}, \Hpm_{g-h}) \cong Z_{[c]} \hookrightarrow
  \mathbf{Sh}^\ast\gx = \varprojlim_{K \to \{1\}} \mathbf{Sh}^\ast_K\gx,
\]
where on both sides the transition maps are finite. Such elements $[c]$ are also
called \textdef{cusp labels}.

\begin{remark}
  The locally closed subscheme $Z_{[c]}$ depends only on the cusp label $[c] \in
  Q_h(\mathbb{Q}) P_h(\af) \backslash \Gstd(\af)$, but the isomorphism between
  $Z_{[c]}$ and the Shimura variety $\mathbf{Sh}(\GSp_{2(g-h)}, \Hpm_{g-h})$
  depends on choice of representative $c \in \Gstd(\af)$.
\end{remark}

\begin{lemma} \label{Lem:CuspLabelInf}
  Consider the character $\omega_h \colon Q_h \to \mathbb{G}_ {m,\mathbb{Q}}$
  given by the determinant of the lower right block. Then the closure of the
  subgroup $Q_h(\mathbb{Q}) P_h(\af) \subseteq \Gstd(\af)$ is
  \[
    \omega_h(\af)^{-1}(\mathbb{Q}^\times) = \lbrace q \in Q_h(\af) : \omega_h(q)
    \in \mathbb{Q}^\times \subseteq (\af)^\times \rbrace,
  \]
  and hence
  \[
    \varprojlim_{K \to \{1\}} Q_h(\mathbb{Q}) P_h(\af) \backslash \Gstd(\af) / K
    \cong \omega_h(\af)^{-1}(\mathbb{Q}^\times) \backslash \Gstd(\af).
  \]
\end{lemma}

\begin{proof}
  When $h = 0$, we have $P_h = Q_h = \Gstd$ and $\omega_h = 1$, thus the
  statement trivially holds. For $h > 0$, we first note that $Q_h(\mathbb{Q})
  P_h(\af)$ is contained in the closed subgroup $Q_h(\af)$, and thus it suffices
  to compute the closure of $Q_h(\mathbb{Q}) P_h(\af) \subseteq Q_h(\af)$.
  Because there is a splitting $Q_h \cong P_h \times (Q_h/P_h)$ as affine
  schemes, this is the preimage of the closure of $(Q_h/P_h)(\mathbb{Q})
  \subseteq (Q_h/P_h)(\af)$. On the other hand, $Q_h / P_h \cong \GL_h$ and by
  strong approximation for $\SL_h$, the closure of $\GL_h(\mathbb{Q}) \subseteq
  \GL_h(\af)$ is the subgroup of matrices whose determinant is in
  $\mathbb{Q}^\times$.
\end{proof}

\paragraph
Using \Cref{Lem:CuspLabelInf}, we obtain a stratification of
$\mathbf{Sh}^\ast\gx$ by $Z_{[c]}$ in the sense that we have a disjoint union
\[
  \mathbf{Sh}^\ast\gx = \coprod_{h=0}^g Z_h, \quad Z_h = \coprod_{[c] \in
  \omega_h(\af)^{-1}(\mathbb{Q}^\times) \backslash \Gstd(\af)} Z_{[c]}
\]
of the underlying topological space by locally closed subsets. In the first
decomposition $Z_h \cup Z_{h+1} \cup \dotsb \cup Z_g$ is closed, and in the
second decomposition each $Z_{[c]}$ is closed in $Z_h$. From now on, we refer to
elements $[c] \in \omega_h(af)^{-1}(\mathbb{Q}^\times) \backslash \Gstd(\af)$ as
\textdef{cusp labels}. The stabilizer of each cusp label $[c]$ is the subgroup
\[
  c^{-1} \omega_h(\af)^{-1}(\mathbb{Q}^\times) c \subseteq c^{-1} Q_h(\af) c
  \subseteq \Gstd(\af).
\]

\begin{proposition} \label{Prop:CuspAction}
  Choose a representative $c \in \Gstd(\af)$ of a cusp label $[c] \in
  \omega_h(\af)^{-1}(\mathbb{Q}^\times) \backslash \Gstd(\af)$. Then the locally
  closed embedding
  \[
    \mathbf{Sh}(\GSp_{2(g-h)}, \Hpm_{g-h}) \cong Z_{[c]} \hookrightarrow
    \mathbf{Sh}^\ast\gx
  \]
  is $c^{-1} \omega_h(\af)^{-1}(\mathbb{Q}^\times) c$-equivariant, where the
  action on the left hand side is through
  \[
    c^{-1} \omega_h(\af)^{-1}(\mathbb{Q}^\times) c \xrightarrow{\mathrm{conj}_c}
    \omega_h(\af)^{-1}(\mathbb{Q}^\times) \subseteq Q_h(\af) \twoheadrightarrow
    \GSp_{2(g-h)}(\af).
  \]
\end{proposition}

\begin{proof}
  It is clear that the action of locally profinite group $c^{-1}
  \omega_h(\af)^{-1}(\mathbb{Q}^\times) c$ on both sides is continuous. Because
  $c^{-1} Q_h(\mathbb{Q}) P_h(\af) c$ is dense in $c^{-1}
  \omega_h(\af)^{-1}(\mathbb{Q}^\times) c$ by \Cref{Lem:CuspLabelInf}, the map
  of $\mathbb{Q}$-schemes
  \[
    \coprod_{c^{-1} Q_h(\mathbb{Q}) P_h(\af) c} (\Spec \mathbb{Q})
    \to \underline{c^{-1} \omega_h(\af)^{-1}(\mathbb{Q}^\times) c}_{\mathbb{Q}}
  \]
  is schematically dense. Since $\mathbf{Sh}^\ast\gx$ is separated,
  it suffices to prove that the locally closed embedding is $c^{-1}
  Q_h(\mathbb{Q}) P_h(\af) c$-equivariant.

  We now trace through the computation of \Cref{Par:CuspFinLevel}. Given $\gamma
  \in c^{-1} Q_h(\mathbb{Q}) P_h(\af) c$, we write it as $\gamma =
  c^{-1} q_\gamma p_\gamma c$ where $q_\gamma \in \ker(Q_h \to
  \GSp_{2(g-h)})(\mathbb{Q})$ and $p_\gamma \in P_h(\af)$. For $p \in P_h(\af)$
  and $x \in \Hpm_{g-h}$, consider the element
  \[
    [x, p H_{c,K}] \in (Q_h(\mathbb{Q}) \cap P_h(\af) H_{c,K}) \backslash
    \Hpm_{g-h} \times P_h(\af) H_{c,K} / H_{c,K},
  \]
  which corresponds to $[x, p c K] \in Q_h(\mathbb{Q}) \backslash \Hpm_{g-h}
  \times \Gstd(\af) / K$. The Hecke action by $\gamma$ sends it to $[x, p c
  K \gamma^{-1}] \in Q_h(\mathbb{Q}) \backslash \Hpm_{g-h} \times \Gstd(\af) /
  \gamma K \gamma^{-1}$. Upon rewriting
  \[
    p c K \gamma^{-1} = p (c \gamma^{-1} c^{-1}) c (\gamma K \gamma^{-1}) =
    q_\gamma (q_\gamma^{-1} p q_\gamma) p_\gamma c (\gamma K \gamma^{-1})
  \]
  and observing that $q_\gamma \in Q_h(\mathbb{Q})$ acts trivially on
  $\Hpm_{g-h}$, we see that $[x, pc K \gamma^{-1}]$ corresponds to
  \begin{align*}
    [x, &(q_\gamma^{-1} p q_\gamma) p_\gamma H_{c,\gamma K \gamma^{-1}}] \\ &\in
    (Q_h(\mathbb{Q}) \cap P_h(\af) H_{c,\gamma K \gamma^{-1}}) \backslash
    \Hpm_{g-h} \times P_h(\af) H_{c, \gamma K \gamma^{-1}} / H_{c, \gamma K
    \gamma^{-1}}.
  \end{align*}

  Upon projecting down along $\pi \colon Q_h \twoheadrightarrow M_h$, we see
  that $p$ and $q_\gamma$ commutes and hence $\gamma$ sends
  \[
    [x, \pi(p) L_{c,K}] \in (M_h(\mathbb{Q}) \cap G_h(\af) L_{c,K}) \backslash
    \Hpm_{g-h} \times G_h(\af) L_{c,K} / L_{c,K}
  \]
  to
  \begin{align*}
    [x, &\pi(p p_\gamma) L_{c,\gamma K \gamma^{-1}}] \\ &\in (M_h(\mathbb{Q})
    \cap G_h(\af) L_{c, \gamma K \gamma^{-1}}) \backslash \Hpm_{g-h} \times
    G_h(\af) L_{c, \gamma K \gamma^{-1}} / L_{c, \gamma K \gamma^{-1}}.
  \end{align*}
  Passing to the limit as $K \to \{1\}$, we see that this is the Hecke action of
  $\pi(p_\gamma)$ on $\mathbf{Sh}(\GSp_{2(g-h)}, \Hpm_{g-h})$, which agrees with
  the action of $c \gamma c^{-1} = q_\gamma p_\gamma$.
\end{proof}

\paragraph
There is a right Kan extended functor from the category of quasi-compact schemes
over $\mathbb{Q}$ to the category v-sheaves, sending $X \to \Spec \mathbb{Q}$ to
the v-sheaf
\[
  X_{\qp}^\lozenge = \varprojlim_{X \to Y}
  (Y_{\qp}^\mathrm{ad})^\lozenge,
\]
where the limit is over all maps from $X$ to finite type schemes $Y \to \Spec
\mathbb{Q}$, and $Y_{\qp}^\lozenge =
(Y_{\qp}^\mathrm{ad})^\lozenge$ is defined as in \Cref{Par:Lozenge}.
Concretely, if $X$ is a cofiltered limit $X = \varprojlim_i X_i$ of finite type
schemes along finite transition maps, then we have $X_{\qp}^\lozenge =
\varprojlim_i X_{i,\qp}^\lozenge$. Taking the limit of the finite
stratifications of varieties at finite level, we obtain a stratification
\[
  \mathbf{Sh}^\ast\gx_{\qp}^\lozenge = \coprod_{h=0}^g
  \biggl( \coprod_{[c] \in \omega_h(\af)^{-1}(\mathbb{Q}^\times) \backslash
  \Gstd(\af)} Z_{[c],\qp}^\lozenge \biggr),
\]
in the sense that all inclusions are locally closed embeddings of v-sheaves and
on topological spaces it is a set-theoretic partition.

\begin{lemma} \label{Lem:NoFixedPts}
  Assume $g \ge 1$. Let $\gamma \in \Gstd(\af)$ be an element that is not
  conjugate to a semisimple element of $\Gstd(\mathbb{Q})$. Then the fixed
  points of $\gamma$ on $\mathbf{Sh}\gx_{\qp}^\lozenge$ is the empty v-sheaf.
\end{lemma}

\begin{proof}
  We observe that the $\gamma$-fixed points can be described as
  \[
    \varprojlim_{K \to \{1\}} \operatorname{eq}(\mathbf{Sh}_{K \cap \gamma K
    \gamma^{-1}}\gx_{\qp}^\lozenge \rightrightarrows
    \mathbf{Sh}_{\gamma K \gamma^{-1}}\gx_{\qp}^\lozenge)
  \]
  where the two maps are the projection and the projection to level $K$ composed
  with the action of $\gamma$. In particular, it is a limit of v-sheaves
  associated to closed subvarieties of finite level Shimura varieties. Hence it
  suffices to show that the $\gamma$-action on
  \[
    \mathbf{Sh}\gx_{\qp}^\lozenge(\cp)_\mathrm{classical} \cong
    \mathbf{Sh}\gx(\qpbar)
  \]
  has no fixed points. Choose an abstract isomorphism $\qpbar \cong \mathbb{C}$
  so that we can identify it with
  \[
    \mathbf{Sh}\gx(\mathbb{C}) = \Gstd(\mathbb{Q}) \backslash
    \Hstd \times \Gstd(\af).
  \]
  If $[x, y] = [x, y \gamma^{-1}]$, then $\delta x = x$ and $\delta y = y
  \gamma^{-1}$ for some $\delta \in \Gstd(\mathbb{Q})$. Because $\delta x = x$
  for some $x \in \Hstd$, we see that $\delta$ lies in a subgroup of
  $\Gstd(\mathbb{R})$ that is compact modulo the center. It follows that
  $\delta$ is semisimple, hence $\gamma = y^{-1} \delta^{-1} y$ contradicts our
  assumption.
\end{proof}

\paragraph
For $0 \le h \le g$, let us consider the unipotent subgroup
\[
  R_h = \left\lbrace \begin{pmatrix} U_1 & \ast & \ast \\ 0 & 1 & \ast \\ 0 & 0
  & U_2 \end{pmatrix} \in \Gstd : U_1, U_2 \in \GL_h \text{ unipotent
  upper-triangular} \right\rbrace,
\]
where the blocks have size $h, 2(g-h), h$.

\begin{proposition} \label{Prop:UnipotentFixedPts}
  Let $0 \le k \le g$ be an integer. Then for each $\gamma \in \Gstd(\af)$ we
  have
  \begin{align*}
    (\mathbf{Sh}^\ast&{}\gx_{\qp}^\lozenge)^{\gamma R_k(\afp)
    \gamma^{-1}} \\ &= \coprod_{h=k}^g \biggl( \coprod_{[c] \in
    \omega_h(\af)^{-1}(\mathbb{Q}^\times) \backslash Q_h(\af) Q_k(\af)
    \Gstd(\qp) \gamma^{-1}} Z_{[c],\qp}^\lozenge \biggr),
  \end{align*}
  meaning that it is a stratification on underlying topological spaces.
\end{proposition}

\begin{proof}
  Because the $R_k(\afp)$-fixed points is a closed v-subsheaf, it is determined
  by its underlying topological space. Hence we may verify the statement on each
  stratum $Z_{[c],\qp}$ separately. We also reduce to the case when $\gamma =
  1$, because we may obtain the general case by applying the action of $\gamma$.

  We first identify the cusp labels $(h, [c])$ whose stabilizer contains
  $R_k(\afp)$, i.e.,
  \[
    R_k(\afp) \subseteq c^{-1} \omega_h(\af)^{-1}(\mathbb{Q}^\times) c.
  \]
  By enlarging $\omega_h(\af)^{-1}(\mathbb{Q}^\times)$ to $Q_h(\af)$ and
  projecting to $\afp$, we obtain
  \[
    R_k(\afp) \subseteq (c^p)^{-1} Q_h(\afp) c^p
  \]
  after writing $c = c^p c_p$ for $c^p \in \Gstd(\afp)$ and $c_p \in
  \Gstd(\qp)$. At each finite prime $v \neq p$ this implies that
  $R_k(\mathbb{Q}_v)$ stabilizes the isotropic $\mathbb{Q}_v$-vector space
  $c_v^{-1} W_{h,\mathbb{Q}_v}$. Using the definition of $R_k$, one computes
  that either (i) $h < k$ and $c_v^{-1} W_{h,\mathbb{Q}_v} =
  W_{h,\mathbb{Q}_v}$, meaning $c_v \in Q_h(\mathbb{Q}_v)$, or (ii) $h \ge k$
  and $c_v^{-1} W_{h,\mathbb{Q}_v} \supseteq W_{k,\mathbb{Q}_v}$, meaning $c_v
  \in Q_h(\mathbb{Q}_v) Q_k(\mathbb{Q}_v)$, because there exists a $\gamma \in
  \Gstd(\mathbb{Q}_v)$ such that $\gamma(c_v^{-1} W_{h,\mathbb{Q}_v},
  W_{k,\mathbb{Q}_v}) = (W_{h,\mathbb{Q}_v}, W_{k,\mathbb{Q}_v})$. Collecting
  this for all $v \neq p$, we conclude that either (i) $h < k$ and $c^p \in
  Q_h(\afp)$ or (ii) $h \ge k$ and $c^p \in Q_h(\afp) Q_k(\afp)$.

  Let us first consider the case when $h < k$ and $c^p \in Q_h(\afp)$. In this
  case, we indeed have
  \[
    c R_k(\afp) c^{-1} = c^p R_k(\afp) (c^p)^{-1} \subseteq
    \omega_h(\af)^{-1}(\mathbb{Q}^\times)
  \]
  because $c^p \in Q_h(\afp)$ and $\omega_h(\af)^{-1}(\mathbb{Q}^\times)$ is
  normal in $Q_h(\af)$. Because $h < k$, there exists an element $\gamma \in
  R_k(\afp)$ whose image under $R_k \hookrightarrow Q_h \to \GSp_{2(g-h)}$ is
  nontrivial unipotent $\delta \in \GSp_{2(g-h)}(\afp)$. By
  \Cref{Prop:CuspAction}, the action of $c^{-1} \gamma c \in R_h(\afp)$ on
  $Z_{[c]} \cong \mathbf{Sh}(\GSp_{2(g-h)}, \Hpm_{g-h})$ agrees with the Hecke
  action by $\delta$. On the other hand, \Cref{Lem:NoFixedPts} shows that the
  $\delta$-fixed points is empty.

  We now consider the case when $h \ge k$ and $c^p \in Q_h(\afp) Q_k(\afp)$.
  Again, we similarly check that
  \[
    c R_k(\afp) c^{-1} = c^p R_k(\afp) (c^p)^{-1} \subseteq
    \omega_h(\af)^{-1}(\mathbb{Q}^\times),
  \]
  as conjugating $R_k(\afp)$ by an element of $Q_k(\afp)$ gives an element of
  $\ker(Q_k \to \GSp_{2(g-k)})(\afp)$, which is contained inside
  $\omega_h(\af)^{-1}(\mathbb{Q}^\times)$. Next, for every element $\gamma \in
  R_k(\afp)$, its action on $Z_{[c]} \cong \mathbf{Sh}(\GSp_{2(g-h)},
  \Hpm_{g-h})$ is identified with the action of $c \gamma c^{-1} \in \ker(Q_h
  \to \GSp_{2(g-h)})(\afp)$ by \Cref{Prop:CuspAction}, which is trivial.
\end{proof}

\paragraph
We now assume that $g \ge 2$, so that the boundary of the minimal
compactification $\mathbf{Sh}_K^\ast\gx$ has codimension at least $2$. In
\cite[Section~9.3]{Zha23p}, Zhang constructs for each neat compact open subgroup
$K^p \subseteq \Gstd(\afp)$ a minimally compactified Igusa stack
\[
  \mathrm{Igs}_{K^p}^\ast\gx \to \mathrm{Bun}_{\Gstd}
\]
by taking a partially compactified relative affinization of the good reduction
Igusa stack $\mathrm{Igs}_{K^p}^\circ\gx \to \mathrm{Bun}_{\Gstd}$. We have by
\cite[Proposition~10.1]{Zha23p} a limit
\[
  \mathrm{Igs}^\ast\gx = \varprojlim_{K^p \to \{1\}} \mathrm{Igs}_{K^p}^\ast\gx
\]
with a natural $\underline{\Gstd(\afp)}$-action, and moreover by
\cite[Theorem~9.38]{Zha23p} a $\underline{\Gstd(\afp)} \times
\underline{\Gstd(\qp)}$-equivariant isomorphism
\[
  \mathbf{Sh}^\ast\gx_{\qp}^\lozenge \cong \mathrm{Igs}^\ast\gx
  \times_{\bung} \grgmu.
\]

\paragraph
Following the construction of \Cref{Par:CechNerve}, we obtain for $g \ge 2$ a
uniformization map
\begin{align*}
  \Theta^\ast \colon \mathbf{Sh}^\ast\gx_{\qp}^\lozenge &\times_{\bung} \grgmu
  \\ &\cong \mathbf{Sh}^\ast\gx_{\qp}^\lozenge \times_{\mathrm{Igs}^\ast\gx}
  \mathbf{Sh}^\ast\gx_{\qp}^\lozenge \\ &\xrightarrow{\mathrm{pr}_2}
  \mathbf{Sh}^\ast\gx_{\qp}^\lozenge
\end{align*}
that is $\underline{\Gstd(\af)} \times \underline{\Gstd(\qp)}$-equivariant. As
noted in \cite[Proposition~5.2.15]{DvHKZ24p}, the isomorphism
\[
  \mathrm{Igs}^\circ\gx \times_{\bung} \grgmu \cong
  \mathbf{Sh}^\circ\gx_{\qp}^\lozenge
\]
from \cite[Corollary~8.15]{Zha23p} satisfies the property that $(\phi, \id)$ on
the left hand side corresponds to $\id$ on the right hand side. Taking the
partially compactified relative affinization over $\grgmu$ on both sides, we see
that
\[
  \mathrm{Igs}^\ast\gx \times_{\bung} \grgmu \cong
  \mathbf{Sh}^\ast\gx_{\qp}^\lozenge
\]
also satisfies the property same equivariance property, and hence the
uniformization map $\Theta^\ast$ is invariant under the $(\phi, \id)$-action on
the source. Finally, the commutativity of the diagrams
(\ref{Item:GlobalUnifDiagrams}) in \Cref{Def:GlobalUniformization} follow
immediately from the definition of $\Theta^\ast$. That is, $\Theta^\ast$
satisfies all the conditions of \Cref{Def:GlobalUniformization}.

\begin{proposition} \label{Prop:MinCompUniformization}
  Assume $g \ge 2$, and write
  \[
    \mathbf{Sh}^\ast\gx_{\qp}^\lozenge = \coprod_{h=0}^g Z_{h,\qp}^\lozenge,
    \quad Z_{h,\qp}^\lozenge = \coprod_{[c] \in
    \omega_h^{-1}(\af)^{-1}(\mathbb{Q}^\times) \backslash \Gstd(\af)}
    Z_{[c],\qp}^\lozenge,
  \]
  so that each $Z_{h,\qp}^\lozenge \subseteq \mathbf{Sh}^\ast\gx_{\qp}^\lozenge$
  is a locally closed v-subsheaf. Then for each $0 \le h \le g$, the map
  $\Theta^\ast$ sends
  \[
    \Theta^\ast \colon Z_{h,\qp}^\lozenge \times_{\bung} \grgmu \to
    Z_{h,\qp}^\lozenge.
  \]
\end{proposition}

\begin{proof}
  Since $Z_{h,\qp}^\lozenge$ is a locally closed v-subsheaf, it suffices to show
  that $\Theta^\ast$ maps the topological space for $Z_{h,\qp}^\lozenge
  \times_{\bung} \grgmu$ to the topological space for $Z_{h,\zp}^\lozenge$.
  Because $\Theta^\ast$ is $\Gstd(\afp)$-equivariant, for each $0 \le k \le g$
  and $\gamma \in \Gstd(\af)$ the map $\Theta^\ast$ sends
  \begin{align*}
    \Theta^\ast \colon (\mathbf{Sh}^\ast\gx_{\qp}^\lozenge)^{\gamma
    R_k(\afp) \gamma^{-1}} &\times_{\bung} \grgmu \\ &\to
    (\mathbf{Sh}^\ast\gx_{\qp}^\lozenge)^{\gamma R_k(\afp) \gamma^{-1}}.
  \end{align*}
  By taking the union of \Cref{Prop:UnipotentFixedPts} for all $\gamma$, we have
  \[
    \bigcup_{\gamma \in \Gstd(\af)} (\mathbf{Sh}^\ast\gx_{\qp}^\lozenge)^{\gamma
    R_k(\afp) \gamma^{-1}} = \coprod_{h=k}^g Z_{h,\qp}^\lozenge,
  \]
  and hence $\Theta^\ast$ maps
  \[
    \Theta^\ast \colon \biggl( \coprod_{h=k}^g Z_{h,\qp}^\lozenge \biggr)
    \times_{\bung} \grgmu \to \biggl( \coprod_{h=k}^g Z_{h,\qp}^\lozenge \biggr)
  \]
  by \cite[Corollary~5.5(i)]{Sch17p}.

  On the other hand, from diagrams of (\ref{Item:GlobalUnifDiagrams}) in
  \Cref{Def:GlobalUniformization} it follows that
  \[ \begin{tikzcd}
    \mathbf{Sh}^\ast\gx_{\qp}^\lozenge \times_{\bung} \grgmu \arrow{d}{(\Theta,
    \pi_\mathrm{HT}^\ast \circ \mathrm{pr}_1)} \arrow{r}{\mathrm{pr}_1} &
    \mathbf{Sh}^\ast\gx_{\qp}^\lozenge \arrow[equals]{d} \\
    \mathbf{Sh}^\ast\gx_{\qp}^\lozenge \times_{\bung} \grgmu \arrow{r}{\Theta} &
    \mathbf{Sh}^\ast\gx_{\qp}^\lozenge
  \end{tikzcd} \]
  commutes. If for some $0 \le k < k^\prime \le g$, the map $\Theta^\ast$ maps a
  point $x \in Z_{k,\qp}^\lozenge \times_{\bung} \grgmu$ to
  $Z_{k^\prime,\qp}^\lozenge \times_{\bung} \grgmu$, then the image of $x$ under
  the upper right path is in the stratum $Z_{k,\qp}^\lozenge$, while the image
  under the lower left path is in $\bigcup_{h=k^\prime}^g Z_{h,\qp}^\lozenge$.
  This is a contradiction, and thus $\Theta^\ast$ sends $Z_{k,\qp}^\lozenge
  \times_{\bung} \grgmu$ to $Z_{k,\qp}^\lozenge$.
\end{proof}

\begin{corollary} \label{Cor:IgusaSiegel}
  For $g \ge 2$, there exists an Igusa stack $\mathrm{Igs}\gx$ for $U =
  \mathbf{Sh}\gx_{\qp}^\lozenge$, which naturally fits in between
  \[
    \mathrm{Igs}^\circ\gx \subset \mathrm{Igs}\gx \subset \mathrm{Igs}^\ast\gx,
  \]
  where both inclusions are open embeddings.
\end{corollary}

\begin{proof}
  Setting $h = 0$ in \Cref{Prop:MinCompUniformization}, we obtain the global
  uniformization map
  \[
    \Theta \colon \mathbf{Sh}\gx_{\qp}^\lozenge \times_{\bung} \grgmu \to
    \mathbf{Sh}\gx_{\qp}^\lozenge
  \]
  by restricting $\Theta^\ast$. This satisfies all the axioms of
  \Cref{Def:GlobalUniformization} because $\Theta^\ast$ does, and therefore by
  \Cref{Prop:IgusaGlobalUnif} there exists a Igusa stack $\mathrm{Igs}\gx$. The
  fact that the inclusions are open embeddings follows from the corresponding
  fact on Shimura varieties together with \cite[Proposition~10.11(i)]{Sch17p}.
\end{proof}

\begin{remark}
  The upcoming work \cite{CHZ25p} performs a much finer analysis of
  $\Theta^\ast$ on the boundary strata, and obtains a precise description of the
  corresponding boundary strata of Igusa stacks. It also works in the more
  general setting of PEL Shimura data.
\end{remark}

}

\bibliographystyle{amsalpha}
\bibliography{references.bib}

\providecommand{\bysame}{\leavevmode\hbox to3em{\hrulefill}\thinspace}
\providecommand{\MR}{\relax\ifhmode\unskip\space\fi MR }
\providecommand{\MRhref}[2]{%
  \href{http://www.ams.org/mathscinet-getitem?mr=#1}{#2}
}
\providecommand{\href}[2]{#2}
\begin{thebibliography}{DvHKZ24}

\bibitem[BGR84]{BGR84}
S.~Bosch, U.~G\"untzer, and R.~Remmert, \emph{Non-{A}rchimedean analysis},
  Grundlehren der mathematischen Wissenschaften [Fundamental Principles of
  Mathematical Sciences], vol. 261, Springer-Verlag, Berlin, 1984, A systematic
  approach to rigid analytic geometry. \MR{746961}

\bibitem[BJ05]{BJ05}
Armand Borel and Lizhen Ji, \emph{Compactifications of symmetric and locally
  symmetric spaces}, Lie theory, Progr. Math., vol. 229, Birkh\"auser Boston,
  Boston, MA, 2005, pp.~69--137. \MR{2126641}

\bibitem[BMS18]{BMS18}
Bhargav Bhatt, Matthew Morrow, and Peter Scholze, \emph{Integral {$p$}-adic
  {H}odge theory}, Publ. Math. Inst. Hautes \'Etudes Sci. \textbf{128} (2018),
  219--397. \MR{3905467}

\bibitem[CHZ]{CHZ25p}
Ana Caraiani, Linus Hamann, and Mingjia Zhang, \emph{Intersection cohomology of
  {I}gusa stacks}, upcoming work.

\bibitem[Con99]{Con99}
Brian Conrad, \emph{Irreducible components of rigid spaces}, Ann. Inst. Fourier
  (Grenoble) \textbf{49} (1999), no.~2, 473--541. \MR{1697371}

\bibitem[Con06]{Con06}
\bysame, \emph{Relative ampleness in rigid geometry}, Ann. Inst. Fourier
  (Grenoble) \textbf{56} (2006), no.~4, 1049--1126. \MR{2266885}

\bibitem[Con11]{Con11p}
\bysame, \emph{Lifting global representations with local properties}.

\bibitem[CS17]{CS17}
Ana Caraiani and Peter Scholze, \emph{On the generic part of the cohomology of
  compact unitary {S}himura varieties}, Ann. of Math. (2) \textbf{186} (2017),
  no.~3, 649--766. \MR{3702677}

\bibitem[Del71]{Del71}
Pierre Deligne, \emph{Travaux de {S}himura}, S\'eminaire {B}ourbaki, 23\`eme
  ann\'ee (1970/1971), Lecture Notes in Math., vol. Vol. 244, Springer,
  Berlin-New York, 1971, pp.~Exp. No. 389, pp. 123--165. \MR{498581}

\bibitem[DLLZ23]{DLLZ23}
Hansheng Diao, Kai-Wen Lan, Ruochuan Liu, and Xinwen Zhu, \emph{Logarithmic
  {R}iemann-{H}ilbert correspondences for rigid varieties}, J. Amer. Math. Soc.
  \textbf{36} (2023), no.~2, 483--562. \MR{4536903}

\bibitem[DMOS82]{DMOS82}
Pierre Deligne, James~S. Milne, Arthur Ogus, and Kuang-yen Shih, \emph{Hodge
  cycles, motives, and {S}himura varieties}, Lecture Notes in Mathematics, vol.
  900, Springer-Verlag, Berlin-New York, 1982. \MR{654325}

\bibitem[DvHKZ]{DvHKZ25p}
Patrick Daniels, Pol van Hoften, Dongryul Kim, and Mingjia Zhang, \emph{Igusa
  stacks and the cohomology of {S}himura varieties {II}}, upcoming work.

\bibitem[DvHKZ24]{DvHKZ24p}
\bysame, \emph{Igusa stacks and the cohomology of {S}himura varieties}, 2024.

\bibitem[Far16]{Far16p}
Laurent Fargues, \emph{Geometrization of the local {L}anglands correspondence:
  an overview}, 2016.

\bibitem[Far20]{Far20}
Laurent Fargues, \emph{{$G$}-torseurs en th\'eorie de {H}odge {$p$}-adique},
  Compos. Math. \textbf{156} (2020), no.~10, 2076--2110. \MR{4179595}

\bibitem[Fon82]{Fon82}
Jean-Marc Fontaine, \emph{Sur certains types de repr\'esentations {$p$}-adiques
  du groupe de {G}alois d'un corps local;\ construction d'un anneau de
  {B}arsotti-{T}ate}, Ann. of Math. (2) \textbf{115} (1982), no.~3, 529--577.
  \MR{657238}

\bibitem[FS24]{FS21p}
Laurent Fargues and Peter Scholze, \emph{Geometrization of the local
  {L}anglands correspondence}, 2024.

\bibitem[GL22]{GL22p}
Ian Gleason and João Lourenço, \emph{On the connectedness of $p$-adic period
  domains}, 2022.

\bibitem[GLX23]{GLX23p}
Ian Gleason, Dong~Gyu Lim, and Yujie Xu, \emph{The connected components of
  affine {D}eligne--{L}usztig varieties}, 2023.

\bibitem[Han16]{Han16p}
David Hansen, \emph{Period morphisms and variations of $p$-adic {H}odge
  structure}, 2016.

\bibitem[HJ23]{HJ23}
David Hansen and Christian Johansson, \emph{Perfectoid {S}himura varieties and
  the {C}alegari--{E}merton conjectures}, J. Lond. Math. Soc. (2) \textbf{108}
  (2023), no.~5, 1954--2000. \MR{4668521}

\bibitem[HK24]{HK25p}
David Hansen and Kiran~S. Kedlaya, \emph{Sheafiness creteria for {H}uber
  rings}, 2024.

\bibitem[HL24]{HL24p}
Linus Hamann and Si~Ying Lee, \emph{Torsion vanishing for some {S}himura
  varieties}, 2024.

\bibitem[Hub93]{Hub93}
Roland Huber, \emph{Continuous valuations}, Math. Z. \textbf{212} (1993),
  no.~3, 455--477. \MR{1207303}

\bibitem[Hub94]{Hub94}
\bysame, \emph{A generalization of formal schemes and rigid analytic
  varieties}, Math. Z. \textbf{217} (1994), no.~4, 513--551. \MR{1306024}

\bibitem[Hub96]{Hub96}
\bysame, \emph{{\'E}tale cohomology of rigid analytic varieties and adic
  spaces}, Aspects of Mathematics, vol. E30, Friedr. Vieweg \& Sohn,
  Braunschweig, 1996. \MR{1734903}

\bibitem[IKY24]{IKY23p}
Naoki Imai, Hiroki Kato, and Alex Youcis, \emph{The prismatic realization
  functor for {S}himura varieties of abelian type}, 2024.

\bibitem[Jan03]{Jan03}
Jens~Carsten Jantzen, \emph{Representations of algebraic groups}, second ed.,
  Mathematical Surveys and Monographs, vol. 107, American Mathematical Society,
  Providence, RI, 2003. \MR{2015057}

\bibitem[Ked05]{Ked05}
Kiran~S. Kedlaya, \emph{Slope filtrations revisited}, Doc. Math. \textbf{10}
  (2005), 447--525. \MR{2184462}

\bibitem[Ked16]{Ked16}
\bysame, \emph{Noetherian properties of {F}argues-{F}ontaine curves}, Int.
  Math. Res. Not. IMRN (2016), no.~8, 2544--2567. \MR{3519123}

\bibitem[KL15]{KL15}
Kiran~S. Kedlaya and Ruochuan Liu, \emph{Relative {$p$}-adic {H}odge theory:
  foundations}, Ast\'erisque (2015), no.~371, 239. \MR{3379653}

\bibitem[KL19]{KL19p}
Kiran~S. Kedlaya and Ruochuan Liu, \emph{Relative {$p$}-adic {H}odge theory,
  {II}: Imperfect period rings}, 2019.

\bibitem[Kos21]{Kos21p}
Teruhisa Koshikawa, \emph{{E}ichler--{S}himura relations for local {S}himura
  varieties}, 2021.

\bibitem[Kot85]{Kot85}
Robert~E. Kottwitz, \emph{Isocrystals with additional structure}, Compositio
  Math. \textbf{56} (1985), no.~2, 201--220. \MR{809866}

\bibitem[KSZ21]{KSZ21p}
Mark Kisin, Sug~Woo Shin, and Yihang Zhu, \emph{The stable trace formula for
  {S}himura varieties of abelian type}, 2021.

\bibitem[LS18]{LS18}
Kai-Wen Lan and Beno\^it Stroh, \emph{Nearby cycles of automorphic \'etale
  sheaves}, Compos. Math. \textbf{154} (2018), no.~1, 80--119. \MR{3719245}

\bibitem[Lur09]{Lur09}
Jacob Lurie, \emph{Higher topos theory}, Annals of Mathematics Studies, vol.
  170, Princeton University Press, Princeton, NJ, 2009. \MR{2522659}

\bibitem[LZ17]{LZ17}
Ruochuan Liu and Xinwen Zhu, \emph{Rigidity and a {R}iemann--{H}ilbert
  correspondence for {$p$}-adic local systems}, Invent. Math. \textbf{207}
  (2017), no.~1, 291--343. \MR{3592758}

\bibitem[Mil05]{Mil05}
J.~S. Milne, \emph{Introduction to {S}himura varieties}, Harmonic analysis, the
  trace formula, and {S}himura varieties, Clay Math. Proc., vol.~4, Amer. Math.
  Soc., Providence, RI, 2005, pp.~265--378. \MR{2192012}

\bibitem[Pin90]{Pin90}
Richard Pink, \emph{Arithmetical compactification of mixed {S}himura
  varieties}, Bonner Mathematische Schriften [Bonn Mathematical Publications],
  vol. 209, Universit\"at Bonn, Mathematisches Institut, Bonn, 1990,
  Dissertation, Rheinische Friedrich-Wilhelms-Universit\"at Bonn, Bonn, 1989.
  \MR{1128753}

\bibitem[PR24]{PR24}
Georgios Pappas and Michael Rapoport, \emph{{$p$}-adic shtukas and the theory
  of global and local {S}himura varieties}, Camb. J. Math. \textbf{12} (2024),
  no.~1, 1--164. \MR{4701491}

\bibitem[Rap18]{Rap18}
Michael Rapoport, \emph{Accessible and weakly accessible period domains}, Ann.
  Sci. \'{E}c. Norm. Sup\'{e}r. (4) \textbf{51} (2018), no.~4, 811--863,
  appendix to \textit{On the $p$-adic cohomology of the Lubin--Tate tower} by
  Peter Scholze. \MR{3861564}

\bibitem[Sch13]{Sch13}
Peter Scholze, \emph{{$p$}-adic {H}odge theory for rigid-analytic varieties},
  Forum Math. Pi \textbf{1} (2013), e1, 77. \MR{3090230}

\bibitem[Sch15]{Sch15}
\bysame, \emph{On torsion in the cohomology of locally symmetric varieties},
  Ann. of Math. (2) \textbf{182} (2015), no.~3, 945--1066. \MR{3418533}

\bibitem[Sch16]{Sch13c}
\bysame, \emph{{$p$}-adic {H}odge theory for rigid-analytic
  varieties---corrigendum [{MR}3090230]}, Forum Math. Pi \textbf{4} (2016), e6,
  4. \MR{3535697}

\bibitem[Sch22]{Sch17p}
\bysame, \emph{{\'E}tale cohomology of diamonds}, 2022.

\bibitem[SGA72]{SGA4I}
\emph{Th\'eorie des topos et cohomologie \'etale des sch\'emas. {T}ome 1:
  {T}h\'eorie des topos}, Lecture Notes in Mathematics, vol. Vol.~269,
  Springer-Verlag, Berlin-New York, 1972, S\'eminaire de G\'eom\'etrie
  Alg\'ebrique du Bois-Marie 1963--1964 (SGA 4), Dirig\'e{} par M.~Artin,
  A.~Grothendieck, et J.~L.~Verdier. Avec la collaboration de N.~Bourbaki,
  P.~Deligne et B.~Saint-Donat. \MR{354652}

\bibitem[She17]{She17}
Xu~Shen, \emph{Perfectoid {S}himura varieties of abelian type}, Int. Math. Res.
  Not. IMRN (2017), no.~21, 6599--6653. \MR{3719474}

\bibitem[Shi22]{Shi22}
Koji Shimizu, \emph{A {$p$}-adic monodromy theorem for de {R}ham local
  systems}, Compos. Math. \textbf{158} (2022), no.~12, 2157--2205. \MR{4522684}

\bibitem[{Sta}]{Stacks}
{Stacks Project Authors}, \emph{Stacks project}.

\bibitem[SW20]{SW20}
Peter Scholze and Jared Weinstein, \emph{Berkeley lectures on {$p$}-adic
  geometry}, Annals of Mathematics Studies, vol. 207, Princeton University
  Press, Princeton, NJ, 2020. \MR{4446467}

\bibitem[vdH24]{vdH24p}
Thibaud van~den Hove, \emph{The stack of spherical {L}anglands parameters},
  2024.

\bibitem[Vie24]{Vie24}
Eva Viehmann, \emph{On {N}ewton strata in the {$B_{\rm dR}^+$}-{G}rassmannian},
  Duke Math. J. \textbf{173} (2024), no.~1, 177--225. \MR{4728690}

\bibitem[Zha23]{Zha23p}
Mingjia Zhang, \emph{A {PEL}-type {I}gusa stack and the $p$-adic geometry of
  {S}himura varieties}, 2023.

\end{thebibliography}

\end{document}